\documentclass[11pt]{article}
\usepackage{amsmath,amsthm,amsfonts,amssymb,graphicx,
color}
\usepackage{enumerate}
\usepackage{pictex,dcpic}

\usepackage{titlesec}
\titleformat{\subsection}[runin]{}{}{}{}[]

\newtheorem{thm}{Theorem}[section]
\newtheorem{lemma}[thm]{Lemma}
\newtheorem{cor}[thm]{Corollary}
\newtheorem{prop}[thm]{Proposition}
\newtheorem*{main}{Main Theorem}{}
{}

\theoremstyle{remark}

\newtheorem{definition}[thm]{Definition}
\newtheorem{remark}[thm]{Remark}

\newtheorem*{definition*}{Definition}
\newtheorem*{remark*}{Remark}

\newtheorem{convention}[thm]{Convention}
\newtheorem{notation}[thm]{Notation}

\def\R{{\mathbb R}}

\def\FF{{\mathbb F}}
\def\ff{\dot} 

\def\o{\omega}

\def\0{0} 
\def\L{L}

\def\Z{{\mathbb Z}}

\def\F{{\mathcal F}}
\def\X{{CV}} 
\def\hX{{cv}}	

\def\rank{\operatorname{rank}}
\def\ill{illegality}
\def\Ill{Illegality}

\def\m{m}   
\def\hm{\check m}   
\def\k{k}   

\def\bG{G} 
\def\hG{\hat G} 
\def\G{\bar G} 
\def\gG{{G}} 

\def\C{{\cal C}}
\def\E{{\mathcal E}}
\def\V{{\mathcal V}}
\def\T{\bar{\mathcal T}} 
\def\W{WH} 
\def\SW{SW} 
\def\CV{\V_c} 

\DeclareMathOperator*{\Lt}{Left}
\DeclareMathOperator*{\Rt}{Right}
\DeclareMathOperator*{\lt}{left}
\DeclareMathOperator*{\rt}{right}

\title{Hyperbolicity of the complex of free factors}

\author{Mladen Bestvina and Mark Feighn\thanks{Both
    authors gratefully acknowledge the support by the National
    Science Foundation.}
}

\date{January 21, 2014}

\begin{document}

\maketitle

\begin{abstract}
  We develop the geometry of folding paths in Outer space and, as an
  application, prove that the complex of free factors of a free group
  of finite rank is hyperbolic.
\end{abstract}
\tableofcontents
\section{Introduction}

The {\it complex of free factors} of a free group $\FF$ of rank $n$ is
the simplicial complex $\F$ whose vertices are conjugacy classes of
proper free factors $A$ of $\FF$, and simplices are determined by
chains $A_1<A_2<\cdots<A_k$. The outer automorphism group $Out(\FF)$
acts naturally on $\F$, which can be thought of as an analog of the
Bruhat-Tits building associated with $GL_n(\Z)$. This complex was
introduced by Hatcher and Vogtmann in \cite{HV} where it is shown that it
has the homotopy type of the wedge of spheres of dimension $n-2$. They
defined this complex in terms of sphere systems in $\#_1^n S^1\times
S^2$ and used variants in their work on homological stability
\cite{HV2,HV3,HV4}.

There is a very useful analogy between $\F$ and the curve complex
$\mathcal C$ associated with a compact surface (with punctures)
$\Sigma$. The vertices of $\mathcal C$ are isotopy classes of
essential simple closed curves in $\Sigma$, and simplices are
determined by pairwise disjoint curves. The curve complex was
introduced by Harvey \cite{harvey} and was classically used by Harer
in his work on duality and homological stability of mapping class
groups \cite{harer1, harer2}. The key result here is that the curve
complex is homotopy equivalent to the wedge of spheres.

More recently, the curve complex has been used in the study of the
geometry of mapping class groups and of ends of hyperbolic
3-manifolds. The fundamental result on which this work is based is the
theorem of Masur and Minsky \cite{MM} that the curve complex is
hyperbolic. In the low complexity cases when $\mathcal C$ is
a discrete set one modifies the definition of $\C$ by adding an edge
when the two curves intersect minimally. In the same way, we modify
the definition of $\F$ when the rank is $n=2$ by adding an edge when
the two free factors (necessarily of rank 1) are determined by a basis
of $\FF$, i.e.\ whenever $\FF=\langle a,b\rangle$, then $\langle
a\rangle$ and $\langle b\rangle$ span an edge. In this way $\F$
becomes the standard Farey graph. The main result in this paper is:

\begin{main} The complex $\F$ of free factors is hyperbolic.
\end{main}

The statement simply means that when the 1-skeleton of $\F$ is
equipped with the path metric in which every edge has length 1, the
resulting graph is hyperbolic.

There are variants of the definition that give rise to quasi-isometric
complexes. For example, one can take the complex of partial bases,
where vertices are conjugacy classes of elements that are part of a
basis, and simplices correspond to compatibility, i.e.\ subsets
of a basis. The $Aut(\FF)$-version of this complex was used in
\cite{day-putman} to study the Torelli group. As another example, $\F$
is quasi-isometric to the nerve of the cover $\{U(A)\}$ of the thin
part of Outer space, where for a conjugacy class of proper free
factors $A$, the set $U(A)$ consists of those marked graphs whose
$A$-cover has core of volume $<\epsilon$, for a fixed small
$\epsilon>0$. 

Our proof is very much inspired by the Masur-Minsky argument, which
uses Teichm\"uller theory. Bowditch \cite{bo:hyp} gave a somewhat
simpler argument. In the remainder of the introduction, we give an
outline of the hyperbolicity of the curve complex which follows
\cite{MM} and \cite{bo:hyp}, and where we take a certain poetic
license. 

The proof starts by defining a coarse projection $\pi:\mathcal
T\to\mathcal C$ from Teichm\"uller space. To a marked Riemann surface
$X$ one associates a curve with smallest extremal length. To see that
this is well defined one must argue that short curves intersect a
bounded number of times (in fact, at most once), and then one uses the
inequality 
$$d_{\mathcal C}(\alpha,\beta)\leq i(\alpha,\beta)+1$$ where $i$
denotes the intersection number. Interestingly, the entire argument
uses only this inequality to estimate distances in $\mathcal C$ (and
only for bounded intersection numbers).

Teichm\"uller space carries the Teichm\"uller metric, and any two
points are joined by a unique Teichm\"uller geodesic. If $t\mapsto
X_t$ is a Teichm\"uller geodesic, consider the (coarse) path
$\pi(X_t)$ in $\mathcal C$. One observes:
\begin{enumerate}[(i)]
\item The collection of paths $\pi(X_t)$ is (coarsely) transitive,
  i.e.\ for any two curves $\alpha,\beta$ there is a path $\pi(X_t)$
  that connects $\alpha$ to $\beta$ (to within a bounded distance).
\end{enumerate}

Next, for any Teichm\"uller geodesic $\{X_t\}$ one defines a projection
$\mathcal C\to \{X_t\}$. Essentially, for a curve
$\alpha$ the projection assigns the Riemann surface $X_t$ on the path
in which $\alpha$ has the smallest length. The key lemma is the
following (see
\cite[Lemma 5.8]{MM}), proved using the intersection number estimate
above:

\begin{enumerate}[(ii)]
\item If $\alpha,\beta$ are adjacent in $\mathcal C$ and
  $X_\alpha,X_\beta$ are their projections to $X_t$, then
  $\pi(X_\alpha)$ and $\pi(X_\beta)$ are at uniformly bounded distance
  in $\mathcal C$.
\end{enumerate}

Consequently, one has a (coarse) Lipschitz retraction $\mathcal
C\to\pi(X_t)$ for every Teichm\"uller geodesic $X_t$. It quickly
follows that the paths $\pi(X_t)$ are reparametrized quasi-geodesics
(this means that they could spend a long time in a bounded set, but
after removing the corresponding subintervals the resulting coarse
path is a quasi-geodesic with uniform constants, after possibly
reparametrizing).

The final step is:

\begin{enumerate}[(iii)]
\item Triangles formed by three projected Teichm\"uller geodesics are
  uniformly thin. 
\end{enumerate}

Hyperbolicity of $\mathcal C$ now follows by an argument involving an
isoperimetric inequality (see \cite[Proposition 3.1]{bo:hyp} and our
Proposition \ref{hyperbolicity}). 

Our argument follows the same outline. In place of Teichm\"uller metric
and Teichm\"uller geodesics we use the Lipschitz metric on Outer space
and folding paths. There are technical complications arising from the
non-symmetry of the Lipschitz metric and the non-uniqueness of folding
paths between a pair of points in Outer space. Our projection from
$\F$ to a folding path comes in two flavors, left and right, and we
have to work to show that the two are at bounded distance from each other
when projected to $\F$. Similarly, we have to prove directly that
projections of folding paths fellow travel, even when the two have
opposite orientations. The role of simple closed curves is played by
{\it simple conjugacy classes} in $\FF$, i.e.\ nontrivial conjugacy
classes contained in some proper free factor.

The first two hints that Outer space has some hyperbolic features was
provided by Yael Algom-Kfir's thesis \cite{yael} and by
\cite{BF2}. Algom-Kfir showed that axes
of fully irreducible automorphisms are strongly contracting. In the
course of our proof we will generalize this result (see Proposition
\ref{contraction}) which states that all folding paths are contracting
provided their projections to $\F$ travel with definite speed. 
In \cite{BF2}
a certain non-canonical hyperbolic $Out(\FF)$ complex was constructed.
It is also known that fully irreducible automorphisms
act on $\F$ with positive translation length \cite{ilya-lustig,BF2}. 

Below is a partial dictionary between Teichm\"uller
space and Outer space relevant to this work.

\vskip .5cm
\begin{tabular}{|r|l|}
\hline
Surfaces & Free groups\\
\hline\hline
curve complex $\mathcal C$&complex of free factors $\F$\\
\hline
simple closed curve	&free factor, \\
&or a simple conjugacy class \\
\hline
intersection number &number of times\\
$i(\alpha,\beta)$	
 &a loop in a graph crosses an edge\\
\hline
Teichm\"uller space	&Outer space\\
\hline
Teichm\"uller distance	&Lipschitz distance\\
\hline
Teichm\"uller geodesic 	&folding path\\
\hline
Teichm\"uller map	&optimal map\\
\hline
quadratic differential	&train track structure on the tension graph\\
\hline
horizontal curve	&legal loop\\
\hline
vertical curve	&illegal loop\\
\hline
\end{tabular}
\vskip 1cm

The paper is organized as follows. In Section~\ref{s:review} we review
the basic notions about Outer space, including the Lipschitz metric,
train tracks, and folding paths. Section~\ref{s:whitehead} proves the
analog of the inequality $d_{\mathcal C}(\alpha,\beta)\leq
i(\alpha,\beta)+1$, using the Whitehead algorithm.
Sections~\ref{s:folding} and \ref{s:gafa} contain some additional
material on folding paths, including the formula for the derivative of
a length function, as well as the key fact that a simple class which
is largely illegal must lose a fraction of its illegal turns after a
definite progress in $\F$. In Section~\ref{s:projection} we define the
(left and right) projections of a free factor to a folding path and
establish that images in $\F$ of folding paths are reparametrized
quasi-geodesics. The key technical lemma in that section is
Proposition~\ref{left-right}(legal and illegal) establishing that the
two projections are at bounded distance when measured in $\F$. We end
this section with a very useful method of estimating where the
projections lie in Lemma~\ref{criterion}. In
Section~\ref{s:hyperbolicity} we recall the argument that for
hyperbolicity it suffices to establish the Thin Triangles condition,
and we also derive the contraction property of folding paths, measured
in $\F$. In Section~\ref{s:FT} we prove the Fellow Travelers property
(which of course follows from the Thin Triangles property), in both
parallel and anti-parallel setting.  Finally, in Section \ref{s:thin}
we establish the Thin Triangles property.

The proofs of the three main technical statements in the paper, namely
Proposition~\ref{gafa}(surviving illegal turns),
Proposition~\ref{left-right}(legal and illegal), and
Proposition~\ref{general}(closing up to a simple class), should be
omitted on the first reading.

\vskip 0.5cm
\noindent {\bf Acknowledgments.} We thank the American Institute of
Mathematics and the organizers and participants of the Workshop on
Outer space in October 2010 for a fruitful exchange of ideas. We
particularly thank Michael Handel for telling us a proof of Lemma
\ref{michael2}.  We thank Saul Schleimer for an inspiring
conversation and Yael Algom-Kfir for her comments on an earlier
version of this paper. We heartily thank the anonymous referee
for a very careful reading and many suggestions we feel greatly
improved the exposition. We also thank the referee for pointing out an
error in one of two arguments we gave to finish the proof of
Proposition~\ref{gafa}.

Since the first version of this paper appeared on the arXiv, there
have been some very interesting developments. 
Lee Mosher and Michael Handel
\cite{hm:hyperbolicity} have proved that the free splitting complex
$\mathcal S$ for $\FF$ is hyperbolic. Ilya Kapovich and Kasra Rafi
\cite{kr:hyperbolicity} have shown that the hyperbolicity of the
complex of free factors follows from the hyperbolicity of the free
splitting complex. Arnaud Hilion and Camille Horbez
\cite{hh:hyperbolicity} have proved that $\mathcal S$ is hyperbolic
from the point of view of the sphere complex. Brian Mann
\cite{bm:cyclic} showed that another $Out(\FF)$-complex called the
{\it cyclic splitting complex} is hyperbolic.

\section{Review}\label{s:review}
In this section we review some definitions and collect standard facts
about Outer space, Lipschitz metric, train tracks, and folding paths.
\vskip 0.5cm
\noindent
{\bf 
Outer space.}
A {\it graph} is a cell complex $\gG$ of dimension $\leq 1$. The {\it
  rose} $R_n$ of rank $n$ is the graph with one 0-cell (vertex) and
$n$ 1-cells (edges). A vertex of $\gG$ not of valence 2 is {\it topological}.  The closure in $\gG$ of a component
of the complement of the set of topological vertices is a {\it
  topological edge}. In particular, a topological edge may be a
circle. A {\it marking} of
a graph $\gG$ is a homotopy equivalence $g:R_n\to \gG$. A {\it metric}
on $\gG$ is a function $\ell$ that to each edge $e$ assigns a positive
number $\ell(e)$. We often view the graph $\gG$ as the path metric
space in which each edge $e$ has length $\ell(e)$.  The
Unprojectivized Outer space $\hX$ is the space of equivalence classes
of triples $(\gG,g,\ell)$ where $\gG$ is a finite graph with no
vertices of valence $\leq 2$, $g$ is a marking of $\gG$, and $\ell$ is
a metric on $\gG$. Two triples $(\gG,g,\ell)$ and $(\gG',g',\ell')$
are equivalent if there is a homeomorphism $h:\gG\to\gG'$ that
preserves edge-lengths and commutes with the markings up to homotopy.
Outer space $\X$ is the space of projective classes of such
(equivalence classes of) triples, i.e.\ modulo scaling the metric.
Equivalently, $\X$ is the space of triples as above where the metric
is normalized so that the volume $vol(\gG):=\sum\ell(e)=1$. Assigning
length 0 to an edge is interpreted as a metric on the graph with that
edge collapsed, and in this way $\X$ becomes a complex of simplices
with missing faces (the missing faces correspond to collapsing
nontrivial loops), which then induces the simplicial topology on $\X$.
Outer space was introduced by Culler and Vogtmann \cite{CV}, who
showed that $\X$ is contractible. We will usually suppress markings
and metrics and talk about $\gG\in\X$.  It is sometimes convenient to
pass to the universal cover and regard $\gG\in\hX$ as an action of
$\FF$ on the tree $\tilde\bG$.

We find the following notation useful. If $z$ is a nontrivial
conjugacy class, it can be viewed as a loop in the rose $R_n$, and via
the marking can be transported to a unique immersed loop in any marked
graph $\gG$. This loop will be denoted by $z|\gG$. The length of
this loop, i.e.\ the sum of the lengths of edges crossed by the loop,
counting multiplicities, is denoted by $\ell(z|\gG)$. Note that if
$z$ is not simple, then $\ell(z|\gG)\geq 2~vol(\gG)$ and in fact $z|\gG$ must
cross every edge at least twice.
\vskip 0.5cm
\noindent
{\bf Morphisms between trees and train tracks.}\label{morphisms0}
Recall that a {\it morphism} between two $\R$-trees $S,T$ is a map
$\tilde\phi:S\to T$ such that every segment $[x,y]\subset S$ can be
partitioned into finitely many subintervals on which $\tilde\phi$ is
an isometric embedding. For simplicity, in this paper we work only
with simplicial metric trees. A {\it direction} at $x\in S$ is a germ
of nondegenerate segments $[x,y]$ with $y\neq x$. The set $D_x$ of
directions at $x$ can be thought of as the unit tangent space; a
morphism $\tilde\phi:S\to T$ determines a map $D\tilde\phi_x:D_x\to
D_{\tilde\phi(x)}$, thought of as the derivative.  A {\it turn} at $x$
is an unordered pair of distinct directions at $x$. A turn $\{d,d'\}$
at $x$ is {\it illegal} (with respect to $\tilde\phi$) if
$D\tilde\phi_x(d)=D\tilde\phi_x(d')$. Otherwise the turn is {\it
  legal}. There is an equivalence relation on $D_x$ where $d\sim d'$
if and only if $d=d'$ or $\{d,d'\}$ is illegal. The equivalence
classes are {\it gates}. The collection of equivalence classes at each
$x$ is called the {\it illegal turn structure on $S$ induced by
  $\tilde\phi$.} If at each $x\in S$ there are at least two gates, the
illegal turn structure is called a {\it train track structure on $S$}.
This is equivalent to the requirement that $\tilde\phi$ embeds each
edge of $S$ and has at least two gates at every vertex. A path in $S$
is {\it legal} if it makes only legal turns.

If $S\overset{\tilde\phi}{\to}T{\to}U$ is a composition of morphisms
then there are two illegal turn structures of interest on $S$: one
induced by $\tilde\phi$ and the other induced by the composition. In
this situation, we will sometimes refer to the second of these as the
{\it pullback illegal turn structure on $S$ via $\tilde\phi$}. Note
that an illegal turn in the first structure is also illegal in the
second. In particular, if the second structure is a train track then
so is the first.

If $S$ and $T$ are equipped with abstract train track structures
(equivalence relation on $D_x$ for every vertex $x$ with at least two
gates), we say that a morphism $\tilde\phi:S\to T$ is a {\it train track
  map} if on each edge $\tilde\phi$ is an embedding and legal turns are sent
to legal turns\footnote{The standard usage of the term {\it train
    track map} is to self-maps of a graph, but this natural extension
  of the terminology should not cause any confusion.}. In particular,
legal paths map to legal paths.

We also extend this terminology to maps between graphs. If
$\phi:\Delta\to\Sigma$ is a map between connected metric graphs such
that the lift $\tilde\phi:\tilde\Delta\to\tilde\Sigma$ is a morphism
of trees, then we also say that $\phi$ is a morphism and we can define
the notion of legal and illegal turns on $\tilde\Delta$, which
descends to $\Delta$. If there are at least two gates at each point,
we have a train track structure on $\Delta$. If $\Delta$ and $\Sigma$
are equipped with abstract train track structures, the map $\phi$ is a
train track map if it sends edges to legal paths and legal turns to
legal turns. When the graphs $\Delta$ or $\Sigma$ are not connected we
work with components separately.  

\vskip 0.5cm
\noindent
{\bf Lipschitz metric and optimal maps.}  Let $\gG$ and $\gG'$ be two
points in $\hX$. The homotopy class of maps $h:\gG\to\gG'$ such that
$hg\simeq g'$ (with $g,g'$ markings for $\gG,\gG'$) is called the {\it
  difference of markings}. If $\gG$ and $\gG'$ are in $\X$, i.e.\ if
they have volume 1, the {\it Lipschitz distance} $d_\X(\gG,\gG')$
between $\gG$ and $\gG'$ is the log of the minimal Lipschitz constant
over all difference of markings maps $\gG\to\gG'$. The Lipschitz
distance is an asymmetric metric on $\X$ inducing the correct
topology. For more information,
see \cite{francaviglia-martino, ab, b:bers}. The basic fact, that
plays the role of Teichm\"uller's theorem in Teichm\"uller theory, is
the following statement, due to Tad White. For a proof see
\cite{francaviglia-martino} or \cite{b:bers}.

\begin{prop}\label{greengraph}
Let $\gG,\gG'\in\hX$. There is a difference of markings map
$\phi:\gG\to\gG'$ with the following properties.
\begin{itemize}

\item
$\phi$ sends each edge of $\gG$ to an immersed path (or a point) with
  constant speed (called the {\it slope} of $\phi$ on that edge).

\item The union of all edges of $\gG$ on which $\phi$ has the maximal
  slope $\lambda=\lambda(\phi)$, is a subgraph of $\gG$ with no
  vertices of valence 1. This subgraph is called the {\em tension
    graph}, denoted $\Delta=\Delta(\phi)$.

\item 
$\phi$ induces a train track structure on $\Delta$.
\end{itemize}
\end{prop}

The last bullet says that the map $\lambda\Delta\to \gG'$ induced by
$\phi$ is a morphism immersing edges and each vertex has at least two
gates. Note that the Lipschitz constant of $\phi$ is $\lambda$ and so,
in the case that $\gG, \gG'\in\X$, we have that $d_\X(\gG,\gG')\leq
\log\lambda$. In fact equality holds since the presence of at least
two gates at each vertex of $\Delta$ guarantees that it contains legal
loops, whose length gets stretched by precisely $\lambda$ so there can
be no better map homotopic to $\phi$. We call any $\phi$ satisfying
Proposition~\ref{greengraph} an {\it optimal map}. A morphism that
induces a train track structure is an example of an optimal map.
Unfortunately, unlike Teichm\"uller maps, optimal maps are not
uniquely determined by $\gG,\gG'$.

A legal loop can be constructed in $\Delta$ by starting with an edge
and extending it inductively to longer legal edge paths until some
oriented edge is repeated. In fact, this guarantees the existence of a
``short'' legal path. We will say a loop in $\Delta$ is a {\it
  candidate} if it is either embedded, or it forms the figure 8, or it
forms a ``dumbbell''.  Every candidate determines a conjugacy class
that generates a free factor of rank 1. Thus graph $\Delta$ admits a
legal candidate. See \cite{francaviglia-martino}.
\vskip 0.5cm
\noindent
{\bf Folding at speed 1.}\label{morphisms}\label{p:folding}
Now assume that $S,T$ represent points of Unprojectivized Outer space
$\hX$ (i.e.\ they are universal covers of marked metric graphs), and
$\tilde\phi:S\to T$ is an equivariant morphism. Equivariantly
subdivide $S$ so that $\tilde\phi$ embeds each edge and the inverse
image under $\tilde\phi$ of a vertex of $T$ is a vertex of $S$. Choose
some $\epsilon>0$ smaller than half of the length of any edge in $S$.
Then for $t\in [0,\epsilon]$ define the tree $S_t$ as the quotient of
$S$ by the equivalence relation: $u\sim_t v$ if and only if there is a
vertex $x$ with $d(x,u)=d(x,v)\leq t$ and
$\tilde\phi(u)=\tilde\phi(v)$.  Then $S_t$ represents a point in $\hX$
and $\tilde\phi$ factors as $S\overset{\tilde\phi_{0t}}\to
S_t\overset{\tilde\phi_{t\infty}}\to T$ for some equivariant morphism
$\tilde\phi_{t\infty}$ and the quotient map $\tilde\phi_{0t}:S\to
S_t$, which is also an equivariant morphism. The trees $S_t$, $t\in
[0,\epsilon]$ form a path in $\hX$, and $S_0=S$. We say that $S_t$ is {\it the path obtained from $S$ by folding all illegal turns at speed 1 with
  respect to $\tilde\phi$}.

If $\tilde\phi$ induces only one gate at some vertex $v\in S$, then
$S_t$ will have a valence 1 vertex for $t>0$. In that case we always
pass to the minimal subtree of $S_t$. When $\tilde\phi$ induces a
train track structure on $S$, $S_t$ is automatically minimal (if $S$
is). For simplicity we state Proposition~\ref{slope 1 folding} in the train
track situation only.

\begin{prop}\label{slope 1 folding}
Let $\tilde\phi:S\to T$ be an equivariant morphism between two trees in $\hX$
inducing a train track structure on $S$. There is a (continuous) path $S_t$
in $\hX$,
$t\in [0,\infty)$, and there are equivariant morphisms $\tilde\phi_{st}:S_s\to S_t$ for
  $s\leq t$ so that the following holds:
\begin{enumerate}[(1)]
\item $S_0=S$, $S_t=T$ for $t$ large;
\item $\tilde\phi_{tt}=Id$, $\tilde\phi_{su}=\tilde\phi_{tu}\tilde\phi_{st}$ for $s\leq t\leq
  u$;
\item $\tilde\phi_{0t}=\tilde\phi$ and $\tilde\phi_{tt'}=Id$ for large
  $t<t'$;
\item each $\tilde\phi_{st}$ isometrically embeds edges and induces at least two
  gates at every vertex of $S_s$;
\item for $s<t,t'$ the illegal turns at vertices of $S_s$ with respect
  to $\tilde\phi_{st}$ coincide with those with respect to $\tilde\phi_{st'}$, so
  $S_s$ has a well-defined train track structure; and
\item for every $s<t$ there is $\epsilon>0$ so that $S_{s+\tau}$,
  $\tau\in [0,\epsilon]$ is obtained from $S_s$ by folding all illegal
  turns at speed 1 with respect to $\tilde\phi_{st}$.
\end{enumerate}
Moreover, this path is unique.
\end{prop}

\begin{proof} Uniqueness is clear from the definition of folding at
  speed 1. There can be no last time $s$ so that two paths satisfying
  the above conditions agree (including the maps $\tilde\phi_{tt'}$) up to
  $S_s$ but no further, by item (6).

There are three methods to establish existence, and they will be only
sketched. 

{\bf \ref{slope 1 folding}.A. Stallings' Method.} This works when $S$
and $T$ can be subdivided so that $\tilde\phi$ is simplicial and all
edge lengths are rational (or fixed multiples of rational numbers). In
our applications we can arrange that this assumption holds. Then we
may subdivide further so that all edge lengths are equal. The path
$S_t$ is then obtained exactly as in the Stallings' beautiful paper
\cite{stallings}, by inductively identifying any pair of edges with a
common vertex that map to the same edge in $T$. This operation of {\it
  elementary folding} can be performed continuously to yield a
1-parameter family of trees, i.e.\ a path, between the original tree
and the folded tree. Putting these paths together gives the path
$S_t$.

{\bf \ref{slope 1 folding}.B. Via the vertical thickening of the graph
  of $\tilde\phi$.} This method is due to Skora
\cite{skora:deformations}, who built on the ideas of Steiner. Skora's
preprint was never published; the interested reader may find the
details in \cite{matt} and \cite{2topologies}. Consider the graph of
$\tilde\phi$ as a subset of $S\times T$ and define the ``vertical
$t$-thickening'' of it as
$$W_t=\{(x,y)\in S\times T\mid d(\tilde\phi(x),y)\leq t\}$$
Next, consider the decomposition $\mathcal D_t$ of $W_t$ into the path
components of the sets $W_t\cap S\times\{y\}$, $y\in T$. Let
$S_t=W_t/\mathcal D_t$ be the decomposition space with the metric
defined as follows. A path in $W_t$ is {\it linear} if its projection
to both $S$ and $T$ has constant speed (possibly speed 0). A piecewise
linear path $\gamma$ in $W_t$ is {\it taut} if the preimages
$\gamma^{-1}(\ell)$ of leaves in $\mathcal D_t$ are connected. Then
define the distance in $S_t$ as the length of the projection to $T$ of
any piecewise linear taut path connecting the corresponding leaves in
$W_t$. In this way $S_t$ becomes a metric tree. The morphisms
$\tilde\phi_{st}$ are induced by inclusion $W_s\hookrightarrow W_t$.

{\bf \ref{slope 1 folding}.C. Via integrating the speed 1 folding
  direction.} Starting at $S$ consider the path $S_t, t\in
[0,\epsilon],$ obtained by folding all illegal turns at speed 1. Now
extend this path by folding all illegal turns of $S_\epsilon$ at speed
1. Continue in this way inductively, and show that either $T$ is
reached in finitely many steps, or there is a well defined limiting
tree, from which folding can proceed. This is the approach taken in
\cite{francaviglia-martino}, to which the reader is referred for
further discussion.

One possible approach is as follows. Say $S_t$ is defined for $t\in
[0,t_0)$ with $S_0=S$.
To define the limiting tree $S_{t_0}$, note that for each conjugacy class, the
length along the path is nonincreasing and thus converges. The
limiting length function defines a tree, representing a point in
compactified Outer space. The lengths of conjugacy classes
are bounded below by their values in $T$, so the limiting tree $S_{t_0}$ is
free simplicial and thus represents a point in Outer space. We may
view the tree $S_{t_0}$ as the equivariant Gromov-Hausdorff limit of the path
$S_t$. The maps $S\to S_t$, viewed as subsets of $S\times S_t$ via
their graphs, subconverge to a morphism $S\to S_{t_0}$, and similarly
by a diagonal argument one constructs morphisms $S_t\to S_{t_0}$ that
compose correctly. To show uniqueness of such morphisms, one uses
Gromov-Hausdorff limits and the fact that the only (equivariant) morphism
$S_{t_0}\to S_{t_0}$ is the identity.
\end{proof}

The path $S_t, t\in [0,\infty)$ from Proposition~\ref{slope 1
  folding} is {\it the path induced by
  $\tilde\phi$}. Lemma~\ref{l:folding segments} will be used in the proof of Theorem~\ref{t:gadget}. Its proof is immediate from
\ref{slope 1 folding}.B.

\begin{lemma}\label{l:folding segments}
Let $S_t$, $t\in [0,\infty)$, be the path in $\hX$ induced by $\tilde
  \phi:S\to T$ and $[s_1,s_2]$ be a path in $S_0$. Suppose
  $\tilde\phi(s_1)=\tilde\phi(s_2)$ and set $h$ equal to the {\em
    outradius of $\tilde\phi([s_1,s_2])$ with respect to
    $\tilde\phi(s_1)$}, i.e.\ $$h=\max_{s\in
    [s_1,s_2]}d_T(\tilde\phi(s_1),\tilde\phi(s))$$ Then $s_1$ and
  $s_2$ are identified by time $h$ but not before,
  i.e.\ $\tilde\phi_{0t}(s_1)=\tilde\phi_{0t}(s_2)$ iff $t\ge h$.
\end{lemma}


\vskip 0.5cm
\noindent
{\bf Folding paths.} Suppose $S\to T$ is an equivariant morphism
between trees in $\hX$ that induces a train track structure on $S$. The
induced path $S_t$, $t\in [0,\infty),$ is a {\it folding path in
    $\hX$}. The projection of $S_t$ to $\X$ is a {\it folding path in
    $\X$}. A folding path $\gG_t$ in $\X$ can be parametrized by
  arclength, so that $d_\X(\gG_{t'},\gG_t)=t-t'$ for $t'\leq t$.

\begin{notation}\label{n:folding path}
  From now on, we switch to thinking of points of Outer space as
  finite graphs when discussing folding paths.  We have several
  meanings for {\it folding path}. To avoid confusion, we will use the
  following notation:
\begin{enumerate}[(1)]
\item\label{i:unprojectivized path} $\hG_t$, $t\in [0,\infty),$ denotes
  a folding path in $\hX$, i.e.\ a path in $\hX$ induced by an
  equivariant morphism $S\to T$ giving a train track structure on $S$. If
  $\o\in [0,\infty)$ is minimal such that $\hG_\o=T$, we also refer to
    $\hG_t$, $t\in [0,\o],$ as a folding path. We call $t$ the {\it
      natural parameter}.
\item\label{i:projectivized path} $\G_t$, $t\in [0,\infty),$ (or
  $\G_t$, $t\in [0,\o]$) denotes a folding path in $\X$ obtained by
  projecting a path $\hG_t$ as in (\ref{i:unprojectivized path}).  So,
  $\G_t=\hG_t/vol(\hG_t)$.
\item\label{i:arclength path} $\bG_t$, $t\in [0,L],$ denotes a path as
  in (\ref{i:projectivized path}), but
  \hyphenation{re-param-e-trized}reparametrized in terms of arclength
  with respect to $d_\X$. So, $\bG_{t(s)}=\G_s$ where
  $$t(s)=d_\X(\G_0,\G_s)=\log\bigg(\frac{vol(\hG_0)}{vol(\hG_s)}\bigg)$$ For all
  $t\le t'\in [0,L]$ there is a morphism $e^{t'-t}\bG_t\to \bG_{t'}$.
\end{enumerate}
Unless otherwise noted, a folding path $\bG_t$ without further
adjectives or decorations will mean a folding path as in
(\ref{i:arclength path}).
\end{notation}

\begin{prop}[\cite{francaviglia-martino}]\label{rescaling}
Let $\gG,\Sigma\in\X$. There is a geodesic in $\X$ from $\gG$ to
$\Sigma$ which is the concatenation of two paths, the first is a
(reparametrized) linear path
in a single simplex of $\X$, and the second is a folding path in $\X$ parametrized by arclength.
\end{prop}

\begin{proof}
  Fix an optimal map $\phi:\gG\to\Sigma$ and let
  $\Delta=\Delta(\phi)\subset \gG$ be its tension graph with maximal
  slope $\lambda=\lambda(\phi)$. If $\Delta=\gG$ then the rescaled map
  $\lambda\phi:\lambda\gG\to\Sigma$ is a morphism and satisfies the
  requirement that it induces a train track structure on $\gG$.
  Proposition~\ref{slope 1 folding} gives a folding path $\hG_t$ in
  $\hX$ from $\lambda\gG$ to $\Sigma$ and morphisms
  $\hat\phi_{st}:\hG_s\to\hG_t$. To see that $\hG_t$ projects to a
  geodesic $\G_t$ in $\X$, note that
  $\hat\phi_{su}=\hat\phi_{tu}\hat\phi_{st}$ for $s<t<u$ implies
  $d_\X(\G_s,\G_u)=d_\X(\G_s,\G_t)+d_\X(\G_t,\G_u)$.

  Now suppose $\Delta\neq\gG$. Denote by $e_1,\cdots,e_k$ the
  (topological) edges of $\gG$ outside of $\Delta$. For each tuple
  $x=(x_1,x_2,\cdots,x_k)$ of lengths in the cube $[0,\ell_1]\times
  [0,\ell_2]\times\cdots\times[0,\ell_k]$, where $\ell_i$ is the
  length of $e_i$ in $\gG$, denote by $\mu(x)$ the smallest maximal
  slope among maps $\gG\to\Sigma$ that are homotopic to $\phi$ rel
  $\Delta$, where $\gG$ is given the metric $x$ outside $\Delta$ (so
  $\mu(x)=\infty$ if some loop is assigned length 0). Among all $x$ in
  the cube with $\mu(x)=\lambda$ choose one with the smallest sum of
  the coordinates, say $x_0$. Denote by $\hG'\in\hX$ the graph $\gG$
  with the metric $x_0$ outside of $\Delta$. (Some edges of $\hG'$ may
  get length 0 in which case its projection $\gG'$ to $\X$ is on the
  boundary of the original simplex.)

Let $\phi_0:\hG'\to\Sigma$ be a map homotopic to $\phi$ rel
$\Delta$, linear on edges, and with the maximal slope
$\lambda(\phi_0)=\lambda$. We claim that $\phi_0$ is optimal with
$\Delta(\phi_0)=\hG'$. Indeed, $\Delta(\phi_0)=\hG'$ (otherwise
some edge length can be reduced contradicting the choice of $x_0$) and
$\phi_0$ induces at least two gates at every vertex (otherwise
$\phi_0$ may be perturbed so that the tension graph becomes a proper
subgraph, see the proof of Proposition~\ref{greengraph} e.g.\ in
\cite{b:bers}).

Since we have maps
$\gG\to\hG'$ and $\hG'\to\Sigma$ with slopes 1 and $\lambda$
respectively, we also have
$d_\X(\gG,\Sigma)=d_\X(\gG,\gG')+d_\X(\gG',\Sigma)$.
\end{proof}

\section{Detecting boundedness in the free factor complex ${\F}$}\label{s:whitehead} 
In this section we define a coarse projection $\pi:\X\to\F$ and prove
an analog of the inequality $d_{\mathcal C}(\alpha,\beta)\leq
1+i(\alpha,\beta)$ (see Lemma \ref{bounded crossing}). An immediate
consequence is that $\pi$ is coarsely Lipschitz (see Corollary
\ref{proj X->F continuous}).

Recall that a nontrivial conjugacy class $x$ in $\FF$ is {\it simple} if any
(equivalently some) representative is contained in a proper free
factor. If $x$ is a simple class, denote by $\ff x$ the conjugacy
class of a smallest free factor containing a representative of $x$. A
proper connected subgraph $P$ of a marked graph $\gG$ that contains a
circle defines a vertex $\ff P$ of $\F$.

\begin{lemma}\label{subgraph distance}
Let $\gG$ be a marked graph. If $P,Q\subset\gG$ are proper
connected subgraphs defining free factors $\ff P,\ff Q$ then
$d_{\F}(\ff P,\ff Q)\leq 4$.
\end{lemma}

\begin{proof}
If $\rank(\FF)=2$ then $d_{\F}(\ff P,\ff Q)\leq 1$ (using the
modified definition of $\F$). Now assume $\rank(\FF)\geq 3$.
Enlarge $P,Q$ to connected graphs $P',Q'$ that contain all but one
edge of $\gG$. Thus their intersection
contains a circle $R$, so we have a path in $\F$ given by subgraphs
$P,P',R,Q',Q$ and we see $d_\F(\ff P,\ff Q)\leq 4$. 
\end{proof}

If $\gG$ is a marked graph, define $$\pi(\gG):=\{\ff P\mid P\mbox{ is
  a proper, connected, noncontractible subgraph of } \gG\}$$ The
induced multi-valued, $Out(\FF)$-equivariant map $\X\to\F$ , still
called $\pi$, given by $\bG\mapsto \pi(\bG)$ is coarsely defined in
that, by Lemma~\ref{subgraph distance}, the diameter of each
$\pi(\gG)$ is bounded by 4. We refer to $\pi$ as the {\it coarse
  projection from $\X$ to $\F$}.

\begin{lemma}\label{bounded crossing}
Let $\gG$ be a marked graph and $x$ a simple class in $\FF$.
If $x|\gG$ crosses an edge $e$ $k$
times, then the distance in $\F$ between $\ff x$ and some free factor
represented by a subgraph of $\gG$ is $\leq 6k+9$.
\end{lemma}

\begin{proof}
This proof will use the fact, due to Reiner Martin \cite{reiner}, that
if the Whitehead graph of a simple class $x$ is connected then it has
a cut point. The classical fact, due to Whitehead \cite{wh}, is the
analogous statement for the special case that $x$ is
primitive\footnote{an element of some basis for $\FF$.}. Stallings'
paper \cite{js:whitehead} is a good modern reference for Whitehead's
result and the reader is directed there for the definitions of
Whitehead graph and Whitehead automorphism.

First assume that $e$ is nonseparating. By collapsing a maximal tree
in $\gG$ that does not contain $e$ we may assume that $\gG$ is a
rose. Let $a_1,a_2,\cdots,a_m,c$ be the associated basis with $c$
corresponding to $e$ and set $A=\langle a_1,\cdots,a_m\rangle$. Thus
$c^{\pm 1}$ appears in the cyclic word for $x$ $k$ times. If the
Whitehead graph of $x$ is disconnected, consider a 1-edge blowup
$\tilde \gG$ of $\gG$ so that $x$ realized in $\tilde \gG$ is contained
in a proper subgraph. In this case $d_\F(\ff x,A)\leq 5$ by Lemma
\ref{subgraph distance} (4 for the distance between $A$ and the free
factor determined by the image of $x$, and 1 more to get to $\ff
x$). If the Whitehead graph is connected then it has a cut point
\cite{wh,reiner}. Let $\phi$ be the associated Whitehead automorphism. If the
special letter is some $a_i^{\pm 1}$ then the free factor $A$ is
$\phi$-invariant. If the special letter is $c^{\pm 1}$ then
$d_\F(A,\phi(A))\leq 6$ ($d_\F(A,\langle c\rangle)\leq 3$ and $\langle
c\rangle$ is fixed by $\phi$). But there are at most $k$ automorphisms
of the latter kind in the process of reducing $x$ until its Whitehead
graph is disconnected. Thus $d_\F(\ff x,A)\leq 6k+5$.

Now assume $e$ is separating. By collapsing a maximal tree on each
side of $e$ we may assume that $\gG$ is the disjoint union of two
roses $R_A$ and $R_B$ connected by $e$. Let $a_1,\cdots,a_n$ and
$b_1,\cdots,b_m$ be the bases determined by $R_A$ and $R_B$
respectively. Notice that the assumption about $e$ means that there
are $k$ times when the cyclic word for $x$ changes from the $a_i$'s to
the $b_j$'s or vice versa ($k$ is necessarily even here).  If the
Whitehead graph of $x$ with respect to $a_1,\cdots,a_n,b_1,\cdots,b_m$
is disconnected, we see as above that $d_\F(\ff x,A)\leq 5$ where
$A=\langle a_1,\cdots,a_n\rangle$. Otherwise there is a cut point and
let $\phi$ be the associated Whitehead automorphism. If the special
letter is some $a_i^{\pm 1}$ then $A$ is $\phi$-invariant. Likewise,
$\phi(A)$ is conjugate to $A$ if the special letter is $b_j^{\pm 1}$
and all the $a_i^{\pm 1}$ are on one side of the cut. If they are not
on one side of the cut, then the subgraph spanned by the $a_i^{\pm
  1}$'s is disconnected and we may consider the associated 1-edge
blowup $\tilde R_A$ of $R_A$. Let $\tilde\gG$ be the 1-edge blowup
of $\gG$ obtained by attaching $e\cup R_B$ to $\tilde R_A$ along
either of the two vertices. The blowup edge $e'$ can be crossed by $x$
only if it is immediately followed or preceded by $e$ (but not
both). Thus $x$ crosses $e'$ at most $k$ times. If $e'$ is
nonseparating then by the first paragraph $d_\F(\ff x,\ff P)\leq
6k+5$ for some subgraph $P\subset\tilde\gG$ and so $d_\F(\ff
x,\langle b_1,\cdots,b_m\rangle)\leq 6k+9$. If $e'$ is separating
replace $A$ by a smaller free factor $A'$ and continue.
\end{proof}

We introduce the following notation. Suppose $\gG,\gG'$ are marked
graphs, $A$ is a proper free factor of $\FF$, and $x$ is a simple
class in $\FF$. Then:
\begin{itemize}
\item
$d_\F(\gG,\gG'):=\sup \{d_\F(A,A')\mid A\in\pi(\gG), A'\in\pi(\gG')\}$

\item
$d_\F(A,\gG'):=\sup\{d_\F(A,A')\mid A'\in\pi(\gG')\}$

\item
$d_\F(\gG,x):=d_\F(\gG,\ff x)$

\item
$d_\F(A, x):=d_\F(A, \ff x)$
\end{itemize}
For example, Lemmas~\ref{subgraph distance} and \ref{bounded crossing}
combine to give the following:

\begin{lemma} \label{bounded crossing 2}
Let $x$ be a simple class in $\FF$ and $\gG$ a marked graph so that
  $x|\gG$ crosses some edge $\leq k$ times. 
Then $d_\F(\gG,x)\leq 6k+13$.
\end{lemma}

\begin{remark}\label{r:bounded crossing 2}
  Let $z$ be a conjugacy class in $\FF$ and $\bG\in\X$ a marked
  graph. Since $vol(\bG)=1$, there is an edge of $\bG$ that is crossed
  at most $[\ell(z|\bG)]$ times by $z|\bG$.
\end{remark}

\begin{cor}\label{proj X->F continuous}
If $d_\X(\gG,\gG')\leq \log K$ then $d_\F(\gG,\gG')\leq 12K+32$.
\end{cor}

\begin{proof}
Let $z$ be a candidate that realizes $d_\X(\gG,\gG')$. Thus
$\ell(z|\gG)<2$, $z|\gG$ crosses some edge once and
$\ell(z|\gG')<2K$, so $z|\gG'$ crosses some edge $<2K$
times. Therefore
$$d_\F(\gG,\gG')\leq d_\F(\gG,z)+d_\F(z,\gG')\leq
19+(12K+13)=12K+32$$ 
\end{proof}

In most of the paper we will be concerned with showing that distances
in $\F$ are bounded above. We will use the obvious terminology: In
Corollary~\ref{proj X->F continuous} we showed that the distance in
$\F$ between projections of graphs from $\X$ are bounded above as a
function of the distance in $\X$.

\begin{convention}[Bounded distance] When we say a distance in $\F$ is
  bounded without any variables, we mean bounded above by a universal
  constant that depends only on the rank $n$ of $\FF$.
\end{convention}

Note that Corollary~\ref{proj X->F continuous} says that $\pi:\X\to\F$
is {\it coarsely Lipschitz}\,: If $d_\X(\gG,\gG')\leq N$ with $N$ an
integer, then $d_\F(\gG,\gG')\leq CN$ for a universal $C>0$. Indeed,
choose a geodesic from $\gG$ to $\gG'$ and apply Corollary~\ref{proj
  X->F continuous} $N$ times to pairs of points at distance $\leq 1$.

By the {\it injectivity radius} $injrad(\gG)$ of a metric graph $\gG$
we mean the length of a shortest embedded loop in $\gG$. If $A$ is a
finitely generated subgroup of $\FF$ and $\hG\in\hX$ we denote by
$A|\hG$ the core of the covering space of $\hG$ corresponding to $A$.
Thus there is a canonical immersion $A|\hG\to \hG$.  We adopt the
following convention: Unless otherwise specified, the metric and
illegal turn structures on $A|\hG$ are the ones obtained by pulling
back via $A|\hG\to\hG$.

\begin{cor}\label{lucas}
Suppose $A$ is a proper free factor of $\FF$ and $\bG\in\X$. If $injrad(A|\bG)<k+1$ then $d_\F(A,\bG)\leq 6k+14$.
\end{cor}

\begin{proof}
This follows immediately from Remark~\ref{bounded crossing} and Lemma~\ref{bounded crossing 2}.
\end{proof}

Coarse paths in $\F$ obtained from folding paths by projecting will
play a crucial role.

\begin{cor}\label{c:coarse paths}
The collection of projections to $\F$ of folding paths in $\X$ is a
coarsely transitive family: for any two proper free factors $A,B$ of
$\FF$ there is a folding path $\bG_t$, $t\in [0,L]$, such that
$A\in\pi(\bG_0)$ and $B\in \pi(\bG_L)$. The same is true for the
subcollection consisting of folding paths induced by morphisms that
satisfy the rationality condition from the Stallings method of folding
in \ref{slope 1 folding}.A.
\end{cor}

\begin{proof}
Let $A,B$ be two free factors. Choose $\gG,\Sigma\in\X$ so that some
subgraph of $\gG$ represents $A$ and some subgraph of $\Sigma$
represents $B$, and so that $\gG$ is a rose, and apply
Proposition~\ref{rescaling} to obtain a geodesic from $\gG$ to $\Sigma$
that is the concatenation of two paths, the first a linear path from
$\gG$ to $\gG'$ in a single simplex, and the second a folding path from
$\gG'$ to $\gG$. The initial linear path keeps the underlying graph a
rose and its coarse projection is constant. Thus, the desired folding
path is the second path from $\gG'$ to $\Sigma$.

To achieve rationality, choose $\Sigma$ to be a rose with rational
edge lengths. Let $\phi:\gG\to \Sigma$ be an optimal map after
adjusting the metric so that $\Delta(\phi)=\gG$. If the vertex of
$\gG$ maps to the vertex of $\Sigma$, rationality is
automatic. Otherwise the vertex of $\gG$ maps to a point in the
interior of some edge. Perturb $\phi$ so that this point is rational,
and adjust the edge lengths in $\gG$ so that the perturbed map is
optimal, with the same train track structure. The new map satisfies
rationality. 
\end{proof}

\section{More on folding paths}\label{s:folding}
We now discuss folding in more detail. Let $\hG_t$, $t\in [0,\o]$, be
a folding path in $\hX$ (from now on we replace trees by quotient
graphs) with the natural parametrization. (See Notation~\ref{n:folding path} to recall our conventions.) So for $s<t$ we have
maps $\hat\phi_{st}:\hG_s\to \hG_t$ that have slope 1 on each edge,
immerse each edge, and induce train track structures.  \vskip 0.5cm
\noindent {\bf Unfolding.}  Traversing a folding path in reverse is
{\it unfolding}.  The main result of this subsection is
Theorem~\ref{t:gadget} giving local and global pictures of both
folding and unfolding. This result will only be used for the proofs of
technical Propositions~\ref{gafa}(surviving illegal turns) and
\ref{left-right}(legal and illegal) which should be skipped in a first
reading. In fact, Theorem~\ref{t:gadget} is obvious in the case of
Stallings' rational paths (see \ref{slope 1 folding}.A) and so could
be avoided altogether (see Corollary~\ref{c:coarse paths}).

\begin{definition}
An {\it (abstract) widget $W$ of radius $\epsilon$} is a metric graph
that is a cone on finitely many, but at least 2, points all the same
distance $\epsilon$ from the cone point. A widget has a canonical
morphism to $W\to [0,\epsilon]$ that sends the cone point to
$\epsilon$ and the other vertices to $0$.  There is also a canonical
path, parametrized by $[0,\epsilon]$, of finite trees from $W$ to
$[0,\epsilon]$ that locally folds all legal turns of the morphism with
speed 1. See Figure~\ref{f:abstract widget}. An {\it (abstract) gadget
  of radius $\epsilon$} is a union of finitely many widgets of radius
$\epsilon$.  The union is required to be disjoint except that widgets
are allowed to meet in vertices and is also required to be a forest.
There is a canonical path, parametrized by $[0,\epsilon]$, of forests
obtained by folding each widget.

\begin{figure}[h]
\begin{center}
\includegraphics[scale=0.5]{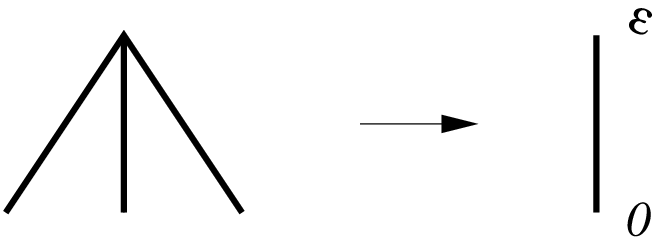}
\caption{An abstract widget of radius $\epsilon$.}
\label{f:abstract widget}
\end{center}
\end{figure}

Let $\hG\in\hX$. A {\it widget (resp.\ gadget) of radius $\epsilon$ in
  $\hG$}, is an embedding in $\hG$ of an abstract widget (resp.\
abstract gadget) of radius $\epsilon$. There is a natural path in
$\hX$ starting with $\hG$ and parametrized by $[0,\epsilon]$ given by
folding the gadget. See Figure~\ref{f:widgets}. An {\it (abstract)
  widget}, resp.\ {\it gadget} is a widget, resp.\ gadget, of some
radius.
\end{definition}

\begin{thm}\label{t:gadget}
  Let $\hG_t$, $t\in [0,\o]$ be a folding path in $\hX$ with its
  natural parametrization. There is a partition
  $0=t_0<t_1<\cdots<t_N=\o$ of $[0,\o]$ such that the restriction of $\hG_t$ to
  each $[t_i,t_{i+1}]$ is given by folding a gadget in
  $\hG_{t_i}$.
\end{thm}

In the context of Theorem~\ref{t:gadget}, we say that, as we traverse
$[t_i,t_{i+1}]$ in reverse, we are {\it unfolding a gadget}. The
analogous result holds for the induced path $\bG_t$, $t\in [0,L]$, in
$\X$ parametrized by arclength and we use the same terminology. For
example, we say $[0,L]$ has a finite partition into subintervals such
the restriction of $\bG_t$ to each subinterval is given by folding
(unfolding) a gadget.

Proposition~\ref{p:tame} is an immediate consequence of
Theorem~\ref{t:gadget}.

\begin{prop}\label{p:tame}
Let $\bG_t$, $t\in [0,L]$, be a folding path in $\X$. There is a
partition of $[0,L]$ into finitely many subintervals so that the
restriction of $\bG_t$ to each subinterval is a (reparametrized)
linear path in a simplex of $\X$.

In particular, a folding path in
$\X$ changes an open simplex only at discrete times. Outside these
times, illegal turns all belong to vertices with 2 gates, and one gate
is a single direction.\qed
\end{prop}
The restriction of a folding path to a simplex of $\X$ need not be
linear. This can happen, for example, if an illegal turn becomes
legal.

The rest of this subsection is devoted to proving
Theorem~\ref{t:gadget}. It is clear from our description of folding in
\ref{slope 1 folding}.C that, for each $t_*\in [0,\o)$, there is
$\epsilon>0$ such that the restriction of $\hG_t$ to
$[t_*,t_*+\epsilon]$ is given by folding a gadget in $\hG_{t_*}$. To
complete the proof of Theorem~\ref{t:gadget}, we will show that, for
each $t_*\in (0,\omega]$, there is $\epsilon>0$ such that the
restriction of $\hG_t$ to $[t_*-\epsilon,t_*]$ is given by folding a
gadget in $\hG_{t_*-\epsilon}$.

Let $N$ be the closed $\epsilon$-neighborhood of a vertex $v$ in
$\hG_{t_*}$ of valence $\ge 3$. We will describe the preimage
$N_\epsilon$ of $N$ in $\hG_{t_*-\epsilon}$ for small enough
$\epsilon>0$. As long as $\epsilon$ is small enough and $N_\epsilon\to
N$ is not injective, we will find a connected gadget of radius
$\epsilon$ in $N_\epsilon$ so that $N_\epsilon\to N$ folds this
gadget. The gadget needed to complete the proof will be the disjoint
union of these connected gadgets of radius $\epsilon$, one for each
such vertex of $\bG_{t_*}$. We will also equip $N$ and $N_\epsilon$
with height functions. For convenience, set
$\hat\phi_\epsilon:=\hat\phi_{t_*-\epsilon,t_*}$.

First we describe the height functions. Assume that $\epsilon$ is
small enough so that $N$ is a cone on a finite set with cone point
$v$. The height function on $N$ is the morphism $N\to
[-\epsilon,\epsilon]$ given as follows. If the length in $N$ of
$[v,w]$ is $\epsilon$ and the direction at $v$ determined by $[v,w]$
has more than one preimage in $N_\epsilon$, then map $[v,w]$
isometrically to $[0,\epsilon]$; otherwise map $[v,w]$ isometrically
to $[0,-\epsilon]$. The height function $h$ on $N_\epsilon$ is the
composition $N_\epsilon\to N\to [-\epsilon,\epsilon]$.

Now we describe $N_\epsilon$ for small enough $\epsilon$, and justify
this description immediately after that. In $N_\epsilon$, the preimage
of $v$ is the set of points of height 0. The set
$h^{-1}([0,\epsilon])$ is a gadget with widgets the closures of the
components of $h^{-1}((0,\epsilon])$.  Each of the height $\epsilon$
vertices has a unique direction in $\hG_{t_*-\epsilon}$ not in the
gadget, and we draw this direction upwards. Height 0 vertices may have
additional directions not contained in the gadget, and we draw those
downwards.  See Figure~\ref{f:widgets}.
\begin{figure}[h]
\begin{center}
\includegraphics[scale=0.5]{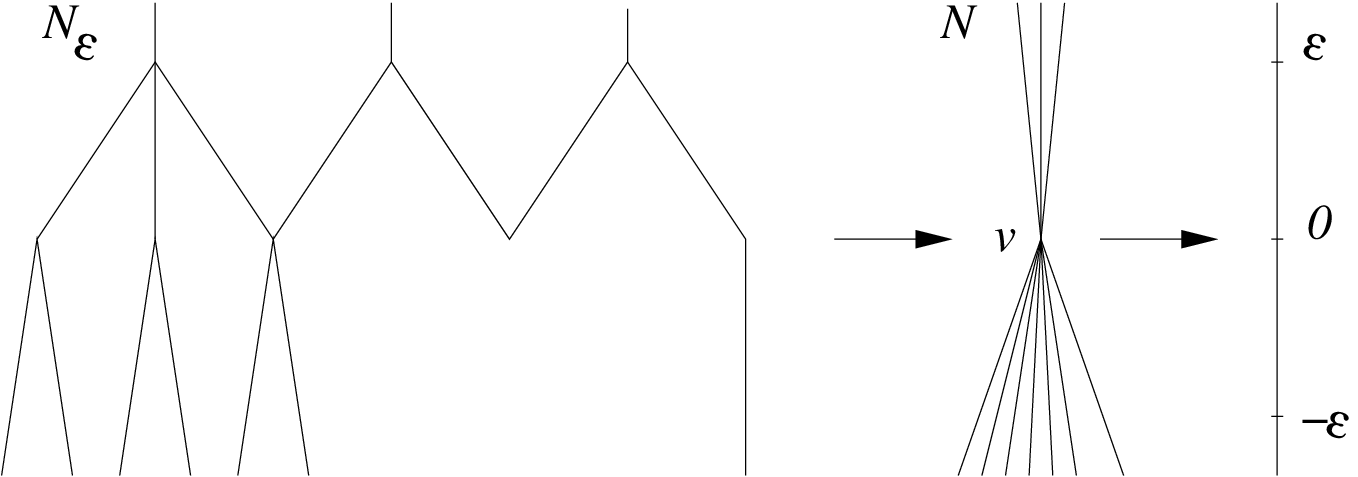}
\caption{An example of $N_\epsilon$ with 3 widgets, 3 vertices at
  height $\epsilon$ and 5 vertices at height 0. The union of the 3 widgets is a gadget.}
\label{f:widgets}
\end{center}
\end{figure}
All illegal turns in $N_\epsilon$ appear at vertices of height
$\epsilon$ and these vertices have two gates in $\hG_{t_*-\epsilon}$
(all downward directions form one gate and the single upward direction
is the other gate). In particular, all turns at the height 0 vertices
are legal.  After the widgets are folded, in $\hG_{t_*}$, the height 0
vertices get identified to $v$, each widget contributes an upward
direction at $v$, and the downward directions at $v$ come from
downward directions in $N_\epsilon$ based at height 0 vertices. Some
pairs of directions may be illegal in $\hG_{t_*}$, but in that case
they have to come from directions in $N_\epsilon$ that don't form a
turn (i.e.\ that are based at different vertices).

Now we justify our description of $N_\epsilon$. First choose
$\epsilon_0$ small enough so that:
\begin{enumerate}[(1)]
\item the closed $\epsilon_0$-neighborhood of $v$ is a cone on a
  finite set with cone point $v$ and
\item the cardinality of the preimage of $v$ in $N_{\epsilon}$ and the number
  of directions based at points in the preimage of $v$ is independent
  of $0<\epsilon\le\epsilon_0$;
\end{enumerate}
Note that it is possible to choose $\epsilon_0$ satisfying (2)
because, since all $\hat\phi_{st}$ are surjective, the cardinality of
$\hat\phi_\epsilon^{-1}(v)$ is non-increasing as $\epsilon$ decreases, and similarly the number of directions based at
$\hat\phi_\epsilon^{-1}(v)$ is non-increasing.

Choose $0<\epsilon_1<\epsilon_0$ so that the closed
$\epsilon_1$-neighborhood $N'$ in $N_{\epsilon_0}$ of the preimage of
$v$ contains no valence $\ge 3$ vertices of $\hG_{t_0-\epsilon_0}$
other than the preimages of $v$. We claim that, for
$0<\epsilon\le\epsilon_1$, $N_\epsilon$ has the description given
above.  It is enough to consider the case $\epsilon=\epsilon_1$
because our description of $N_\epsilon$ is stable under decreasing
$\epsilon$.

By definition, $N_{\epsilon_1}$ is the preimage in
$\hG_{t_*-\epsilon_1}$ of $N$. It is equal to the closed
$\epsilon_1$-neighborhood of the preimage of $v$. (Indeed, it is clear
that the neighborhood is in $N_{\epsilon_1}$. If $w\in N_{\epsilon_1}$
does not map to $v$, then since $\hG_{t_*-\epsilon_1}$ has a train
track structure there is a direction at $w$ whose image in $N$ points
toward $v$ and there is a legal path of length $|h(w)|$ starting at
this direction. The endpoint of the path then maps to $v$.) Similarly,
$N'$ is the preimage in $\hG_{t_*-\epsilon_0}$ of $N$. In particular,
$N'\to N_{\epsilon_1}$ is surjective.

$N'$ is a disjoint union of cones on points of height $\pm\epsilon_1$
with cone points the preimages of $v$. Notice that all turns in $N'$
are legal (or else (2) fails). $N'\to N_{\epsilon_1}$ is an embedding
off the points of height $\pm\epsilon_1$ in $N'$ (otherwise there
would be a path in $N_{\epsilon_1}$ between distinct directions at
preimages of $v$ whose image has outradius with respect to $v$ that is
less than $\epsilon_1$, by Lemma~\ref{l:folding segments} these
distinct directions would then be identified before time $t_*$, and we
contradict (2)). In fact, $N'\to N_{\epsilon_1}$ is an embedding
off points of height $\epsilon$ (otherwise there would be a path
$\sigma$ in $N_{\epsilon_1}$ that is cone on a pair of distinct points
in the preimage of $v$ with cone point of height $-\epsilon_1$,
contradicting the definition of $h$ since the directions at the ends
of $\sigma$ are identified in $\hG_{t_*}$). We see that
$N_{\epsilon_1}$ is the obtained from $N'$ by identifying some pairs
of points of height $\epsilon$.  The picture is completed with a few
observations.
\begin{itemize}
\item {\it $N_{\epsilon_1}$ is connected:} By Lemma~\ref{l:folding
    segments}, any two points in the the preimage of $v$ in
  $N_{\epsilon_1}$ are connected by a path in $N_{\epsilon_1}$
  (or else they aren't identified in $\hG_{t_*}$). So, the preimage of
  $v$ is contained in a single component and from the picture we've
  developed so far we see that $N_{\epsilon_1}$ is connected.
\item {\it $N_{\epsilon_1}$ is a tree:} A loop $\sigma$ in $N_{\epsilon_1}$
  has homotopically trivial image in $N$. Since
  $\hat\phi_{\epsilon_1}$ is a homotopy equivalence, $\sigma$ is
  homotopically trivial.
\item {\it Every point $w$ in $N_{\epsilon_1}$ of height $\epsilon_1$ has
  at least two downward directions:} Otherwise, there is a point $v'$ in the
  preimage of $v$ and a path $[v',w]$ of increasing height with the
  property that the induced direction at $v'$ is not identified by
  $\hat\phi_{\epsilon_1}$ with any other direction. This contradicts
  the definition of positive height.
\end{itemize}
This completes the proof of Theorem~\ref{t:gadget}.

Fix $\epsilon\le\epsilon_1$. We introduce a little more terminology
for later use. By construction, every point in $\hG_{t_*}$ at distance
$\epsilon$ from $v$ has a unique direction pointing away from $v$. If
$d$ is a direction at $v$, there is then a unique corresponding
direction $d'$ pointing away from $v$ and based at a point at distance
$\epsilon$ from $v$ determined by ``$d$ points to $d'$''. If the
height of $d'$ is $\epsilon$ then we denote $d'$ by $d^{\epsilon}$; if
the height of $d'$ is $-\epsilon$ then we denote $d'$ by
$d^{-\epsilon}$. We say that $d^{\epsilon}$ {\it points up} and
$d^{-\epsilon}$ {\it points down}. The direction $d^{\pm\epsilon}$ has
a unique lift $\tilde d^{\pm\epsilon}$ to $\hG_{t_*-\epsilon}$. We say
$\tilde d^{\epsilon}$ {\it points up } and $\tilde d^{-\epsilon}$ {\it
  points down }. See Figure~\ref{f:notation}.

\vskip 0.5cm
\noindent
{\bf Unfolding a path.}
Given an immersed path $\gamma$ in $\hG_\omega$, one may
try to ``lift'' it along the folding path, i.e.\ to find immersed paths
$\gamma_t$ in $\hG_t$ that map to $\gamma$ (up to homotopy rel
endpoints). This is always possible, since it is clearly possible
locally. At discrete times new illegal turns may appear inside the
path. Note that at discrete times the lifts are not unique, when an
endpoint of the path coincides with the vertex of an illegal turn.
Figure~\ref{f:unfolding} illustrates the nonuniqueness
of lifts.

\begin{figure}[h]
\begin{center}
\includegraphics[scale=0.5]{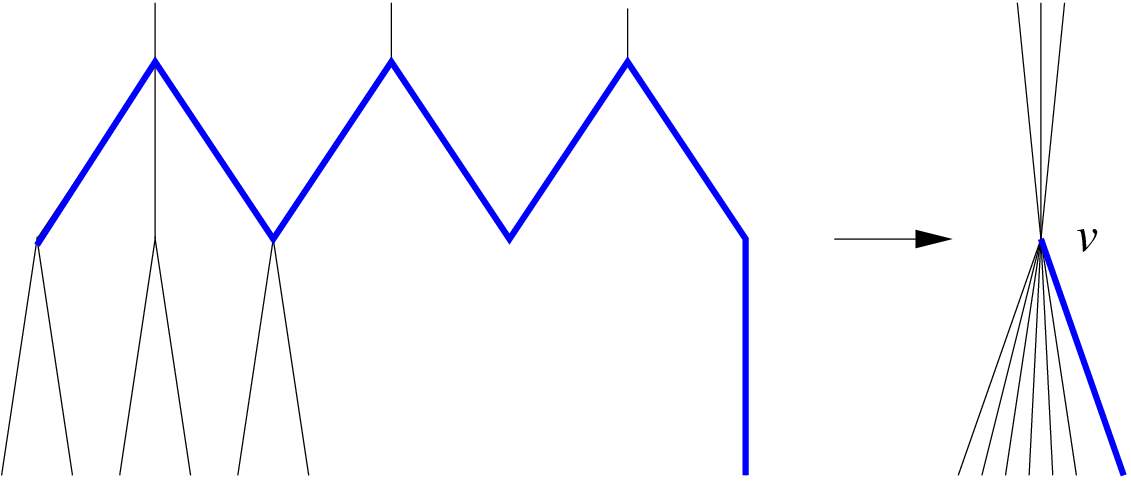}
\caption{The figure illustrates the ambiguity in lifting paths under
  unfolding which we see is parametrized by the preimage of $v$.}
\label{f:unfolding}
\end{center}
\end{figure}

To get uniqueness, we can remove the end of the path that lifts
nonuniquely. Thus we may have to remove segments\footnote{We use
  ``segment'' synonymously with ``immersed path'', but usually to
  connote a smaller piece of something.} at the ends whose size grows
at speed 1. Now suppose there are illegal turns in the path. As we
unfold, each illegal turn makes the length of the path grow with speed
2, and the illegal turn closest to an end moves away from the end at
speed 1. We deduce that {\it lifting is unique between the first and
  last illegal turns along the path $\gamma$, including the germs of
  directions beyond these turns.} We call this the {\bf unfolding
  principle}.

In particular, this applies to illegal turns themselves: if a loop
$z|\hG_\omega$ has two occurrences of the same illegal turn, pulling
back these turns produces two occurrences of the same path (most of
the time a neighborhood of a single illegal turn, but see
Figure~\ref{f:unfold.v2}).
\begin{figure}[h]
\begin{center}
\includegraphics[scale=0.5]{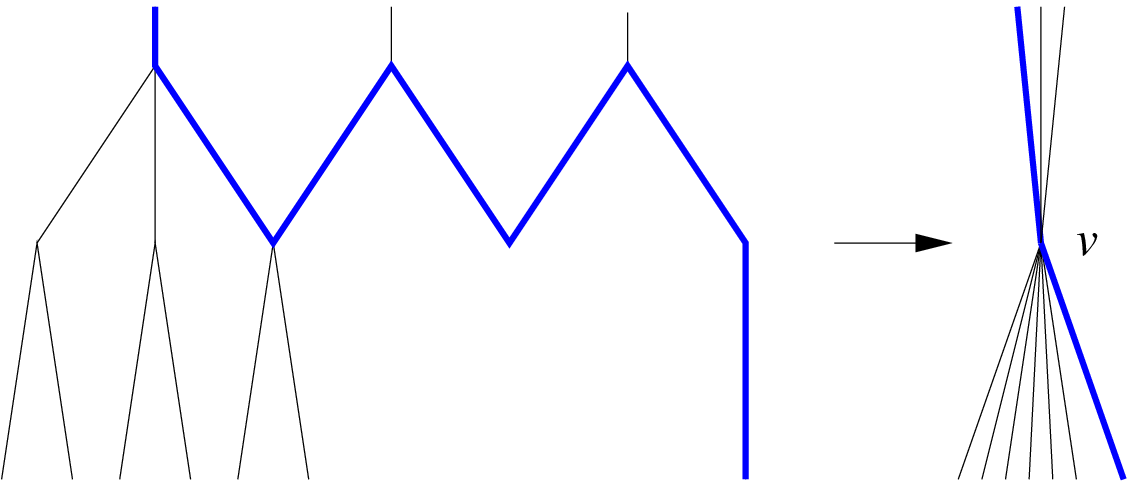}
\caption{A path unfolds to a path with two illegal turns.}
\label{f:unfold.v2}
\end{center}
\end{figure}
But note that distinct illegal turns might pull back to the same
illegal turn, see Figure~\ref{f:fold}.

\begin{figure}[h]
\begin{center}
\includegraphics[scale=0.4]{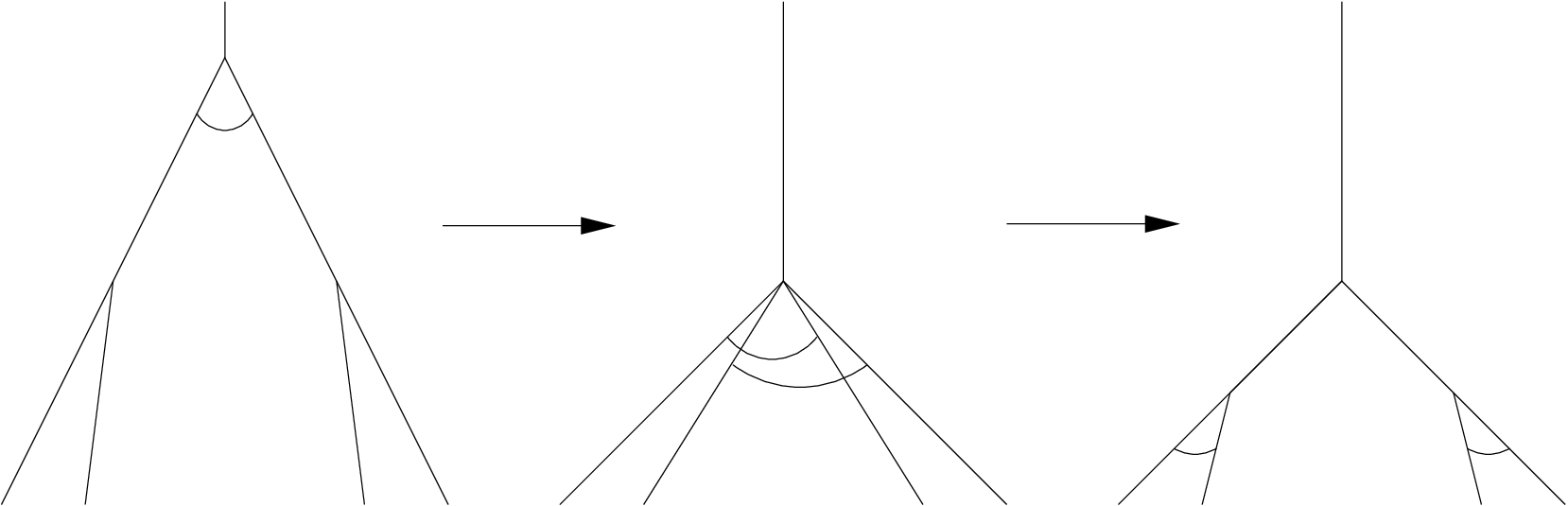}
\caption{In this folding path the number of illegal turns
  grows from 1 to 2.}
\label{f:fold}
\end{center}
\end{figure}

\vskip 0.5cm
\noindent {\bf \Ill.} \label{p:illegality} If a graph $\gG$ is
equipped with a track structure, by the {\it \ill\ of $\gG$} we mean
$$\m(\gG)=\sum_{v}\sum_{\Omega_v}(|\Omega_v|-1)$$
where the sum is over all vertices $v$ of $\gG$ and all gates $\Omega_v$
at $v$. Thus a gate that contains $k\geq 1$ directions contributes
$k-1$ to the count. The
right derivative of the function
$$t\mapsto vol(\hG_t)$$ at $t=t_0$ is $-\m(\hG_0)$, the negative of
the \ill\ of $\hG_{t_0}$. If $\gG$ is marked and $z$ is a conjugacy
class in $\FF$ then $\k(z|\gG)$ is by definition the number of illegal
turns in $z|\gG$, i.e.\ the number of illegal turns in the illegal
turn structure on $z|\gG$ induced by $z|\gG\to\gG$.

{\bf For the rest of this section,} $\bG_t$, $t\in [0,L]$, is a
folding path in $\X$ parametrized by arclength. We will sometimes
abbreviate $\m(\bG_t)$ by simply $\m_t$, $z|\bG_t$ by $z_t$ and
$\k(z_t)$ by $\k_t$.

\begin{lemma}\label{derivative}
\begin{enumerate}[(1)]
\item\label{i:loop} Let $z$ be a conjugacy class in $\FF$. The the
  right derivative of the length function $t\mapsto\ell(z_t)$ at $t=0$
  is
$$\frac d{dt}\ell(z_t)|_{t=0^+}=\ell(z_{0})-2\frac{\k_0}{\m_0}$$
\item\label{i:segment} Let $\sigma_0$ be a nondegenerate immersed path
  in $\bG_0$ whose initial and terminal directions are in illegal
  turns of $\bG_0$. Then, for small $t$, there is a corresponding path
  $\sigma_t$ whose initial and terminal directions are in illegal
  turns of $\bG_t$. The right derivative at 0 of the length $L_t$ of
  $\sigma_t$ is $L_0-2\frac{\k_0}{\m_0}$, where now $\k_0-1$ is the
  number of illegal turns in the interior of $\sigma_0$. If the
  initial and terminal directions of $\sigma_0$ are in the same gate,
  then the same is true for $\sigma_t$.
\end{enumerate}
\end{lemma}

\begin{proof}(1): Let $\hG_t$ be the naturally parametrized by path in $\hX$ so that
  $\bG_{s(t)}=\G_t=\hG_t/vol(\G_t)$ and $vol(\hG_0)=1$ (see
  Notation~\ref{n:folding path}). For small $t\ge 0$,
  $vol(\hG_t)=1-\m_0t$, $\ell(z|\hG_t)=1-2\k_0t$, and
  $s(t)=-\log(vol(\hG_t))$. The proof of the derivative formula in (2)
  is identical.
\end{proof}

We will sometimes abuse notation and write $\ell'(z_{t_0})$ for $\frac
d{dt}\ell(z|\bG_t)|_{t=t_0^+}$.

\begin{cor}\label{c:length}
Let $z$ and $w$ be conjugacy classes in $\FF$ and suppose $\m_t$ and
$\k_t=\k(z_t)$ are constant for $t\in [0,\epsilon)$.
\begin{enumerate}[(1)]
\item\label{i:local}
The length $\ell(z_t)=a e^t+b$ on $[0,\epsilon)$ where $a=\ell(z_0)-\frac{2\k_0}{\m_0}$ and $b=\frac{2\k_0}{\m_0}$.
\item\label{i:compare}
If $\ell'(z_0)\geq \ell'(w_0)$ then $\ell(z_t)-\ell(w_t)$ is
nondecreasing on $[0,\epsilon)$.
\item\label{i:average}
If the average length $A$ of a maximal legal segment in the loop
$z_0$ is $>2/\m_0$ the loop grows in length on $[0,\epsilon)$, and
if it is $<2/\m_0$ it shrinks.
\end{enumerate}
\end{cor}

\begin{proof}
(\ref{i:local}) and (\ref{i:compare}) follow easily from
  Lemma~\ref{derivative}. For (\ref{i:average}), $\k_0A=\ell(z_0)$
  and so, by (\ref{i:local}), $a=\ell(z_0)(1-2/\m_0A)$.
\end{proof}

\begin{cor}\label{c:global}
Let $z$ be a conjugacy class in $\FF$. The length $\ell(z_t)$ is
piecewise exponential on $[0,L]$, i.e.\ $[0,L]$ has a finite partition
such that the restriction of $\ell(z_t)$ to each subinterval is as
in Corollary~\ref{c:length}(\ref{i:local}).
\end{cor}

\begin{proof}
Since we may subdivide $[0,L]$ as $0=s_0<s_1<\cdots<s_k=L$ so that, on
each $[s_i,s_{i+1})$, $\m_t$ and $\k_t$ are constant, the corollary
  follows directly from Corollary~\ref{c:length}(\ref{i:local}) and
  Proposition~\ref{p:tame}.
\end{proof}

We will say that a segment {\it has endpoints illegal turns} if the
path has been infinitesimally extended at each end, i.e.\ directions
have been specified at the initial and terminal endpoints, and these
directions together with the segment determine illegal turns. If the
segment is degenerate then this means that the added directions form
an illegal turn.

\begin{cor}\label{c:segments}
Let $\sigma_L$ be an immersed path in $\bG_L$ with endpoints illegal
turns and, for $t\in [0,L]$, let $\sigma_t$ be the immersed path in
$\bG_t$ with endpoints illegal turns obtained by applying the
Unfolding Principle to $\sigma_L$. Corollaries~\ref{c:length} and
\ref{c:global} hold if $\ell(z_t)$ is replaced by the length of
$\sigma_t$.
\end{cor}

\begin{proof}
Since the derivative formula in
Lemma~\ref{derivative}(\ref{i:segment}) is obtained from that in
Lemma~\ref{derivative}(\ref{i:loop}) by such a
replacement, proofs of corresponding statements are
identical.
\end{proof}

\begin{cor}\label{3}
  A legal segment $\sigma_0$ of length $L_0\geq 2$ inside $z_0$
  gives rise to a legal segment $\sigma_t$ of length $L_t\geq
  2+(L_0-2)e^t$ inside $z_t$ for $t\in [0,L]$. In particular, a
  legal segment of length $L_0\geq 3$ grows exponentially.
\end{cor}

\begin{proof}
$(L_t-2)'\ge L_t-2$ using Lemma~\ref{derivative}(\ref{i:segment}).
\end{proof} 

Similar considerations control the lengths of topological edges $e$ of
$\bG_t$ that are not {\it involved}\, in any illegal turn, i.e.\ the
directions determined by the ends of $e$ aren't in any illegal turns.

\begin{lemma}\label{l:good edges}
Suppose that $e_0$ is an edge of $\bG_\0$ that is not involved in an
illegal turn. Then for small $t$, there is a corresponding edge $e_t$
of $\bG_t$ not involved in an illegal turn. Further, the length $L_t$
of $e_t$ satisfies $L_t=L_0e^t$.
\end{lemma}

\begin{proof}
For small $t$, the morphism $e^t\bG_0\to\bG_t$ restricted to $e_0$ is an isometric embedding with image an edge of $\bG_t$.
\end{proof}

\begin{lemma}\label{l:bounded everywhere}
Suppose $z$ is a conjugacy class such that $k(z_0)$ and $\ell(z_L)$
are bounded. Then $\ell(z_t)$ is bounded for all $t\in [0,L]$.
\end{lemma}

\begin{proof}
By Corollary~\ref{c:length}(\ref{i:average}) and
Corollary~\ref{c:global}, if $\ell(z_t)/\k(z_t)\ge\ell(z_t)/\k(z_0)\ge
2\ge 2/m_t$, then $\ell(z_t)$ grows.  Therefore,
$\ell(z_t)\le\max\{2k(z_0), \ell(z_L)\}$ for all $t\in [0,L]$.
\end{proof}

By a {\it surface relation} we mean a conjugacy class that, with
respect to some rose, crosses every edge twice and has a circle as its
Whitehead graph (equivalently, attaching a 2-cell to the rose via the
curve results in a surface).

\begin{lemma}\label{monogon}
Suppose $\m_t\ge m$, for all $t\in [0,L]$, and let $w$ be a conjugacy
class in $\FF$. Assume $\k(w_0)=m$. If $\ell(w_L)\leq K$ then
either
\begin{enumerate}[(i)]
\item there is a simple class $u$ such that:
\begin{itemize} 
\item
$\k(u_0)$ is bounded; and
\item
$\ell(u_t)$ is bounded by a function of $K$ for all $t\in [0,L]$ 
\end{itemize}
(in particular $d_\F(\bG_0,\bG_L)$ is bounded by a function of $K$);
or
\item $w$ is a surface relation.
\end{enumerate}
Moreover, if $\m_t>m$ for all $t$ then
(i) holds. 
\end{lemma}

\begin{proof}  
Arguing as in Lemma~\ref{l:bounded everywhere}, we have
$\ell(w_t)\leq\max\{2,K\}$ (loops of length $>2$ with no more illegal
turns than the \ill\ grow under folding). So we may take $u=w$
provided $w$ is simple. See Lemma~\ref{bounded crossing 2} and
Remark~\ref{r:bounded crossing 2} for the parenthetical remark. We now
consider four cases.

{\it Case 1.} $\ell(w_0)<2$. Then $w$ is simple as $w_0$
crosses some edge at most once.

{\it Case 2.} $\ell(w_0)=2$. Then either $w$ is simple or
$w_0$ crosses every edge exactly twice. In the latter case,
collapse a maximal tree in $\bG_0$ -- with respect to the
resulting rose the Whitehead graph of $w$ is either a circle (and then
$w$ is a surface relation) or the disjoint union of at least two
circles (and then $w$ is simple). 

{\it Case 3.} $2<\ell(w_0)<2+injrad(\bG_0)$. Then either
$w$ is simple or $w_0$ crosses every edge at least
twice. Assume the latter. Under our assumption the edges crossed more
than twice form a forest. Collapse a maximal tree that contains this
forest and argue as in Case~2.

{\it Case 4.} $\ell(w_0)\geq 2+injrad(\bG_0)$. Choose a conjugacy
class $v$ with $\ell(v_0)=injrad(\bG_0)$. We now claim that
$\ell(v_t)\leq\ell(w_t)-2$ for all $t$. This is clearly true
at $t=0$. In fact, this condition {\it persists} in that there is no
last time $t_0<L$ where it is true. Indeed,
Lemma~\ref{derivative}(\ref{i:loop}) shows that $$\ell'(v_{t_0})\leq \ell(v_{t_0}) \leq
\ell(w_{t_0})-2\leq \ell'(w_{t_0})$$ and so
the inequality continues to hold for $t>t_0$ (see
Corollary~\ref{c:length}(\ref{i:compare})).  Thus $v$ is a simple
class with both $\ell(v_t)$ bounded. We may take $u=v$.

For the moreover part, we have $\ell'(w_t)\geq
\ell(w_t)-2\frac m{m+1}$ for all $t$. Thus if $\ell(w_0)<2$
then $w$ is simple and the statement follows with $u=w$. If
$\ell(w_0)\geq 2$, then we claim $d_\X(\bG_0,\bG_L)$ is bounded
(see Corollary~\ref{proj X->F continuous}). Indeed, by
Corollary~\ref{c:length} we have $\ell(w_t)\ge ae^t+b$ where
$a=\ell(w_0)-2\frac{m}{m+1}\ge \frac{2}{m+1}$ and $b\ge 2$. In
particular, $K\ge \frac{2}{m+1}e^L+2$. We may take $u=v$ such that
$\ell(v_0)=injrad(\bG_0)$.
\end{proof}

We also have the following variant.

\begin{lemma}\label{monogon2}
Suppose in addition to the hypotheses of Lemma~\ref{monogon} that, for
some illegal turn in $w|\bG_0$, one of the two edges $e$ forming the
turn is nonseparating and has length a definite fraction $p>0$ of
$injrad(\bG_0)$. Then either:
\begin{enumerate}[(i\,$'$)]
\item there is a simple class $u$ such that:
\begin{itemize}
\item
$\k(u_0)$ is bounded; and
\item
$\ell(u_t)$ is bounded
  as a function of $K$ and $p$ for all $t\in [0,L]$
\end{itemize}
(in particular 
$d_\F(\bG_0,\bG_L)$ is bounded in terms of $K$ and $p$); or
\item $w$ is a surface relation and any class $z$ such that $z_0$
  contains a segment $S=e\cdots e$ that closes up (by identifying the
  two copies of $e$) to $w_0$ fails to be simple.
\end{enumerate}
\end{lemma}

\begin{proof}
Referring to the proof of Lemma~\ref{monogon}, in Cases~1 and 4, ({\it i\,$'$})
holds; so assume we are in Cases~2 or 3. In fact we are free to assume
$\ell(w_0) < 2+p~ injrad(\bG_0)$, for otherwise the argument of
Case~4 shows that for $v$ with $\ell(v_0)=injrad(\bG_0)$ we have
$\ell(v_t)\leq \frac{\ell(w_t)-2}p$ for any $t$ and ({\it
  i\,$'$}) follows (and this time the bound also depends on $p$). Now
the forest consisting of the edges crossed by $w_0$ more than
twice (assuming all edges are crossed at least twice) does not include
$e$, and we may collapse a maximal tree that contains this forest but
does not contain $e$. Now $z_0$ can be thought of as $e\cdots
e\cdots=eAeB$ with the subpath $S=eAe$ giving $w$. Since $e$ is not
collapsed, the Whitehead graph of $z$ in the rose contains the
Whitehead graph of $w$, which is a 1-manifold. So if $w$ is not
simple, neither is $z$.
\end{proof}

\section{Loops with long illegal segments}\label{s:gafa}
The key result of this section is Proposition~\ref{gafa}(surviving
illegal turns) which is a generalization of \cite[Lemma~2.10]{gafa}.
Before stating Proposition~\ref{gafa}, we
need a bit of terminology and some preliminary lemmas to be used in
the proof. In this section, $\bG_t$, $t\in [0,L]$, is a folding path
in $\X$ parametrized by arclength.  Consider a conjugacy class $z$ in
$\FF$ and the induced path of loops $z_t:=z|\bG_t$. The illegal turns
along $z_t$ are folding as $t$ increases, but at discrete times an
illegal turn may become legal, or several illegal turns may collide
and become one (e.g.\ see Figure~\ref{f:unfold.v2}). We say that a
consecutive collection of illegal turns along $z$ {\it survives to
  $\bG_L$} if none of them become legal nor do they collide with a
neighboring illegal turn in the collection, for any $t\in [0,L]$. In
particular, each illegal turn in the collection in $z_t$ unfolds to a
single illegal turn in the collection in $z_{t'}$ for $t'\le t$. We
will call the portion\footnote{usually a segment, but possibly all of
  $z_t$} of $z_t$ spanned by our collection the {\it good portion of
  $z_t$}. The turns in the collection in $z_t$ are, in order,
$\tau_{t,1}, \tau_{t,2},\dots$ and have vertices $p_{t,1},
p_{t,2},\dots$. In particular, if $t'<t$ then $\tau_{t,i}$ unfolds to
$\tau_{t',i}$. The image of $\tau_{t,i}$ in $\bG_t$ is
$\bar\tau_{t,i}$. In this context, the Unfolding Principle gives the
implication
$\bar\tau_{t,i}=\bar\tau_{t,j}\implies\bar\tau_{t',i}=\bar\tau_{t',j}$.
We also say that $\bar\tau_{t,i}$ {\it unfolds to} $\bar\tau_{t',i}$.

Suppose that a consecutive collection of illegal turns along $z$
survives to $\bG_L$. For each $t\in [0,L]$ denote by $\T_t$ the set of
turns that occur in the given consecutive collection, i.e.\
$\T_t=\{\bar\tau_{t,1}, \bar\tau_{t,2},\dots\}$. Let $D_t$ denote the
set of directions in $\bG_t$ that occur in a turn in $\T_t$.  Of
course, $D_t$ is partitioned into equivalence classes with respect to
the relation ``being in the same gate'', but we consider a finer
equivalence relation generated by $d\sim d'$ if $\{d,d'\}\in\T_t$. We
will call the equivalence classes {\it subgates}. The {\it Whitehead
  graph of $\bG_t$}, denoted $\W(\bG_t)$ or often just $\W_t$, is
the simple\footnote{no multiple edges, no edges that are
  loops} graph whose vertices are directions at
vertices of $\bG_t$ and edges are illegal turns in $\bG_t$. Note that directions at vertices are in the same gate if and only if they are in the same component of $\W_t$. Define
$\SW(\bG_t)$, or just $\SW_t$, to be the subgraph of $\W_t$ spanned by
the edges in $\T_t$, i.e.\ a vertex in $\SW_t$ is a direction in a
turn in $\T_t$ and an edge is an element of $\T_t$. Note that
directions are in the same subgate if and only if they are in the same
component of $\SW_t$.

\begin{lemma}\label{sublemma}
In the situation above, let $d_1,d_2,\cdots,d_k$ be the vertices along
an embedded closed curve in $SW_t$ (that is,
$d_i\neq d_j$ for $i\neq j$). Then for any $t'<t$ there is
an induced embedded closed curve in $SW_{t'}$.
\end{lemma}

Specifically, each turn $\{d_i,d_{i+1}\}$ (taken $\mbox{mod } k$, so
including $\{d_k,d_1\}$) unfolds to a turn in $\T_{t'}
$; Lemma~\ref{sublemma} says that these turns also form an embedded closed
curve in a subgate (in particular, all are based at the same vertex).

\begin{proof}[Proof of Lemma~\ref{sublemma}]
Let $\hG_{s(t)}$ be the path in $\hX$ giving rise to $\bG_t$ (see
Notation~\ref{n:folding path}).  As combinatorial graphs, we identify
$\bG_t$ and $\hG_{s(t)}$ (they differ only by homothety). In
particular, a direction in one is naturally identified with a
direction in the other and we will use the same names for two such
directions. If we set $s:=s(t)$ and $s':=s(t')$, it suffices to argue
in the case $s'=s-\epsilon$ for small $\epsilon>0$; for the conclusion
clearly holds in the limit. To that end, we use the description of
$N_\epsilon$ developed in Section~\ref{s:folding}. So, let $v$ be the
vertex in $\hG_s$ that is the base of the directions $d_i$ in our
subgate, $N$ be the $\epsilon$-neighborhood in $\hG_s$ of $v$, and
$N_\epsilon$ be the preimage of $N$ in $\hG_{s'}$. Set
$d_i^*=d_i^\epsilon$ or $d_i^{-\epsilon}$ depending on whether or not
$d_i$ points up or down and similarly for $\tilde d_i^*$. For $i\not=
j$, let $[d^*_i,d^*_j]$ denote the path in $\hG_s$ from the base of
$d^*_i$ to the base of $d^*_j$ extended infinitesimally by the
outgoing (germs of) directions $d^*_i$ and $d^*_j$. Similarly,
$[\tilde d^*_i,\tilde d^*_j]$ denotes the unique immersed path lifting
$[d^*_i,d^*_j]$. Note that if $z|\hG_s$ contains the illegal turn
$\{d_i,d_j\}$ then it also contains the (infinitesimally extended)
path $[d^*_i,d^*_j]$ and $z|\hG_{s'}$ contains $[\tilde d^*_i,\tilde
  d^*_j]$. If $\tilde d_i^*$ points down then it is {\it supported} by
a widget $W$ if $\tilde d_i^*$ is in the downward direction from some
$x\in W$ of height 0. In this case, we also say that $\tilde d_i^*$ is
{\it supported} by $x$. If $\tilde d_i^*$ points up, then it is {\it
  supported} by the unique widget at which it is based. See
Figure~\ref{f:notation}.

\begin{figure}[h]
\begin{center}
\includegraphics[scale=0.5]{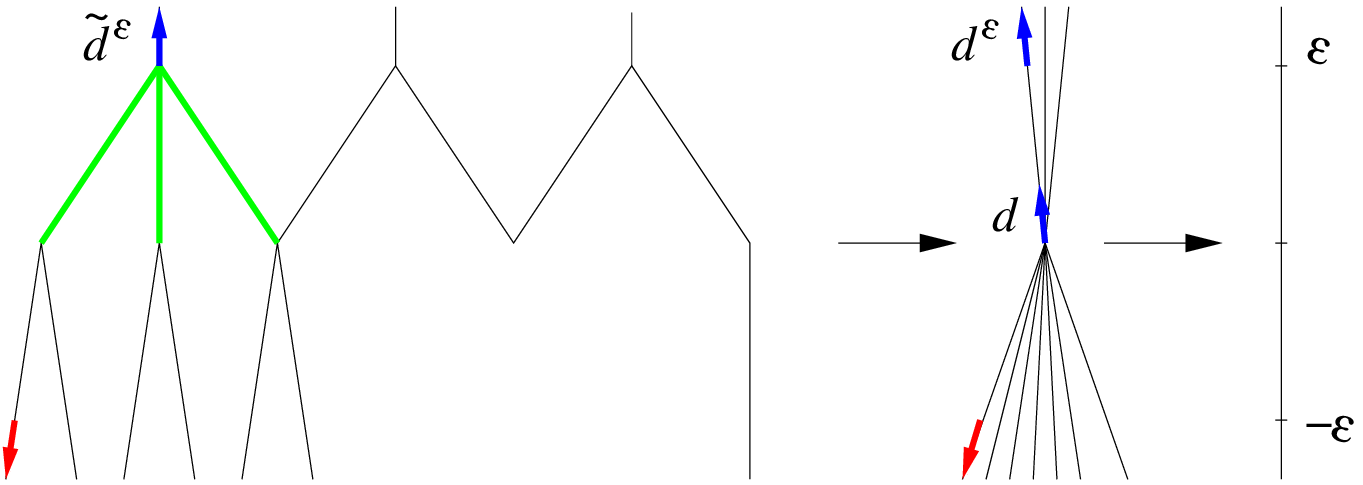}
\caption{An upward and a downward direction supported by the leftmost widget.}
\label{f:notation}
\end{center}
\end{figure}

By hypothesis, $[\tilde d^*_i,\tilde d^*_{i+1}]$ contains a unique
illegal turn that we abusingly denote $\{\tilde d^*_i,\tilde
d^*_{i+1}\}$.  Recall that all illegal turns in $\bG_{s'}$ have height
$\epsilon$ and are in a widget. Note that if $[\tilde d^*_i,\tilde
  d^*_{i+1}]$ (or indeed any path between height $\pm\epsilon$
vertices) contains its unique illegal turn in $W$ then it falls into one of the
following three cases:
\begin{itemize}
\item
$\tilde d^*_i$ and $\tilde d^*_{i+1}$ point up in distinct widgets
  adjacent to $W$
\item
$\tilde d^*_i$ and $\tilde d^*_{i+1}$ point down and are supported by $W$
\item one of $\tilde d^*_i$ or $\tilde d^*_{i+1}$ points down and is supported by $W$ and the other points up in a widget adjacent to $W$
\end{itemize}

We now make the following observations.
\begin{enumerate}[(1)]
\item If $\tilde d^*_i$ points up and is supported by the widget $W$,
  then all directions pointing upward at height 0 in $W$ are mapped to
  $d_i$. If $\tilde d^*_i$ points down, then it is in a unique
  downward direction from a height 0 point and this direction maps to
  $d_i$.

\item The directions $\tilde d^*_i$ are all distinct. Indeed, by
  hypothesis the $d_i$ (hence the $d^*_i$) are distinct, and $\tilde
  d^*_i$ is the unique lift of $d^*_i$.

\item It is not possible for a single widget to support both an upward
  $\tilde d^*_i$ and a downward $\tilde d^*_j$. Indeed, this would
  force $\{d_i,d_j\}$ to be illegal and then $[\tilde d^*_i,\tilde
    d^*_j]$ would have a height 0 illegal turn. However, all illegal
  turns occur at height $\epsilon$.

\item It is not possible for adjacent widgets to both support upward
  $\tilde d^*_i$ and $\tilde d^*_j$. This is because this would force an
  illegal turn at the common height 0 vertex formed by directions that map to
  $d_i$ and $d_j$. 

\item It is not possible for downward $\tilde d^*_i$ and $\tilde
  d^*_j$ to be supported by the same height 0 vertex. Indeed, this
  would then force an illegal turn based at this height 0 vertex.
\end{enumerate}

Recall that we want to prove $\{\tilde d^*_1,\tilde d^*_{2}\},
\{\tilde d^*_2,\tilde d^*_{3}\},\dots,\{\tilde d^*_k,\tilde d^*_{1}\}$
gives rise to an embedded closed curve in $\SW_{s'}$. To do this we
prove two things.

\begin{itemize}
\item (there is a loop) The turns $\{\tilde d^*_i,\tilde d^*_{i+1}\}$
  are all based at the same vertex $w$, i.e.\ the base of the illegal turn
  crossed by $[\tilde d^*_i,\tilde d^*_{i+1}]$ is independent of $i$.
\item (the loop is embedded) For $i\not=j$, $\tilde d^*_i$ and $\tilde d^*_{j}$ are not in the same direction from $w$.
\end{itemize}

To see there is a loop, suppose we have three consecutive directions
$\tilde d^*_{i-1},\tilde d^*_i,\tilde d^*_{i+1}$ that determine two
illegal turns in $N_\epsilon$ not based at the same vertex. There are
two cases. First suppose $\tilde d^*_i$ points down and is supported
by some height 0 vertex $x$. Paths from $\tilde d_i^*$ to $\tilde
d^*_{i\pm 1}$ lead through two distinct widgets each containing
$x$. Since our directions $\tilde d^*_1,\dots,\tilde d^*_k$ are
cyclically ordered and $N_\epsilon$ is a tree, there must be some $j$
with $j\neq i\neq j+1$ so that $[\tilde d^*_j,\tilde d^*_{j+1}]$
passes through $x$. (Indeed, otherwise all $\tilde d^*_j$, for
$j\not=i$, lie in the same component of
$N_\epsilon\setminus\{w\}$. Since $\tilde d^*_{i-1}$ and $\tilde
d^*_{i+1}$ lie in distinct components, this is a contradiction.) The
path $[\tilde d^*_j,\tilde d^*_{j+1}]$ cannot terminate at $\tilde
d^*_i$ by (2), and by (5) it cannot terminate in any downward
direction supported by $x$, but must continue to another (adjacent)
widget. By (3) the widgets containing $x$ do not support upward
$\tilde d^*_j$ and $\tilde d^*_{j+1}$, so the path crosses two illegal
turns, contradiction. The other case is that $\tilde d^*_i$ is upward,
say based at a height $\epsilon$ vertex $x$ inside a widget $W$. Then
there are distinct widgets $W_+$ and $W_-$ adjacent to $W$ so that $[\tilde d^*_i,\tilde d^*_{i\pm 1}]$ crosses an illegal
turn in $W_\pm$. Again there must be some $j$ so that $[\tilde d^*_j,\tilde d^*_{j+1}]$ either crosses $x$ (if $W_+\cap
W_-=\emptyset$) or crosses the intersection point $W_+\cap W_-$ (if
there is one, and then this point is in $W$ as well). In the latter
case the path does not terminate at any direction supported by this
point by (3) nor at an upward direction supported by $W_\pm$ by
(4). Thus this path has two illegal turns, contradiction. In the
former case, $[\tilde d^*_j, \tilde d^*_{j+1}]$ must
have at least 3 illegal turns: one at $x$ and one on each side of $x$,
by (3) and (4).

We have established the first item. To see that the loop is embedded,
suppose that there are $i\not= j$ such that $\tilde d^*_i$ and $\tilde
d^*_j$ are in the same direction from $w$. Each of these directions is
contained in a path between height $\pm \epsilon$ vertices with one
illegal turn, that illegal turn being based at $w$. Considering the
three cases listed above for such paths, $\tilde d^*_i$ and $\tilde
d^*_j$ have to either be both downward and supported by the same
height 0 vertex of the widget $W$ containing $w$, contradicting (5),
or one is downward and the other upward supported in the same widget
adjacent to $W$, contradicting (3), or they are upward and supported
by widgets adjacent to $W$ and to each other, contradicting (4).
\end{proof}

Given a finite, simple graph $\gG$, let $\sim$ be the equivalence
relation on the set of (unoriented) edges $\E(\gG)$ of $\gG$ generated
by $e\sim e'$ if there is an embedded loop containing $e$ and $e'$. We
also view $C\in\E(\gG)/\sim$ as a subgraph of $\gG$, specifically the
union of the edges in $C$. Let $\V(\gG)$ denote the set of vertices of
$\gG$ and let $\CV(\gG)\subset\V(\gG)$ be the set of cut vertices.

\begin{lemma}\label{l:graph}
  Suppose $\gG$ is a finite, simple, and connected graph. 
\begin{enumerate}[(1)]
\item\label{i:component} $e\sim e'$ if and only if they are not
  separated by any cut vertex $v\in\gG$, i.e.\ their interiors are in
  the same component of $\gG\setminus\{v\}$.
\item\label{i:count}
$|\V(\gG)|-1\ge |\E(\gG)/\sim|$ with equality iff $\gG$ is a tree.
\end{enumerate}
\end{lemma}

\begin{proof}
  (1): If an embedded circle contains $e$ and $e'$, then $e$ and $e'$
  are not separated by a cut vertex. The same then holds if
  $e\sim e'$. Now assume $e$ and $e'$ are separated by no cut vertex.
  Let $e=e_0, e_1, \dots, e_N=e'$ be an embedded edge path. We claim
  there is an embedded loop containing each $e_i, e_{i+1}$. Indeed, if
  $v$ is the common vertex of $e_i$ and $e_{i+1}$, then the claim is
  clear if $v$ is not a cut vertex, and the claim follows from our
  assumption if $v$ is a cut vertex.

  (2): Consider the bipartite graph $\Gamma$ whose vertex set is the
  disjoint union of $\E(\gG)/\sim$ and $\CV(\gG)$ and where
  $C\in\E(\gG)/\sim$ and $v\in\CV(\gG)$ span an edge if $v\in C$. We
  claim $\Gamma$ is a tree. Indeed, $\Gamma$ is connected since $\gG$
  is connected. Also, we see from the proof of (1) that the graph
  $C$ is connected and so an embedded loop in $\Gamma$ would give rise to an
  embedded loop in $\gG$ containing inequivalent edges. Note that
  elements of $\CV(\gG)$ are not leaves of $\Gamma$. Fix a leaf $C_0$
  of $\Gamma$ and consider the function $\CV\to\E(\gG)/\sim$ given by
  $v\mapsto C$ where $C$ is the vertex adjacent to $v$ in the
  direction of $C_0$. The image is the union of $\{C_0\}$ and
  complement of the set of leaves of $\Gamma$. The inequality follows
  by noting that, since every $C$ contains at least one edge, every
  $C$ has at least two vertices. The inequality is an equality iff $C$
  has valence one or two in $\Gamma$ and consists of a single edge,
  i.e.\ iff $\gG$ is a tree.
\end{proof}

We won't need it, but the proof also shows that the graphs $C$ have no
cut vertices of their own. Also, $e\sim e'$ iff $e=e'$ or there is an
embedded loop in $\gG$ containing $e$ and $e'$. (Hint: prove by
induction on $i$ that if $e\not= e'$ are equivalent and connected by a
length $i$ embedded edge path then this path can be extended to an
embedded loop.)

We now apply this to $\SW_t$, i.e.\ we consider the equivalence
relation on $\T_t$ generated by $\bar\tau_{t,i}\sim\bar\tau_{t,j}$ if
there is an embedded loop $\SW_t$ containing $\bar\tau_{t,i}$ and
$\bar\tau_{t,j}$. Recall that $\m_t=\m(\bG_t)$ is the \ill\ of
$\bG_t$, see Page~\pageref{p:illegality}.

\begin{lemma}\label{l:|classes|}
$|\,\T_t\,/\!\!\sim\!\!|\le\m_t$ with equality iff subgates
  coincide with gates and $\SW_t$ is a forest.
\end{lemma}

\begin{proof}
  We suppress the subscript $t$. Let $\W=\sqcup_i\W_j$ and
  $\SW=\sqcup_i\SW_i$ be decompositions into components. We have
$$|\T/\sim|\le\sum_i(|\V(\SW_i)|-1)\le \sum_j(|\V(\W_j)|-1)=m$$ where
  the first inequality follows from Lemma~\ref{l:graph}(\ref{i:count})
  (with equality iff $\SW$ is a forest) and the second follows because
  $\SW_t$ is a subgraph of $\W_t$ (with equality iff the sets of gates
  and subgates coincide).
\end{proof}

We are now ready for the main result of this section. Let $\hm$ denote
the maximal possible number of illegal turns for any train track
structure on any element of $\X$.

\begin{prop}[surviving illegal turns]\label{gafa}
  Let $z$ be a simple class and $\bG_t$, $t\in [0,L],$ a folding path
  in $\X$ parametrized by arclength. Assume that $M=2\hm+1$
  consecutive illegal turns of $z|\bG_0$ survive to $\bG_L$ and that
  the legal segments between them in $z|\bG_L$ have bounded size. Then
  $d_\F(\bG_0,\bG_L)$ is bounded.
\end{prop}

We saw above that two illegal turns in our consecutive collection in
$z|\bG_t$ that give the same element of $\T_t$ also give the
same element of $\T_{t'}$ for $t'\le t$, i.e.\ we saw
$\bar\tau_{t,i}=\bar\tau_{t,j}$ implies
$\bar\tau_{t',i}=\bar\tau_{t',j}$. However, distinct illegal turns
might unfold to the same illegal turn, i.e.\
$\bar\tau_{t',i}=\bar\tau_{t',j}$ does not imply
$\bar\tau_{t,i}=\bar\tau_{t,j}$. So $|\T_{t'}|\le|\T_t|$. By partitioning $[0,L]$ into a bounded number of subintervals
and renaming, we may assume in proving Proposition~\ref{gafa} that
$|\T_t|$ is constant on $[0,L)$. Likewise, we may assume that
$|\,\T_t\,/\!\!\sim\!\!|$ is constant (in general $|\,\T_t\,/\!\!\sim\!\!|$ may decrease under unfolding when a new circle is
formed). The proof of Proposition~\ref{gafa} breaks into two cases: for small $\epsilon>0$,
\begin{enumerate}[{Case}~1:]
\item
$\SW_{L-\epsilon}$ is not a forest;
\item
$\SW_{L-\epsilon}$ is a forest.
\end{enumerate}
In Lemma~\ref{l:case 1}, resp.\ Lemma~\ref{l:case 2}, we show that
Proposition~\ref{gafa} holds in Case~1, resp.\ Case~2. So once we have
proved these next two lemmas, we will also have proved
Proposition~\ref{gafa}. In Lemma~\ref{l:case 1}, we prove a little more.

\begin{lemma}\label{l:case 1}
  Suppose that, in addition to the hypotheses of
  Proposition~\ref{gafa}, 
\begin{itemize}
\item
$\SW_{L-\epsilon}$ is not a forest; and
\item
$|\T_t|$ and $|\,\T_t\,/\!\!\sim\!\!|$ are constant on $[0,L)$.
\end{itemize}
Then there is a simple class $u$ such that
$\ell(u|\bG_t)$ is bounded for all $t\in [0,L]$.
In particular, $d_\F(\bG_0,\bG_L)$ is bounded (see Lemma~\ref{bounded crossing 2}
and Remark~\ref{r:bounded crossing 2}).
\end{lemma}

\begin{proof}
By restricting to $[0,L-\epsilon]$, we may assume that $\SW_L$ is
defined and not a forest, and that $|\T_t|$ and
$|\,\T_t\,/\!\!\sim\!\!|$ are constant on $[0,L]$. In light of
Lemma~\ref{l:bounded everywhere}, we only need to prove there is a
simple $u$ such that $\k(u|\bG_{0})$ and $\ell(u|\bG_{L})$ are
bounded. Since $\SW_L$ is not a forest, by Lemma~\ref{sublemma} the
same is true at every $t$. Choose distinct illegal turns $\tau_{L,i},
\tau_{L,j}$ in $z_L$ that are equivalent in $\T_L$ so that the number
of illegal turns between in the good portion of $z_L$ is smaller than
the number of equivalence classes. (This is possible because embedded
circles $\SW_L$ have more than one edge.)  Let $[p_{t,i},p_{t,j}]$
denote the resulting segment in the good portion of $z_t$ and let
$\sigma_t$ be the loop obtained by closing up our segment, i.e.\ by
identifying $p_{t,i}$ and $p_{t,j}$. We refer to $\sigma_t$ as a {\it
  monogon} because it is immersed in $\bG_t$ except possibly at the
point $p_{t,i}=p_{t,j}$. In particular, the conjugacy class $w(t)$ in
$\FF$ represented by $\sigma_t$ is nontrivial. Of course, $w(t)$ is
also represented by the immersed circle $w(t)|\bG_t$ which is obtained
by tightening $\sigma_t$. By construction and Lemma~\ref{l:|classes|},
$\k(w(t)|\bG_t)\le|\,\T_t\,/\!\!\sim\!\!|<\m_t$. In particular,
$w(L)|\bG_L$ has bounded length. We claim that $w(0)=w(L)$. Once this
claim is established, the last sentence of Lemma~\ref{monogon} (with
$m$ equal to the constant $|\,\mathcal T_t\,/\!\!\sim\!\!|$) completes
the proof of this lemma.

To prove our claim, we must show that $\sigma_{0}$ and $\sigma_L$
determine the same conjugacy class. In folding $z_{0}$ to $z_L$,
maximal arcs in the directions of $\tau_{0,k}$ are identified in
$\bG_L$, i.e.\ they have the same image which is an immersed arc
$\alpha_{k}$ in $\bG_L$. If we tighten the image of
$[p_{0,i},p_{0,j}]$ in $\bG_L$, the result is the image of
$[p_{L,i},p_{L,j}]$ extended at its ends by $\alpha_{i}$ and
$\alpha_{j}$. The claim follows since
$\alpha_{i}=\alpha_{j}$. Indeed, $\tau_{0,i}$ and $\tau_{0,j}$ are
equivalent, and so in the same subgate, and so in the same gate. We
see $\alpha_i=\alpha_j$ at least if $L$ is small enough and that this
condition persists.
\end{proof}

\begin{lemma}\label{l:case 2}
  Suppose that, in addition to the hypotheses of
  Proposition~\ref{gafa}, 
\begin{itemize}
\item
$\SW_{L-\epsilon}$ is a forest; and
\item
$|\T_t|$ and $|\,\T_t\,/\!\!\sim\!\!|$ are constant on $[0,L)$.
\end{itemize} Then
  $d_\F(\bG_0,\bG_L)$ is bounded.
\end{lemma}

\begin{proof} As in the proof of Lemma~\ref{l:case 1}, we restrict to $[0,L-\epsilon]$ and so assume that $SW_L$ is a forest and that $|\T_t|$ and $|\,\T_t\,/\!\!\sim\!\!|$ are constant on $[0,L]$. In particular, $s:=|\T_L|=|\,\T_L\,/\!\!\sim\!\!|\le \m_t$ and $\SW_t$ is a forest for all
  $t\in [0,L]$ (Lemma~\ref{l:|classes|}).

First assume that there are two occurrences of the same illegal turn
in the consecutive collection at time $L$ that are separated by $<s-1$
illegal turns, i.e.\ there are $\tau_{L,i}$ and $\tau_{L,j}$ with
$\bar\tau_{L,i}=\bar\tau_{L,j}$ and $0<j-i<s-1$. Closing up gives a
curve with $<s$ illegal turns, so again the conclusion follows by
arguing as in Case~1.

So from now on we assume that this does not happen, i.e.\ all $s$
illegal turns occur repeatedly in a cyclic order in the consecutive
collection along $z_t$. If it so happens that one of these illegal
turns at $t=0$ involves an edge $e$ which is nonseparating and has
length a definite fraction $p$ of $injrad(\bG_0)$, we argue using
Lemma~\ref{monogon2} as follows. Consider the loop in $\bG_0$ obtained
by closing up the segment that starts with $e$ and ends at the next
occurrence of the same illegal turn. (Such a segment exists since
there are $M>2\hm$ illegal turns in our collection.) If the edge
following this segment is $e$, we can appeal to Lemma~\ref{monogon2}
to deduce the conclusion of the lemma (because $z$ is simple). If the
edge following the segment is not $e$, then the last edge $\bar e$ of
the segment is $e$ with the opposite orientation and closing up forces
cancellation. If the tightened loop has length $<2$, it is simple and
its image in $\bG_L$ is bounded, so the conclusion follows (cf.\ the
first sentence of the proof of Lemma~\ref{monogon}). If the tightened
loop has length $\geq 2$, then the original segment from $e$ to $\bar
e$ has length $\geq 2+2p~injrad(\bG_0)$ and the same argument as in
Case~4 of Lemma~\ref{monogon} (using
Lemma~\ref{derivative}(\ref{i:segment}) instead of
Lemma~\ref{derivative}(\ref{i:loop}); see also proof of
Lemma~\ref{monogon2}) shows that for $\ell(v|\bG_0)=injrad(\bG_0)$ we
must have $\ell(v|\bG_t)$ bounded.

To summarize, the conclusion of the lemma holds whenever the following
condition is satisfied at $t=0$:

\begin{itemize}
\item[$(\star)$]
There is a nonseparating edge $e$ in $\bG_t$ such that:
\begin{enumerate}[(i)]
\item
$e$ is in the good portion of $z_t$;
\item
$\ell_{\bG_t}(e)\ge p~injrad(\bG_t)$; and
\item
$e$ is involved in some turn in our collection $\T_t$.
\end{enumerate}
\end{itemize}
So now all that remains is to reduce to the case where $(\star)$ is
satisfied at $t=0$.

As a warmup, first consider the case where $\m(\bG_t)=s$ for all
$t$. In particular, every direction that is involved in an illegal turn of
$\bG_t$ is involved in an illegal turn of our collection
(Lemma~\ref{l:|classes|}). Let $\beta\in [0,L]$ be the first time that
a nonseparating edge of length $\geq \frac 1{3n-3}injrad(\bG_\beta)$
is involved in an illegal turn of $\bG_\beta$. If there is no such
$\beta$ then we claim that the conclusion of the lemma holds. Indeed,
if $v$ is a conjugacy class with $\ell(v_0)=injrad(\bG_0)$ where
$v_0:=v|\bG_0$ then there is a nonseparating edge $e_0$ of length
$\geq \frac 1{3n-3}injrad(\bG_0)$ in $v_0$. By Lemma~\ref{l:good
  edges}, there is a corresponding edge $e_t$ for small $t$. Note
that $$\ell(e_t)=\ell(e_0)e^t\ge\frac{\ell(v_0)e^t}{3n-3}\ge\frac{\ell(v_t)}{3n-3}\ge\frac
1{3n-3}injrad(\bG_t)$$ and so $e_t$ is not involved in any illegal
turns. In fact, we see that there can be no first time where
$\ell(e_t)<\frac 1{3n-3}injrad(\bG_t)$. We have $\ell(v_t)\le
e^t\ell(v_0)\le (3n-3){e^t}\ell(e_0)=(3n-3)\ell(e_t)<3n-3$ is bounded.

If there is such a $\beta$, then for the same reason the conclusion
holds for $\bG_t$, $t\in [0,\beta]$. Also, the conclusion holds for
$\bG_t$, $t\in [\beta,L]$ since ($\star$) holds at $t=\beta$.

It remains to consider the case when $\m_t$ is perhaps sometimes $>s$.
Let $\bG'_t$, $t\in [0,L']$ be the path in $\X$ starting at $\bG_0$
obtained as in \ref{slope 1 folding}.C by folding only the $s$ illegal
turns in our collection $\T_t$ (and then normalizing and
reparametrizing by arclength). Note that by the Unfolding Principle
each illegal turn $\tau_{0,i}$ our collection at time 0 comes with a
pair of legal paths that get identified in the folding process.
Folding only these turns amounts to identifying these paths. There are
induced morphisms $\bG_0\to e^{-L'}\bG'_{L'}\to e^{-L}\bG_L$. (In
particular, $\bG'_t$ is a folding path.) We can scale the second of
these morphisms to obtain $\phi:\bG'_{L'}\to \bG_L$. By the special
case $\m(\bG'_t)=s$, there are simple classes $v$ and $u$ and a
partition of $[0,L']$ into two subintervals such that $\ell(v|\bG'_t)$
is bounded on the first and $\ell(u|\bG'_t)$ is bounded on the second.
Note that $\phi$ is an immersion on the segment spanned by our
collection of turns, i.e.\ this segment is now legal in the train
track structure induced by $\phi$. Let $x$ be the conjugacy class
obtained by identifying two consecutive occurrences of the same
oriented element of our collection (possible since $M>2\hm$). The
assumption in the paragraph after the statement of
Proposition~\ref{gafa} guarantees that the definition of $x$ is
independent of $t$. By the construction, $x|\bG'_{L'}$ is legal with respect to
$\bG'_{L'}\to\bG_L$. Also 
$\ell(x|\bG_L)$ is bounded by the hypothesis of
Proposition~\ref{gafa}. If 
$\ell(x|\bG'_{L'})<2$, then $x$ is simple and
the length of $x$ along the path from $\bG'_{L'}$ to $\bG_L$ is
bounded. Hence, the conclusion of the lemma holds. If
$\ell(x|\bG'_{L'})\ge 2$, then $d_\X(\bG'_{L'},\bG_L)=L-L'$ is
bounded. Indeed, since
$$\ell(x|\bG_L)=\ell(x|\bG'_{L'})e^{L-L'}\ge 2e^{L-L'}$$ is bounded, so is $L-L'$.
In particular, the length of the simple class $u$ along the path from
$\bG'_{L'}$ to $\bG_L$ is bounded by $\ell(u|\bG'_{L'})e^{L-L'}$ and
again the conclusion of the lemma holds. 
\end{proof}

We have completed the proof of Proposition~\ref{gafa}.\qed

\begin{definition}\label{d:illegal loop}
An immersed path or a loop in a metric graph $\gG$ equipped with a
train track structure is {\it illegal} if it does not contain a legal
segment of length 3 (in the metric on $\gG$).
\end{definition}

\begin{lemma}\label{unfolding}
  Let $z$ be a simple class and $\bG_t$, $t\in [0,L],$ a folding path in
  $\X$ parametrized by arclength.  Assume $z_t$ is illegal for all
  $t$. Then either $\ell(z_L)<\ell(z_0)/2$ or $d_\F(\bG_0,\bG_L)$ is
  bounded.
\end{lemma}

\begin{proof}
  There are two cases. First suppose that the average distance between
  consecutive illegal turns in $z_0$ is $\geq 1/\hm$.  Then, by
  Proposition~\ref{gafa}(surviving illegal turns), after a bounded
  progress in $\F$ the loop $z$ must lose at least $1/M$ of its
  illegal turns. Repeating this a bounded number of times, we see
  that, after bounded progress in $\F$, $\k_t\le \k_0/6\hm$. Thus
  either the length of $z_t$ is less than $1/2$ of the length of $z_0$
  or the average distance between illegal turns at time $t$ is $\geq
  3$, so there is a legal segment.

  Now suppose the average distance between illegal turns in $z_0$ is
  $<1/\hm$. By Lemma~\ref{derivative},
$$\ell'(z_0)=\ell(z_0)-2\frac{\k_0}{\m_0}$$
We are assuming that
$\k_0/\m_0>\ell(z_0)$ so the above derivative is
$<-\ell(z_0)$. 
Thus in
this case the length of $z$ decreases exponentially, until either half
the length is lost after a bounded distance in $\X$, or the average
distance between illegal turns becomes $\geq 1/\hm$, when the above
argument finishes the proof.
\end{proof}

\begin{remark}
One source of asymmetry between legality and illegality is that a long
legal segment gets predictably longer under folding, while a long
illegal segment may not get longer under unfolding. For example, take
a surface relation inside a subgraph where folding amounts to an axis
of a surface automorphism. But the lemma above implies that an illegal
segment inside a loop representing a simple class will get predictably
longer under unfolding after definite progress in $\F$.
\end{remark}

\section{Projection to a folding path}\label{s:projection}
We thank Michael Handel for pointing out the technique for proving the
following lemma (see \cite[Proposition 8.1]{michael-lee}). 

\begin{lemma}\label{michael2}
  Let $\bG\in\X$ be a metric graph with a train track structure and
  $A<\FF$ a free factor. 
Suppose $A|\bG$ satisfies:
\begin{itemize}
\item there is an illegal loop $a$ in $A|\bG$ (see
  Definition~\ref{d:illegal loop}),
\item there is an immersed legal segment in $A|\bG$ of length
  $3(2n-1)$, $n=\rank(\FF)$.
\end{itemize}
Then $d_\F(A,\bG)$ is bounded.
\end{lemma}

\begin{proof}
If the injectivity radius of $A|\bG$ is $\leq 3(2n-1)$ the conclusion
follows from Corollary~\ref{lucas}.

Choose a complementary free factor $B$ to $A$ and add a wedge of
circles to $A|\bG$ representing $B$ to get a graph $H$. Extend
$A|\bG\to \bG$ to a homotopy equivalence (difference of markings)
$H\to\bG$, which is an immersion on each 1-cell. Pull back the metric
to $H$ and consider the folding path in $\hX$ induced by $H\to\bG$.
Let $H'$ be the first graph on this folding path with injectivity
radius $3(2n-1)$. Now give $H'$ the pullback illegal turn structure
via $H'\to\bG$. Since $A|\bG\to H'$ is an immersion, the interior of
the legal segment in $A|\gG$ embeds in $H'$. Since $H'$ has at most
$2n-2$ topological vertices, the interior of the legal segment meets at
most $2n-2$ topological vertices. So there is a legal segment of length
$3$ inside one of the topological edges of $H'$. Thus $a|H'$ does not
cross this topological edge and hence $d_\F(a,H')\leq 1+4$ (1 bounds the
distance between $\ff a$ and the free factor given by the subgraph of
$H'$ traversed by $a$ and 4 coming from Lemma~\ref{subgraph
  distance}).  Since $d_\F(a,A)\leq 1$ and $d_\F(H',\gG)$ is bounded
by Corollary~\ref{lucas} applied to a shortest loop in $H'$, the
statement follows.
\end{proof}

The following is a slight generalization.

\begin{lemma}\label{michael3}
  Let $\gG\in\X$ be a metric graph with a train track structure and
  $A<\FF$ a free factor. 
Suppose $A|\gG$ satisfies:
\begin{itemize}
\item there is a loop $a$ in $A|\gG$ with
the maximal number of pairwise disjoint legal segments of length 3
bounded by $N$, 
\item there is an immersed legal segment in $A|\gG$ of length
  $3(2n-1)$, $n=\rank(\FF)$.
\end{itemize}
Then $d_\F(A,\gG)$ is bounded as a function of $N$.
\end{lemma}

\begin{proof}
  The proof is similar. With $H'$ defined as in the proof of
  Lemma~\ref{michael2}, there is an edge in $H'$ which is crossed by
  $a|H'$ at most $N$ times. Hence, by Lemma~\ref{bounded crossing 2},
  $d_\F(H',a)\le 6N+13.$
\end{proof}

For the rest of this section, let $\bG_t$, $t\in [0,L],$ be a folding
path in $\X$ parametrized by arclength and let $A$ be the conjugacy class
of a proper free factor. A folding path has a natural orientation given by the
parametrization. We will think of this orientation as going left to
right.  

\label{p:I} Set
$$I:=(18\hm(3n-3)+6)(2n-1)$$ 
where recall that $\hm$ is the maximal possible number of illegal
turns in any $\gG\in\X$ (so $\hm$ is some linear function of the
rank). The number $I$ comes from Proposition~\ref{left-right}(legal
and illegal) which will say that if $A|\bG_0$ has a long (i.e.\ of
length $>3$) legal segment and $A|\bG_L$ has a long (i.e.\ of length
$>I$) illegal segment then $d_\F(\bG_0,\bG_L)$ is bounded.  Recall
that a segment is illegal if it does not contain a legal subsegment of
length 3.  This motivates the following definitions.

\begin{definition}
  Denote by $\lt_{\bG_t}(A)$ (or just $\lt(A)$ if the path $\bG_t$ is
  understood) the number
$$\inf \{t\in [0,L]\mid A|\bG_t \mbox{ has
  an immersed legal segment of length } 3\}$$ 
The {\it left projection $\Lt_{\bG_t}(A)$ of $A$ to the path
$\bG_t$} is $\bG_{\lt(A)}$.

Denote by $\rt_{\bG_t}(A)$ (or just $\rt(A)$ if $\bG_t$ is understood)
the number $$\sup \{t\in [0,L]\mid A|\bG_t \mbox{ has an immersed
  illegal segment of length } I\}$$ The {\it right projection
  $\Rt_{\bG_t}(A)$ of $A$ to the path} is $\bG_{\rt(A)}$. If the above
sets are empty, we interpret $\inf$ as $L$ and $\sup$ as $0$.

For a simple class $a$, we define $\lt_{\bG_t}(a):=\lt_{\bG_t}(\ff
a)$, $\Lt_{\bG_t}(a):=\Lt_{\bG_t}(\ff a)$,
$\rt_{\bG_t}(a):=\rt_{\bG_t}(\ff a)$, and
$\Rt_{\bG_t}(a):=\Rt_{\bG_t}(\ff a)$. In all cases, we may suppress subscripts if
the path is understood. Note that the first set
displayed above is closed under the operation of increasing $t$.
Clearly, $\lt(A)\leq\rt(A)$.

We also generalize these definitions from free factors to marked
graphs in the obvious way. If $H\in\X$, $\lt(H):=\min\lt(\pi(H))$,
$\Lt(H):=\bG_{\lt(H)}$, $\rt(H):=\max\rt(\pi(H))$, and
$\Rt(H):=\bG_{\rt(H)}$.
\end{definition}

\begin{prop}\label{L2}
Suppose $B<A$ are free factors. Then:
\begin{itemize}
\item $\lt(A)\leq \lt(B)$, $\rt(A)\geq \rt(B)$; and 
\item either $d_\X({\Lt(A)},{\Lt(B)})=e^{\lt(B)-\lt(A)}$ is bounded or
  the distance in $\F$ between $A$ and every element of $\{\bG_t\mid t\in
  [\lt(A),\lt(B)]\}$ is bounded.
\end{itemize}
\end{prop}

\begin{proof}
The first bullet is clear since we are taking $\inf$ and $\sup$ over
smaller sets.
Denote by $\lt'(A)$ the first time along the folding path that
$A|\bG_t$ has a legal segment of length $3(2n-1)$. Then
$\lt(A)\leq\lt'(A)$, and $\lt'(A)-\lt(A)$ is bounded by Corollary~\ref{3}. If
$\lt(B)\leq \lt'(A)$ we are done, so suppose
$\lt'(A)<\lt(B)$. It follows from Lemma~\ref{michael2} that the
  set of $\bG_t$'s for $t\in [\lt'(A),\lt(B)]$ has a bounded
  projection in $\F$, and the projection is close to $A$.
\end{proof}

Given a constant $K>0$, we will say that a coarse path
$\gamma:[\alpha,\omega]\to\F$ is a {\it reparametrized quasi-geodesic}
if there is a subdivision $\alpha=t_0<t_1<\cdots<t_m=\omega$ such that
$diam_\F(\gamma([t_i,t_{i+1}]))\leq K$, $m\leq
d_\F(\gamma(\alpha),\gamma(\omega))$, and $|i-j|\leq
d_\F(\gamma(t_i),\gamma(t_j))+2$ for all $i,j$. In particular, a map
$[0,m]\to \F$ given by mapping $x\in [0,m]$ to an element of
$\gamma(t_{[x]})$ is a quasi-geodesic with constants depending only on
$K$.  A collection $\{\gamma_i\}_{i\in I}$ of reparametrized
quasi-geodesics is {\it uniform} if the $K$ appearing in the
definition of $\gamma_i$ is independent of $i$ and is, in fact, a
function of the rank $n$ of $\FF$ alone. A {\it coarse Lipschitz}
function $f:X\to Y$ between metric spaces is one that satisfies
$d_Y(f(x_1),f(x_2))\leq K~ d_X(x_1,x_2)+K$ for all $x_1,x_2\in X$.  A
function $f:X\to A\subseteq X$ is a {\it coarse retraction} if
$d(a,f(a))\leq K$ for all $a\in A$.  In all these cases, $f$ is
allowed to be multivalued with the bound of $K$ on the diameter of a
point image.

\begin{cor}\label{coarse Lipschitz}
For any folding path $\bG_t$ the projection
$$\F\to\pi(\bG_t)$$
$$A\mapsto \pi({\Lt(A)})$$ is a coarse Lipschitz retraction
with constants depending only on $\rank(\FF)$. Consequently, the
collection of paths $\{\pi(\bG_t)\}$ is a uniform collection of
\hyphenation{re-pa-ram-e-trized} reparametrized quasi-geodesics in
$\F$.
\end{cor}

\begin{proof}
  That the
  map is coarsely Lipschitz follows from Proposition~\ref{L2}. To
  prove that it is a coarse retraction, we need to argue that
  $d_\F(\Lt(\bG_{t_0}),\bG_{t_0})$ is bounded for $t_0\in [0,L]$. Let
  $a$ be the conjugacy class of a legal candidate in $\bG_{t_0}$. In
  particular, $\langle a\rangle$ is a rank one free factor,
  $\ell(a_{t_0})<2$, and $d_\F(a, \bG_{t_0})$ is bounded.  Since our
  map is coarsely Lipschitz, it is enough to show that
  $d_\F(\Lt(a),\bG_{t_0})$ is bounded. Let $t'$ be the smallest
  parameter such that $a_{t'}$ is legal. We have $\lt(a)\le t'\le
  t_0$. For $t\in [t',t_0]$, $a_t$ is legal and $\ell(a_{t_0})<2$, and
  so $\ell(a_t)<2$. This has two consequences. First,
  $d_\F(\bG_t,\bG_{t_0})$ is bounded (Lemma~\ref{bounded crossing 2}
  and Remark~\ref{r:bounded crossing 2}). Second, $\lt(a)=t'$ (since, for
  $t''<t'$, a legal segment of length 3 in $a_{t''}$ would force
  $\ell(a_{t'})\ge 3$). Hence $d_\F(\bG_{t'}=\Lt(a),\bG_{t_0})$ is
  bounded.

  The argument for the second part is from \cite{bo:hyp}.  Let $\bG_t$
  be a folding path so that $\pi(\bG_t)$ is a coarse path joining free
  factors $A$ and $B$. Choose a geodesic $C_i$, $i=0,\cdots,m$ of free
  factors joining $C_0=A$ and $C_m=B$ in $\F$. Consider the coarse
  projection $D_i$ of $C_i$ to $\pi(\bG_t)$. By Proposition~\ref{L2}
  the diameter of the segment bounded by $D_i$ and $D_{i+1}$ is
  uniformly bounded. Now the $D_i$'s may not occur monotonically along
  $\pi(\bG_t)$. To fix this, let $i_1<i_2<\cdots<i_k$ be the sequence
  defined inductively by $i_1=0$ and $i_{j+1}$ is the smallest index
  such that $D_{i_{j+1}}$ occurs after $D_{i_j}$ in the order on
  $\bG_t$ given by $t$. Then by construction the interval between
  $D_{i_j}$ and $D_{i_{j+1}}$ has uniformly bounded diameter and the
  number $k$ is bounded by $m=d_\F(A,B)$. Call a subdivision
  satisfying these properties a {\it admissible}. To ensure the last
  property $|i-j|\leq d_\F(\gamma(t_i),\gamma(t_j))+2$ take an
  admissible subdivision with minimal $k$.
\end{proof}

\begin{definition}\label{d:projection to Gt}
  Given $A\in\F$, $H\in\X$, and a folding path $\bG_t$, $\pi(\Lt(A))$
  is {\it the projection of $A$ to $\pi(\bG_t)$} and $\pi(\Lt(H))$ is
  {\it the projection of $H$ to $\pi(\bG_t)$}.
  \end{definition}

\begin{definition}
  For $\kappa>0, C>0$ we say a folding path $\bG_t$ makes {\it
    $(\kappa,C)$-definite progress in $\F$} if for any $D>0$ and
  $s<t$, $d_\X(\bG_s,\bG_t)>D\kappa+C$ implies $d_\F(\bG_s,\bG_t)>D$.
\end{definition} 

\begin{cor}
For any folding path $\bG_t$ the projection
$$\F-R(\bG_t)\to \{\bG_t\}$$
$$A\mapsto {\Lt(A)}$$
where $R(\bG_t)$ is the set of free factors at a certain bounded
distance from $\bG_t$ measured in $\F$, is coarsely Lipschitz (with
respect to the path metric in $\F-R(\bG_t)$).

Moreover, the projection is coarsely defined and coarsely Lipschitz on
all of $\F$ provided $\bG_t$ makes $(\kappa,C)$-definite progress in
$\F$ (with constants depending on $\kappa,C$).  \qed\end{cor}

\begin{lemma}\label{forward immersion}
  Let $\bG_t$, $t\in [0,L],$ be a folding path and $A$ a free factor.
  The length of any illegal segment contained in a topological edge of
  $A|\bG_L$ is less than
$$\frac 32 m\cdot edgelength(A|\bG_0)+6$$
where $m=\max\{m_t\mid t\in [0,L]\}$ and $edgelength(A|\bG_0)$ is the
maximal length of a topological edge in $A|\bG_0$.
\end{lemma}

\begin{proof}
  Fix an illegal segment of length $\ell_L$ in the interior of a
  topological edge of $A|\bG_L$. We will assume that the endpoints are
  illegal turns and argue
\begin{equation*}
\ell_L\leq \frac 32 m\cdot edgelength(A|\bG_0)
\tag{$\diamond$}\end{equation*}

\noindent After adding $<3$ on each end we recover any illegal path.
By the Unfolding Principle our path lifts to an illegal segment
bounded by illegal turns inside some topological edge of $A|\bG_t$,
whose length will be denoted $\ell_t$. In particular, $\ell_0\leq
edgelength(A|\bG_0)$. In order to obtain a contradiction, assume
($\diamond$) fails. Let $t_0$ be the first time the right derivative
of $\ell_t$ is nonnegative (if such $t_0$ does not exist then
$\ell_L\leq \ell_0$ and ($\diamond$) holds, contradiction). Thus
$\ell_{t_0}\leq \ell_0$ and the average length of a maximal legal
segment inside the path is $\geq 2/\m_{t_0}\geq 2/m$ by
Corollaries~\ref{c:length}(\ref{i:average}) and \ref{c:segments}.
Since $\ell_L\ge \frac 32m\ell_{t_0}$, the average length of a legal
segment of our path is guaranteed to be $\ge 3$, contradicting the
hypothesis that our segment is illegal. Thus ($\diamond$) holds and
the lemma follows.
\end{proof}

Recall that the number $I$ used in the next proposition was defined on
Page~\pageref{p:I}.

\begin{prop}[legal and illegal]\label{left-right}
  Let $\bG_t$, $t\in [0,L],$ be a folding path and $A$ a free factor.
  Assume that $A|\bG_0$ has a legal segment of length $3$, and that
  $A|\bG_L$ has an illegal segment of length $I$. Then
  $d_\F(\bG_0,\bG_L)$ is bounded.
\end{prop}

\begin{proof}
Since a legal segment of length 3 grows to a legal segment of length
$>12(3n-3)(2n-1)$ in bounded time (Corollary~\ref{3}), by replacing
$\bG_0$ with $\bG_{t}$ for a bounded $t$, we may assume that $A|\bG_0$
has a legal segment of length $12(3n-3)(2n-1)$.  In order to obtain a
contradiction, assume the distance $d_\F(\bG_0,\bG_L)$ is large. Let
$\tau\in [0,L]$ then be chosen so that $d_\F (\bG_0, \bG_\tau)$, $d_\F
(\bG_\tau , \bG_L )$ and $d_\F (\bG_\tau , A)$ are all large.

Wedge $A|\bG_\tau$ onto a rose representing a complementary free
factor to $A$ in order to obtain a graph $H'\in\hX$ and a difference
of markings morphism $H'\to\bG_\tau$ extending $A|\bG_\tau\to\bG_\tau$
and which is an isometric immersion on every edge. In particular,
$H'\to\bG_\tau$ induces a train track structure on $H'$. If $H'$ has
bounded injectivity radius then $d_\F (A, \bG_\tau )$ is bounded,
contradicting the choice of $\bG_\tau$. So suppose the injectivity
radius of $H'$ is large and fold $H'\to \bG_\tau$ until a graph $H''$
is reached which is the last time there is an edge $E''$ of length
$4$.  Cf.\ \cite[Proposition 8.1]{michael-lee}.

In particular, $vol(H'')\le 4(3n-3)$. We continue by folding with
speed 1 the subset of those illegal turns of $H''\to\bG_\tau$ that
don't involve $E''$. Since $H''\to\bG_\tau$ induces a train track
structure on $H''$, so does our subset. For small $t$, we obtain a
graph $H''_t$ where the image $E''_t$ of $E''$ (perhaps no longer
topological) still
has length 4 and a morphism $H''_t\to\bG_\tau$ inducing a train track
structure. We continue folding all illegal turns not involving $E''_t$
(as in \ref{slope 1 folding}.C) until we obtain morphism
$H\to\bG_\tau$ inducing a train track structure and isometrically
immersing both the image $E$ in $H$ of $E''$ and its complement. The
only illegal turns of $H\to\bG_\tau$ involve $E$.  In fact, since our
illegal turns give a train track structure, the only illegal turns
involve the topological edge containing $E$. We will now use $E$ for the
name of this topological edge. After folding from $H''$ to $H$, some
edge lengths may now be $>4$. But since $vol(H)\le vol(H'')$, edge
lengths in $H$ are at most $4(3n-3)$ (and $E$ still has length at
least 4).

We may assume that the complement of $E$ does not have a valence 1
vertex. Indeed, assuming otherwise, with respect to the train track
structure induced by $H\to\bG_\tau$ there are two possibilities for
the illegal turns (which recall must involve $E$), see
Figure~\ref{f:lollipop}.  In the left picture, the length of $E$ stays
constant under folding. We continue folding until the separating edge
folds in with $E$, and this is our new $H$.  The right picture is
impossible: $E$ is a ``monogon'' (of length at least 4) and the
folding towards $\bG_\tau$ stops before the loop degenerates. But this
means that $\bG_\tau$ has volume $>2$, a contradiction.

\begin{figure}[h]
\begin{center}
\includegraphics[scale=0.4]{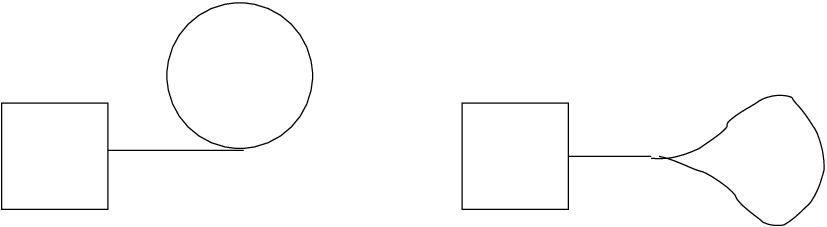}
\caption{Two possibilities when $E$ is a loop attached to a separating
edge. The square represents the remainder of the graph.}
\label{f:lollipop}
\end{center}
\end{figure}

We will also assume for concreteness that the complement of $E$ is
connected, and denote by $B$ the free factor determined by it. When
the complement is disconnected, there are two free factors determined
by the components. The changes are straightforward and left to the
reader.

We have morphisms $H\to\bG_\tau\to \bG_L$ and now also bring in the
pullback illegal turn structure via $H\to\bG_\tau$. To distinguish
between the two structures, terms like $p$-legal and $p$-illegal will
refer to this pullback, i.e.\ the one induced by $H\to\bG_L$. Terms
like $i$-legal and $i$-illegal will refer to the structure induced by
$H\to\bG_\tau$.  The same terminology will be applied to turns in
$\hat H_t$ (resp.\ $H_t$ and $K_t$) with respect to $\hat\psi_s:\hat
H_s\to\hG_s$ (resp.\ $H_t\to \bG_t$ and $K_t\to\bG_t$) constructed
below. Note that $H$ may have $p$-illegal turns in the interior of
topological edges.  (Consider that perhaps $H=H'$.) By construction,
$i$-illegal turns must involve $E$.

There are two cases.

{\bf Case 1.} $E$ contains a $p$-legal segment of length 3. As $G_\tau$
folds toward $\bG_L$, we will use the technique of \ref{slope 1
  folding}.C to fold $H$ and produce a new path (though usually not a
folding path) in $\X$. To describe this path, it is convenient to view
the folding path $\bG_t$ as in Proposition~\ref{slope 1 folding},
i.e.\ without rescaling and folding with speed 1. So, let
$\hG_{s(t)}$, $t\in [\tau,L],$ be the folding path $\hX$ induced by
the morphism $\bG_\tau\to e^{\tau-L}\bG_L$ with natural parameter $s$.

We claim that, for $s\in [s(\tau),s(L)]$, there is a path $\hat H_s$
in $\hX$ that starts at $H$ and satisfies:
\begin{enumerate}[(1)]
\item $\hat H_s=B|\hG_s\cup \hat E_s$, where $\hat E_s$ is a
  topological edge containing a $p$-legal segment of length at least
  $3\, vol(\hG_s)$,
\item the immersion $B|\hG_s\to\hG_s$ extends to a morphism
  (difference of markings) $\hat\psi_s:\hat H_s\to \hG_s$ inducing a train
  track structure on $\hat H_s$. In particular, $\hat \psi_s$ is an
  isometric immersion on $\hat E_s$,
\item 
$A|\hG_s\to \hG_s$ factors through $\hat\psi_s$. In particular,
  $A|\hG_s\to\hat H_s$ is an isometric immersion.
\end{enumerate}
By construction (1--3) hold for $\hat H_{s(\tau)}:=H$ and $\hat
E_{s(\tau)}:=E$. Following \ref{slope 1 folding}.C, assume $\hat H_s$
has been defined on a subinterval $J$ of $[s(\tau),s(L)]$ containing
$s(\tau)$. If $J=[s(\tau),s_0]$ with $s_0\not=s(L)$, then we can, for
small time $\epsilon>0$, fold all $p$-illegal turns of $\hat H_{s_0}$
at speed 1 (see Page~\pageref{p:folding}) thereby extending the path
to $[s(\tau), s_0+\epsilon]$. We see that (1--3) hold for $s\in
[s_0,s_0+\epsilon]$. Indeed, $\hat H_{s_0}\to\hG_{s}$ factors through
$\hat H_{s_0}\to\hat H_{s}$ and so $\hat H_s\to\hG_s$ is a morphism.
By Lemma~\ref{3}, $\hat H_s$ has a topological edge $\hat E_s$
containing a $p$-legal segment of length at least $3\,vol(\hG_s)$ and
whose complement has core representing $B$. Since $B|\hat H_{s_0}$,
$A|\hat H_{s_0}$, and the interior of $\hat E_{s_0}$ contain no
$i$-illegal turn, the same is true at $s$. There must be an
$i$-illegal turn of $\hat H_{s_0}$ involving both $\hat E_{s_0}$ and
an edge in $B|\hat H_{s_0}$ (or else $\hat H_s$ has a monogon as above
which has been ruled out).

We move to the case $J=[s(\tau),s_0)$. As in \ref{slope 1 folding}.C,
we may define a limit tree $\hat H_{s_0}\in \hX$. By Lemma~\ref{3},
$\hat H_{s_0}$ has a topological edge containing a $p$-legal segment
of length at least $3\,vol(\bG_{s_0})$ and whose complement has core
representing $B$. The limit of these morphisms is a morphism, so (2)
and (3) also hold. Finally, $\hat E_{s_0}$ can't be a loop connected
to $B|\hat H_{s_0}$ by a separating edge (or else the same would have
been true at smaller $s$).

Set $H_s:=\hat H_s/vol(\hG_s)$ and define the image of $\hat E_s$ in
$H_s$ to be $E_s$. Reverting to our original parametrization, we now
have our original path $\bG_t$, $t\in [\tau,L],$ in $\X$ and a new
path $H_t$, $t\in [\tau, L],$ in $\hX$ such that, for each $t$,
$H_t=B|H_t\cup E_t$, there is a morphism $\psi_t:H_t\to \bG_t$,
$A|\bG_t$ isometrically immerses in $H_t$, $B|\bG_t$ isometrically
embeds in $H_t$, and $E_t$ contains a $p$-legal segment of length at
least 3.

We need one more modification to control the length of $E_t$. Define
$K_t\in\hX$ as follows. If the length of $E_t$ is $\le 4$, $K_t:=H_t$.
If the length of $E_t$ in $H_t$ is $>4$, define $K_t$ to be the graph
obtained by folding $i$-illegal turns of $H_t\to\bG_t$ until the
length of $E_t$ is 4.  Since $A|\bG_t$ and $B|\bG_t$ are immersed in
$H_t$, the only effect is to fold pieces of the end of $E_t$ into
$B|\bG_t$.  In particular, the analogues of (1--3) hold for $K_t$,
except it is possible that $E_t$ no longer has a $p$-legal segment of
length 3.

By keeping in mind that the length of $E_t$ in $K_t$ is at most 4 and
applying Lemma~\ref{forward immersion} to $\bG_t$, $t\in [\tau, L],$
and $B$, the length of any $p$-illegal segment contained in a
topological edge of $K_t$ is bounded by
$$\frac 32\hm\cdot
edgelength(B|\bG_\tau)+6\le \frac 32\hm\cdot 3\,
edgelength(K_\tau)+6\le 18\hm(3n-3)+6$$ Since the number of
topological vertices of $K_t$ is $\le 2n-2$, a $p$-illegal segment in
$K_t$ of length $I=(18\hm(3n-3)+6)(2n-1)$ meets some topological
vertex of $K_t$ twice.  We see that, for $t\in [\tau,L]$, either the
$injrad(K_t)$ is bounded by $I$ (in which case $d_\F(\bG_t, B)$, and
hence $d_\F(\bG_t, G_\tau)$, is bounded), or there are no $p$-illegal
segments in $K_t$ of length $I$ and hence the same holds for
$A|\bG_t$. Applying this to $t=L$ (by hypothesis there is
  a $p$-illegal segment of length $I$ in $A|\bG_L$) 
we see that
$d_\F(\bG_\tau,\bG_\L)$ is bounded, contradicting the choice of
$\bG_\tau$.

{\bf Case 2.} $E$ doesn't contain a $p$-legal segment of length 3. In
particular, the interior of $E$ crosses a $p$-illegal turn. Let
$\hG_{s(t)}$, $t\in [0,\tau],$ be the folding path in $\hX$ giving
rise to $\bG_t$ and ending at $\bG_\tau$. We will produce a path $\hat
H_{s(t)}$, $t\in [0,\tau]$ (usually not a folding path) in $\hX$
ending at $H$ and, for each $s\in [s(0),s(\tau)]$, satisfying:
\begin{enumerate}
\item $\hat H_s=B|\hG_s\cup \hat E_s$, where $\hat E_s$ is a single
  edge,
\item the immersion $B|\hG_s\to\hG_s$ extends to a morphism (difference
  of markings)
  $\hat\psi_s:\hat H_s\to\hG_s$ which is an isometric immersion on $\hat E_s$,
\item $A|\hG_s\to\hG_s$ factors through $\hat\psi_s$. In particular,
  $A|\hG_s\to\hat H_s$ is an isometric immersion.
\end{enumerate}

Note that (1--3) hold for $\hat H_{s(\tau)}=H$. Let
$0=s_0<s_1<\dots<s_N=s(\tau)$ be a partition of $[0,s(\tau)]$ so that the
restriction of $\hG_t$ to each $[s_i,s_{i+1}]$ is given by folding a
  gadget.  Assume $\hat H_s$ has been defined for $s\in [s_{i},s_N]$
  satisfying (1--3).  We now work to extend $\hat H_s$ over
  $[s_{i-1},s_N]$ still satisfying (1--3).

  We first define $\hat H_s$, $s\in (s_{i-1},s_i]$, via the following
  local operations. $\hat H_s$ is defined as $B|\hG_s$ with an edge
  $\hat E_s$ attached, and we specify the attaching points. Consider
  first the case that a direction $e$ of an end of $\hat E_{s_i}$
  forms a $i$-illegal turn with a direction $b$ in $B|\hG_{s_i}$ (such
  a direction is then unique).  Intuitively, as $s$ decreases,
  $B|\hG_s$ unfolds and we choose to fold $b$ and $e$ with speed 1. A
  more elaborate description follows.

Let $\hat\phi=\hat\phi_{ss_i}:\hG_{s}\to \hG_{s_i}$ be the folding
morphism. It induces a morphism $\hat\phi_B:B|\hG_{s}\to B|\hG_{s_i}$.
Let $\epsilon=s_i-s$ and let $\tilde N$ be the $\epsilon$-neighborhood
in $B|\hat H_{s_i}=B|\hG_{s_i}$ of the vertex $v$ of $e$. $N(B)$ is a
subset of the $\epsilon$-neighborhood $N$ of $v$ in $\hG_{s_i}$.
$N_{\epsilon}$ denotes the preimage in $\hG_{s}$ of $N$ and $\tilde
N_{\epsilon}$ is the preimage of $\tilde N$ in $B|\hG_{s}$.  Using the
language of widgets, we attach the end
of $\hat E_{s}$ corresponding to $e$ to the base of $\tilde b^*$ in
$\tilde N_\epsilon$.  (To recall notation, see
Figure~\ref{f:notation}.)  To define $\hat E_{s}$ delete a length
$\epsilon$ segment from the end of $\hat E_{s_i}$.
$B|\hG_{s}\to\hG_{s}$ now extends to a morphism $\hat\psi_{s}:\hat
H_{s}\to\hG_{s}$.  Figure~\ref{f:case2.1} illustrates the diagram
$$
\begindc{\commdiag}[40]
\obj(1,2)[12]{$\tilde N_{\epsilon}$}
\obj(2,1)[21]{$N$}
\obj(2,2)[22]{$\tilde N$}
\obj(1,1)[11]{$N_{\epsilon}$}
\mor{11}{21}{}
\mor{12}{22}{}[\atleft,\solidarrow]
\mor{22}{21}{}
\mor{12}{11}{}[\atright,\solidarrow]
\enddc
$$
with ends of $\hat E_{s}$ and $\hat E_{s_i}$ attached to (resp.) $\tilde N_{\epsilon}$ and $\tilde N$.

\begin{figure}[h]
\begin{center}
\includegraphics[scale=0.3]{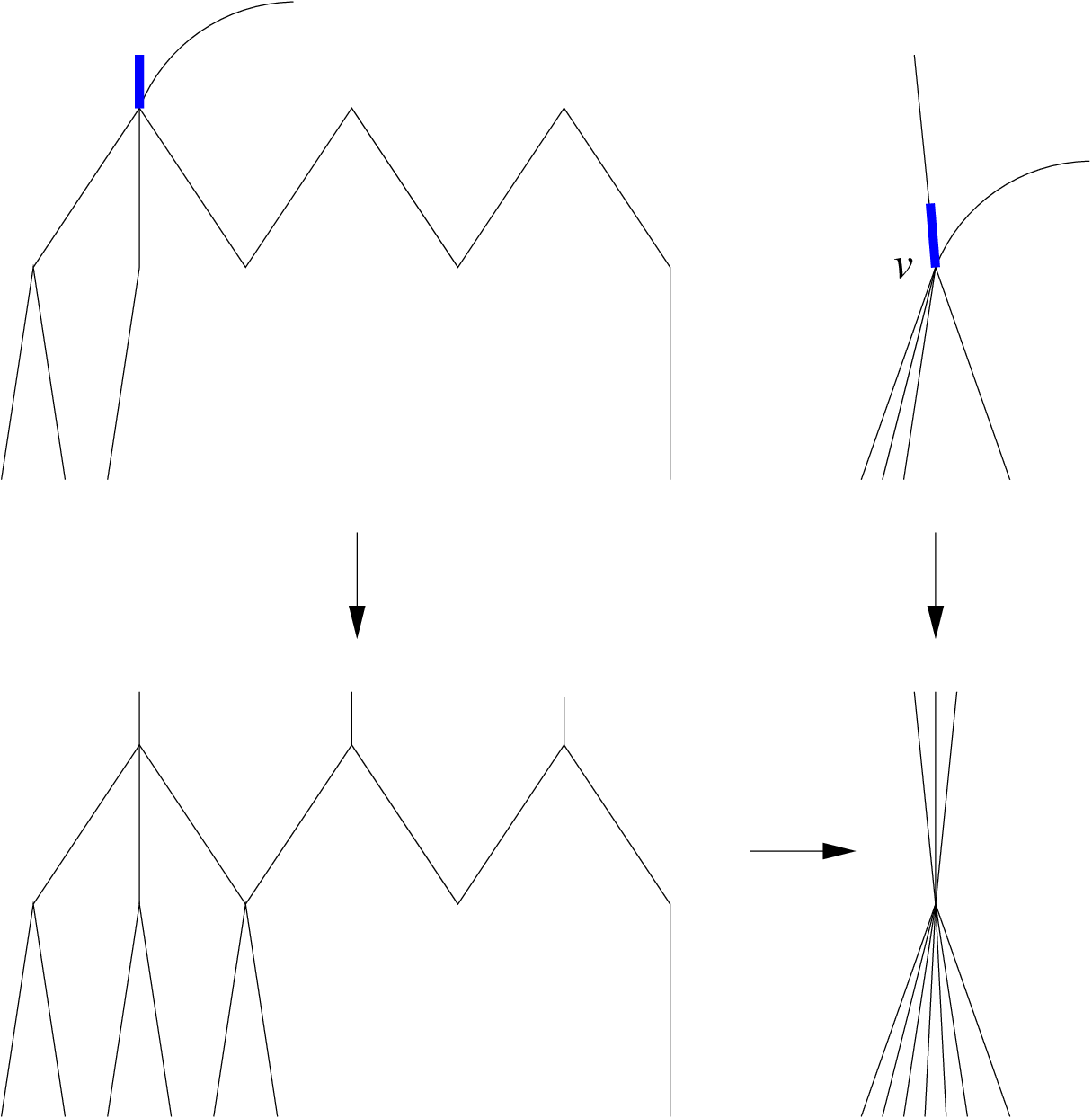}
\caption{The thickened segments represent $\tilde b^*$ and $b$. The curved segments represent ends of $\tilde E_{s}$ and $\tilde E_{s_i}$.}
\label{f:case2.1}
\end{center}
\end{figure}

Now suppose that the direction $e$ does not form an $i$-illegal turn
with any direction in $B|\hG_{s_i}$. Then there is a natural way to
construct the attaching point in $\hG_s$ by watching $\hG_{s_i}$
unfold to $\hG_s$. In terms of widgets, we attach the end of $\tilde
E_{s}$ corresponding to $e$ to the point in $B|\hG_{s}$ closest to
$\tilde e^*$.  Figure~\ref{f:case2.2} illustrates the diagram
$$
\begindc{\commdiag}[40]
\obj(1,2)[12]{$\tilde N_\epsilon$}
\obj(2,1)[21]{$N$}
\obj(2,2)[22]{$\tilde N$}
\obj(1,1)[11]{$N_{\epsilon}$}
\mor{11}{21}{}
\mor{12}{22}{}[\atleft,\solidarrow]
\mor{22}{21}{}
\mor{12}{11}{}[\atright,\solidarrow]
\enddc
$$
with $\hat E_{s}$ and $\hat E_{s_i}$ attached.
\begin{figure}[h]
\begin{center}
\includegraphics[scale=0.3]{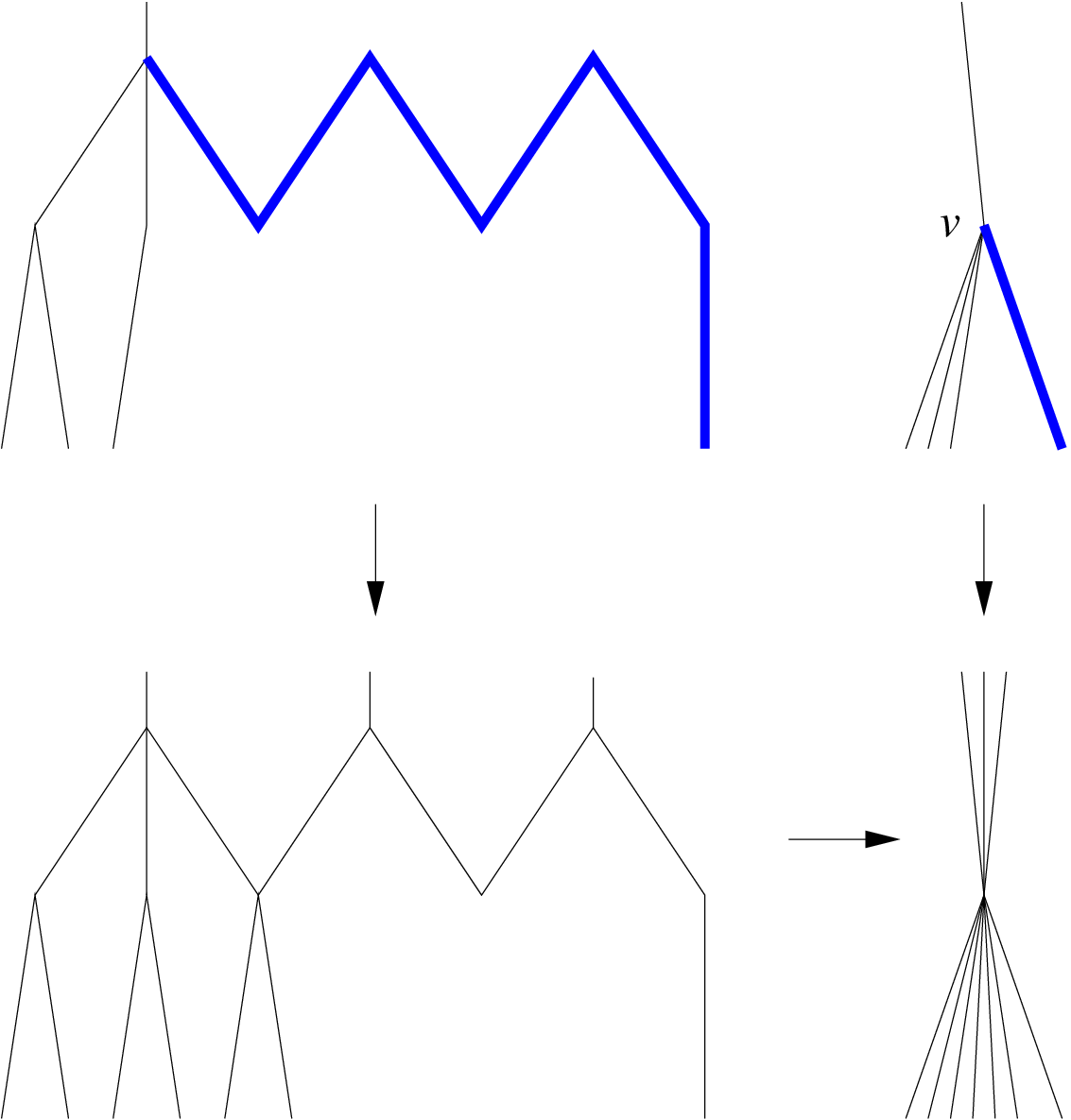}
\caption{To decide where to attach $\hat E_s$ mimic what happens in
  $\hG_s$. The thickened lines represent $\hat E_{s}$ and $\hat E_{s_i}$.}
\label{f:case2.2}
\end{center}
\end{figure}

There is a unique homotopy class of paths in $\hG_{s}$
connecting the images of attaching points. The map
$\hat\psi_{s}$ is defined so that it isometrically immerses
$\hat E_{s}$ to the immersed path in the above homotopy class.

Now suppose $\hat H_s$ is defined for $s\in (s_{i-1},s(\tau)]$ and we
want to define $\hat H_{s_{i-1}}$. Let $\sigma$ be a conjugacy class
in $\FF$.  By construction, for $s\in (s_{i-1},s_i)$,
$\ell(\sigma|\bG_s)=\ell(\sigma|\bG_{s_{i}})+(2x-y)(s_i-s)$ where $x$
is the number of $p$-illegal turns crossed by $\sigma|\bG_s$ that are
$i$-legal and $y$ is the number of $i$-illegal turns crossed. Note
that $x$ and $y$ are constant on $(s_{i-1},s_i)$. In particular,
$\lim_{s\to s_{i-1}^+}\ell(\sigma|\hat H_s)$ exists, thereby defining
$\hat H_{s_{i-1}}$. That $\hat H_{s_{i-1}}$ is in $\hX$ follows from
the existence of the limiting morphism $\hat
H_{s_{i-1}}\to\hG_{s_{i-1}}$.  Finally, note that $\hat E_s$ doesn't
degenerate to a point in this limit. Indeed, $\hat E_{s(\tau)}$
crosses a $p$-illegal turn and this property persists by the Unfolding
Principle (as $s$ decreases, $p$-illegal turns in $\hat E_s$ move away
from the endpoints of $\hat E_s$ which balances any loss at the ends
of $\hat E_s$ due to $i$-illegal turns).

Set $H_s=\hat H_s/vol(\hG_s)$ and revert to our original
parametrization. We now have a path $H_t$, $t\in [0,\tau]$. Define $K_t$
exactly as before, i.e.\ if $\ell(E_t)>4$, then fold $i$-illegal turns
of $H_t\to\bG_t$ until $E_t$ has length 4.

A $p$-legal segment of length $>3\cdot 4(3n-3)$ interior to an edge of
$K_s$ would produce an edge in $B|\bG_s$, hence also in $B|\bG_\tau$,
of length $>3\cdot 4(3n-3)$. We would then have an edge of length
$>4(3n-3)$ in $K_\tau$, contradiction. A $p$-legal segment in $K_s$ of
length $>12(3n-3)(2n-1)$ is then forced meet least $2n-1$ topological
vertices which implies $injrad(K_s)\le 12(3n-3)(2n-1)$. By assumption,
$A|\bG_0$, hence also $K_0$, has a $p$-legal segment of length
$>12(3n-3)(2n-1)$.  Arguing as at the end of Case~1, we get the
contradiction that $d_\X(\bG_0,\bG_\tau)$ is bounded.
\end{proof}

\begin{cor}\label{left-right2}
The image in $\F$ of the interval
$[\Lt_{\bG_t}(A),\Rt_{\bG_t}(A)]$ has bounded diameter.
\end{cor}

\begin{proof}
  The endpoints have bounded $\F$-distance by
  Proposition~\ref{left-right}(legal and illegal), and therefore the
  whole interval projects to a bounded set by Corollary~\ref{coarse
    Lipschitz}.
\end{proof}

Corollary~\ref{left-right2} says that the projection of $A\in\F$ to
$\pi(\bG_t)$ is bounded distance from
$\pi([\Lt_{\bG_t}(A)),\Rt_{\bG_t}(A)])$. (Recall
Definition~\ref{d:projection to Gt}.)  We will, in
Lemma~\ref{criterion}, see a way to estimate where the this projection
lies.  To prove Lemma~\ref{criterion}, we first need a simple lemma
about canceling paths in a graph and then a general fact that in a
different form appears in \cite[Proposition~5.10 and claim on p.~2218]{yael}.

If $X$ is an edge-path in a graph, then $[X]$ denotes the path
obtained from $X$ by {\it tightening}, i.e.\ $[X]$ is the immersed
edge-path homotopic rel endpoints to $X$. If the endpoints of $X$
coincide and the resulting closed path is not null-homotopic, then
$[[X]]$ denotes the loop obtained from $X$ by {\it tightening}, i.e.\
$[[X]]$ is the immersed circle freely homotopic to $X$.

\begin{lemma}\label{l:simple}
  Let $V$ be an immersed edge path in a graph $G$. Suppose $V$
  represents an immersed circle, i.e.\ $V$ begins and ends at a vertex
  $P$ and has distinct initial and terminal directions. 
Let $W$ be a
  nontrivial initial edge subpath of $V$ ending at a vertex $Q$ and
  let $V'=[W^{-1}VW]$ (so that $V'$ also represents an immersed
  circle).  Also, let $X$ and $Y$ be immersed edge paths in $G$ starting at $P$ and $Q$ respectively. Suppose that $WY$ is immersed in $G$. Then:
\begin{enumerate}[(1)]
\item\label{i:3 segments} the maximal common initial subpath of $X$, $WY$, and $WV'Y$ has
  the form $V^NW'$ for some $N\ge 0$ and some initial subpath $W'$ of $V$.
\item \label{i:2 segments} The maximal initial subpath of $X$ and $VX$ has the form $V^NW'$
  for some $N\ge 0$ and some initial subpath $W'$ of $V$.
\end{enumerate}
\end{lemma}

\noindent The proof of Lemma~\ref{l:simple} is left to the reader. Note that
(2) follows from (1) applied to $X$, $VX$, and $V^2X$.

If $a$ and $b$ are conjugacy classes in $\FF$ and $\bG\in\X$ then by
{\it $a|\bG$ and $b|\bG$ share a segment of length $K$} we mean that
there is an isometric immersion $[0,K]\to \bG$ that lifts to both
$a|\bG$ and $b|\bG$.

\begin{prop}[closing up to a simple class]\label{general}
There is a constant $C_n>0$ so that the following holds.  Suppose
$G,H\in\X$, $z$ is any class (not necessarily simple), and $K>0$. If
$\ell(z|G)\geq C_nK\ell(z|H)$ then there is a class $u$ in $\FF$ such
that:
\begin{itemize}
\item $\ell(u|H)<2$;
\item $d_\F(H,u)$ is bounded; and 
\item $u|G$ and $z|G$ share a segment of length $K$.
\end{itemize}
In particular, $u$ is simple.
\end{prop}

If $\ell(z|H)<2$ then $u=z$ {\it works}, i.e.\ $z$ satisfies the
conclusions of the lemma. So, in the proof of
Proposition~\ref{general} we assume $\ell(z|H)\ge 2$.

Fix a map $\phi:H\to G$ so that each edge is immersed (or collapsed)
and each vertex has at least two gates (e.g. first change the metric
on $H$ as in Proposition \ref{rescaling} so that the tension graph is
all of $H$, but in the rest of the proof we use the original metric).

In the proof we will not keep track of exact constants, but will talk
about ``long segments in common with $z|G$''. For example, suppose an
edge path $A$ in $H$ is the concatenation $A=BC$ of two sub-edge
paths, and we look at $[\phi(A)]$, which is the
tightening of the composition $[\phi(B)][\phi(C)]$. If $[\phi(A)]$
contains a long segment in common with $z|G$, then so does $[\phi(B)]$
or $[\phi(C)]$ (or both), but ``long'' in the conclusion means about a
half of ``long'' in the assumption. The number of times this argument
takes place will be bounded, and at the end the length can be taken as
large as we want by choosing the original length (i.e.\ $C_n$) large.

Represent $z|H$ as a composition of $\sim\ell(z|H)$ paths $z_i$ where
each $z_i$ is either an edge, or a combinatorially long (but of length
$\leq 1$) path contained in the thin subgraph (union of immersed loops
of small length). Thus the loop $z|G$ is obtained by tightening the
composition of the paths $[\phi(z_i)]$. In the process of tightening,
everything must cancel except for a (possibly degenerate) segment
$\sigma_i\subset [\phi(z_i)]$ in each path, and at least one
$\sigma_i$ must have length $\geq\sim C_nK$. So we conclude that, for
some $z_i$, $[\phi(z_i)]$ contains a long segment in common with
$z|G$.  There are now two cases, depending on whether $z_i$ is
contained in the thin part or is an edge. Lemmas~\ref{l:closing1} and
\ref{l:closing2} will prove in turn that in each case the conclusions
of the proposition hold. After proving these lemmas, we will have
completed the proof of Proposition~\ref{general}.

\begin{lemma} \label{l:closing1} Assume that, in addition to the
  hypotheses of Proposition~\ref{general}, there is an edge $e$ so
  that $\phi(e)$ contains a segment of length $\sim C_nK$ in common
  with $z|G$. Then the conclusions of Proposition~\ref{general} hold.
\end{lemma}

\begin{proof}
Start extending $e$ to a legal edge path until an edge is
repeated. There are several possibilities.

{\it Type~0.} The first repetition is $e$ itself, i.e.\ we have
$e..e$. Then identifying the $e$'s gives a legal loop $u$ that crosses each
edge at most once, and $u$ works. 

{\it Type~1.} The first repetition is either $e^{-1}$ or another
edge with reversed orientation, i.e.\ $e..e^{-1}$ or
$e..a..a^{-1}$. Schematically we picture this as a monogon. Note
that there are two ways to traverse the monogon starting with $e$ and
ending with $e^{-1}$, both legal.

{\it Type~2.} The first repetition is an edge $a$ different from $e$
and with the same orientation, i.e.\ $e..a..a$. We picture
this as a spiral.

A monogon or spiral has its {\it tail} and its {\it loop}. In
$e..e^{-1}$ the tail is $e$ and the loop is represented by the edge
path between $e$ and $e^{-1}$; in $e..a..a^{-1}$, the tail is $e..a$
and the loop is represented by the edge path between $a$ and
$a^{-1}$; and in $e..a..a$ the tail is $e..a$ and the loop is
represented by the edge path between the $a$'s.

We can also extend $e$ in the opposite direction until an edge
repeats, and so we have three subcases.

{\it Subcase 1.} Type 1-1, i.e.\ we have Type 1 on both sides. Here we
have a morphism to $H$ from a graph $Y$ as in Figure~\ref{f:type.1.1}
whose induced illegal turns form a subset of those indicated.

\begin{figure}[h]
\begin{center}
\includegraphics[scale=0.4]{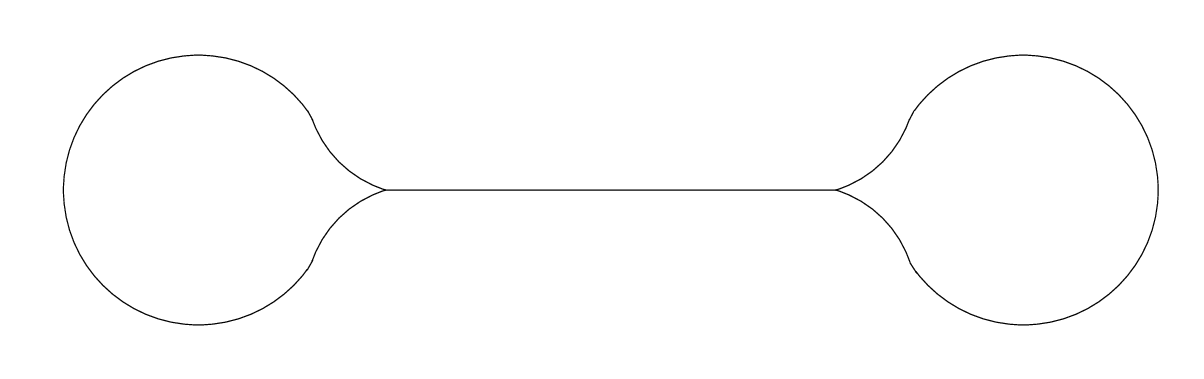}
\caption{Type 1-1.}
\label{f:type.1.1}
\end{center}
\end{figure}

If there is an edge $b$ (different from $e$) crossed by both monogons
then by switching from one copy of $b$ to the other we may form a
legal loop $u$ that crosses $e$ once and all other edges at most
twice.  Indeed, if at least one copy of $b$ is in a loop of a monogon,
then there is a legal segment in $Y$ of the form $b..e..b$ that
crosses $e$, $b$, and $b^{-1}$ only as indicated. If both copies of
$b$ are in tails, then there is either $b..e..b$ as above or
$b..e..b^{-1}..b$. In the latter case, choose the first $b$ as close
as possible to $e$ (to guarantee all edges in our legal loop are
crossed at most twice), and this loop works. If there is no such $b$,
the loop $u$ that traverses both monogons once is legal, crosses each
edge at most twice, and crosses some edge once, so $u$ works.

{\it Subcase~2.} Type 1-2, i.e.\ we have a monogon on one side and a
spiral on the other. 

\begin{figure}[h]
\begin{center}
\includegraphics[scale=0.4]{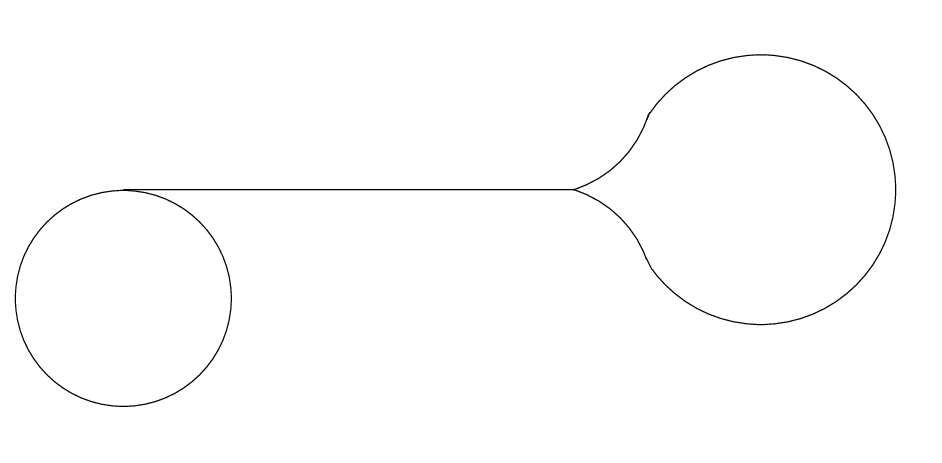}
\caption{Type 1-2.}
\label{f:type.1.2}
\end{center}
\end{figure}

If some edge $b$ different from $e$ is crossed by both the spiral and
the monogon, we can form a legal loop $u$ that crosses $e$ once as in
Subcase~1, and we are done.  Otherwise, let $u$ be the loop that
crosses both the spiral and the monogon, so it has (potentially) one
illegal turn. We claim that either the loop $v$ of the spiral or $u$
works. Indeed, $u$ crosses some edge once and all edges at most
twice. Write $u$ as $Ev'E^{-1}v^{-1}$, where $E$ is the edge-path
formed by the two tails and $v'$ is the loop of the monogon.
Schematically, $u$ can be drawn as in Figure~\ref{f:monogon}.

\begin{figure}[h]
\begin{center}
\begin{picture}(0,0)%
\includegraphics{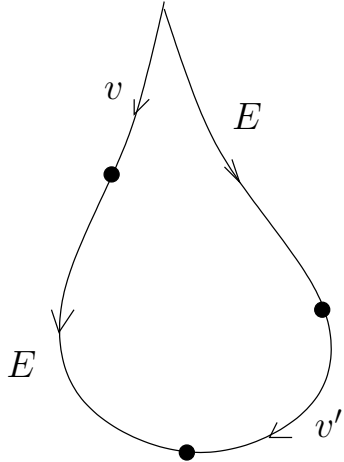}%
\end{picture}%
\setlength{\unitlength}{4144sp}%
\begingroup\makeatletter\ifx\SetFigFont\undefined%
\gdef\SetFigFont#1#2#3#4#5{%
  \reset@font\fontsize{#1}{#2pt}%
  \fontfamily{#3}\fontseries{#4}\fontshape{#5}%
  \selectfont}%
\fi\endgroup%
\begin{picture}(1971,2765)(3946,-3354)
\put(5311,-1366){\makebox(0,0)[lb]{\smash{{\SetFigFont{14}{16.8}{\rmdefault}{\mddefault}{\updefault}{\color[rgb]{0,0,0}$E$}%
}}}}
\put(3961,-2851){\makebox(0,0)[lb]{\smash{{\SetFigFont{14}{16.8}{\rmdefault}{\mddefault}{\updefault}{\color[rgb]{0,0,0}$E$}%
}}}}
\put(5806,-3211){\makebox(0,0)[lb]{\smash{{\SetFigFont{14}{16.8}{\rmdefault}{\mddefault}{\updefault}{\color[rgb]{0,0,0}$v'$}%
}}}}
\put(4546,-1186){\makebox(0,0)[lb]{\smash{{\SetFigFont{14}{16.8}{\rmdefault}{\mddefault}{\updefault}{\color[rgb]{0,0,0}$v$}%
}}}}
\end{picture}%
\caption{$u$ in Subcase 2.}
\label{f:monogon}
\end{center}
\end{figure}

Now consider the image $\phi(u)$. To see how much cancellation occurs,
let $Z$ be the maximal common initial segment of $[\phi(E)]$ and
$[\phi(vE)]$. By Lemma~\ref{l:simple}(\ref{i:2 segments}), $Z$ has the form $[\phi(v)]^NW'$
for some initial segment $W'$ of $[\phi(v)]$. If $Z$ shares a long
segment with $z|G$ then $v$ works. Otherwise, $u|G$
shares a long segment with $z|G$, and so works.

{\it Subcase~3.} Type 2-2, i.e.\ we have two spirals. 

\begin{figure}[h]
\begin{center}
\includegraphics[scale=0.4]{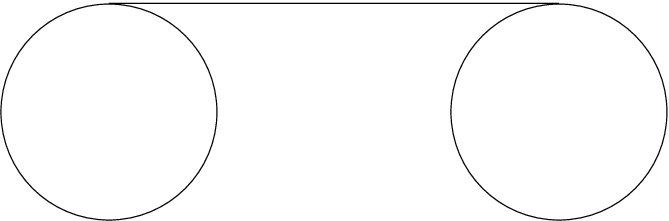}
\caption{Type 2-2.}
\label{f:type.2.2}
\end{center}
\end{figure}

The first case is that the two spirals do not contain any edges in
common except for $e$. Then the loop $u$ that crosses both spirals can
be written in the form $u=Ev'E^{-1}v^{-1}$, pictured as a bigon (see
Figure~\ref{f:bigon}). The argument is now similar to Subcase~2.
Consider the maximal common initial segment $V$ (resp.\ $V'$) of
$\phi(E)$ and $[\phi(vE)]$ (resp.\ $\phi(E^{-1})$ and
$[\phi(v'E^{-1})]$).  If either $V$ or $V'$ shares a long segment with
$z|G$ then $v$ or $v'$ works. Otherwise, $u$ works.

\begin{figure}[h]
\begin{center}
\begin{picture}(0,0)%
\includegraphics{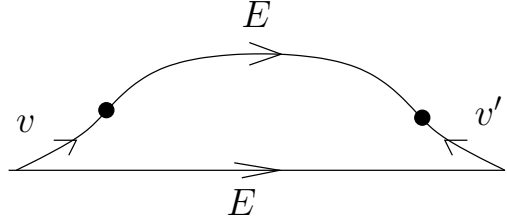}%
\end{picture}%
\setlength{\unitlength}{4144sp}%
\begingroup\makeatletter\ifx\SetFigFont\undefined%
\gdef\SetFigFont#1#2#3#4#5{%
  \reset@font\fontsize{#1}{#2pt}%
  \fontfamily{#3}\fontseries{#4}\fontshape{#5}%
  \selectfont}%
\fi\endgroup%
\begin{picture}(2994,1369)(3499,-3059)
\put(4816,-2986){\makebox(0,0)[lb]{\smash{{\SetFigFont{14}{16.8}{\rmdefault}{\mddefault}{\updefault}{\color[rgb]{0,0,0}$E$}%
}}}}
\put(4906,-1861){\makebox(0,0)[lb]{\smash{{\SetFigFont{14}{16.8}{\rmdefault}{\mddefault}{\updefault}{\color[rgb]{0,0,0}$E$}%
}}}}
\put(3556,-2491){\makebox(0,0)[lb]{\smash{{\SetFigFont{14}{16.8}{\rmdefault}{\mddefault}{\updefault}{\color[rgb]{0,0,0}$v$}%
}}}}
\put(6301,-2446){\makebox(0,0)[lb]{\smash{{\SetFigFont{14}{16.8}{\rmdefault}{\mddefault}{\updefault}{\color[rgb]{0,0,0}$v'$}%
}}}}
\end{picture}%
\caption{$u$ in the first case of Subcase 3.}
\label{f:bigon}
\end{center}
\end{figure}

The second case is that some edge $b$, other than $e$, occurs on both
spirals and we can construct a loop that crosses $e$ only once by
jumping from one $b$ to the other. If this loop is legal, we can take
it for $u$. Otherwise, it has one illegal turn, and there are two
possibilities.

Suppose first that $b$ in the tail of a spiral whose loop is $v$.
There is then a segment of the form $b..e..b^{-1}..v..b$ where the
only illegal turn is the initial point of $v$. Let $u$ be the loop
obtained by identifying the first and last $b$'s. Exactly as in
Subcase~2, either the canceling segments of $[\phi(u)]$ share a long
segment with $z|G$ (in which case $v$ works) or else $u|G$ shares a
long segment with $z|G$.

Secondly, suppose that $b$ appears in the loop $v$ of a spiral. Here
there is a segment of the form $b..e..w$ where $w$ is an initial
segment of $v$ ending with $b$ and the only illegal turn is the
initial vertex of $w$. Let $u$ be the loop obtained by identifying the
$b$'s in $b..e..w$. Let $u'$ be the loop obtained by identifying the
first and last $b$'s of $b..e..vw$. By Lemma~\ref{l:simple}(\ref{i:3
  segments}), either the canceling segments in one of $\phi(u)$ and
$\phi(u')$ shares a long segment with $z|G$ (in which case $v$ works)
or else one of $u|G$ and $u'|G$ shares a long segment with $z|G$.

Finally, by choosing the first $b$ of our segment as close as possible
to $e$, we guarantee that the loop we produce crosses each edge at
most twice and some edge once and so works.
\end{proof}

\begin{lemma}\label{l:closing2}
Suppose that, in addition to the hypotheses of Proposition~\ref{general}, there is no edge as in Lemma~\ref{l:closing1}, but there is a path $w$
in the thin part of $H$ of length $\leq 1$ such that $[\phi(w)]$
contains a segment of length $\sim C_nK$ in common with $z|G$. Then the conclusions of Proposition~\ref{general} hold.
\end{lemma}

\begin{proof}
  First, if necessary, concatenate $w$ with a combinatorially bounded
  path in the thin part so that its endpoints coincide. After this
  operation $[\phi(w)]$ still has a long piece in common with $z$
  since $\phi$-images of edges do not.  If taking $u=w$ does not work,
  then the path $[\phi(w)]$ has the form $VUV^{-1}$ with $V$ having a
  long piece in common with $z|G$. Choose a combinatorially short loop
  $a$ in the thin part, based at the endpoints of $w$. We aim to show
  that either $a$ or $aw$ works.

Write $[\phi(a)]=ABA^{-1}$ so that $B=[[B]]$. The loop $[[\phi(aw)]]$
is obtained by tightening the loop $ABA^{-1}UVU^{-1}$ which has at
most two illegal turns (the two occurrences of $\{A,U\}$). See
Figure~\ref{f:aw}.

\begin{figure}[h]
\begin{center}
\includegraphics[scale=1]{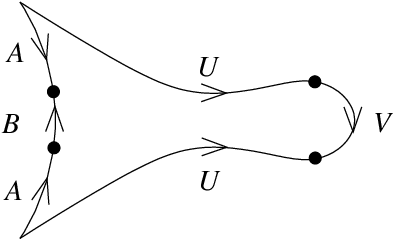}
\caption{$\phi(aw)$}
\label{f:aw}
\end{center}
\end{figure}

There are two cases. If the maximal common initial segment of $A$ and
$U$ is a proper segment of $A$ then our loop becomes immersed after
canceling the copies of this this common initial segment at the two
illegal turns. Here $aw$ works.

If $U=AU'$, then after canceling the common $A$'s, we are left with
$BU'VU'^{-1}$. Since the initial and terminal directions of $B$ are
distinct, one of the turns $\{B, U'\}$ and $\{B^{-1},U'\}$ is legal.
Using Lemma~\ref{l:simple}(\ref{i:2 segments}) again, either a power of $B$ has a long segment
in common with $U$ (in which case $a$ works) or not (in which case
$aw$ works).
\end{proof}

We have completed the proof of Proposition~\ref{general}.\qed

\begin{lemma}\label{criterion}
  Let $\bG_t$, $t\in[0,L],$ be a folding path in $\X$ parametrized by
  arclength, $H\in\X$, and $z$ a class in $\FF$.  For some $\tau\in
  [0,L]$, assume that $\ell(z|\bG_\tau)\geq\ell(z|H)$.
\begin{enumerate}[(i)]
\item If $z|\bG_\tau$ is legal then either
  $\lt(H)\le \tau$ or
  $d_\F(\Lt(H),\bG_\tau)$ is bounded. (Recall Definition~\ref{d:projection to Gt}.)
\item[(i$\,'$)] If $z|\bG_\tau$ has no immersed illegal segment of
  length $I$ then either $\lt(H)\le \tau$ or $d_\F(\Lt(H),\bG_\tau)$
  is bounded.
\item If $z$ is simple and $z|\bG_\tau$ is illegal then either\,
  $d_\F(\Rt(H),\bG_\tau)$ is bounded or $\rt(H)\ge\tau$.
\end{enumerate}
\end{lemma}

\begin{proof}
  We first prove (i). Let $C_n$ be the constant from
  Proposition~\ref{general}(closing up to a simple class). Since $z|\bG_\tau$
  is legal, by replacing $\tau$ by $\tau+e^{3C_n}$ we may assume that
  $\ell(z|\bG_\tau)\geq 3C_n\ell(z|H)$. Proposition~\ref{general} then
  provides a simple class $u$ with $\ell(u|H)<2$ and $d_\F(H,u)$
  bounded so that $u|\bG_\tau$ has a segment of length 3 in common
  with $z|\bG_\tau$; in particular $u|\bG_\tau$ contains a legal
  segment of length 3 and so $\lt_{\bG_t}(u)\le\tau$. We are done by
  Lemma~\ref{bounded crossing 2} and Corollary~\ref{coarse Lipschitz}.

  The proof of (i$'$) is similar. Because $z|\bG_\tau$ has no illegal
  segment of length $I$, its length grows at an exponential rate under
  folding. Also, the property of not having an illegal segment of
  length $I$ is stable under folding.  Therefore we may assume that
  $\ell(z|\bG_\tau)\ge C_nI\ell(z|H)$ and conclude that $u|\bG_\tau$
  produced by Proposition~\ref{general} shares a length $I$ segment with
  $z|\bG_\tau$ and so has a legal segment of length 3.

  The proof of (ii) is also analogous. Using Lemma~\ref{unfolding}, by
  moving left from $\bG_\tau$ a bounded amount in $\F$ we may assume
  $\ell(z|\bG_\tau)\geq C_nI\ell(z|H)$. Then Proposition~\ref{general}
  provides a simple class $u$ with $\ell(u|H)<2$ and $d_\F(H,u)$
  bounded so that $u|\bG_\tau$ has a segment of length $I$ in common
  with $z|\bG_\tau$. In particular, this segment is illegal.
\end{proof}

We summarize the conclusions in Lemma~\ref{criterion}(i) and (i$'$)
[resp.~(ii)] by saying that {\it the projection of $H$ to $\bG_t$ is
  coarsely to the left [resp.\ right] of $\bG_\tau$, measured in
  $\F$.} (Recall Corollary~\ref{left-right2}.)

\section{Hyperbolicity}\label{s:hyperbolicity}
The following proposition provides a blueprint for proving
hyperbolicity of $\F$. In the case of the curve complex the same
blueprint was used by Bowditch \cite{bo:hyp}.

\begin{prop}\label{hyperbolicity}
$\F$ is hyperbolic if and only if the following holds for projections of
  folding paths. There is $C>0$ so that:
\begin{enumerate}[(i)]
\item (Fellow Travel) Any two projections
  $\pi(G_t)$ and $\pi(H_t)$ of folding paths that start and end
  ``at distance 1'' (coarsely interpreted) are in each other's
  Hausdorff 
$C$-neighborhood.
\item (Symmetry) If $\pi(G_t)$ goes from $A$ to $B$ and $\pi(H_t)$
  from $B$ to $A$ then the two projections are in each other's
  Hausdorff 
$C$-neighborhood.
\item (Thin Triangles) Any triangle formed by projections of three
  folding paths is $C$-thin.
\end{enumerate}
\end{prop}

More precisely, in (i), if $G$ and $H$ are the initial points of the
two paths, the hypothesis means that there exist adjacent free factors
$A$ and $B$ such that $A\in\pi(G)$ and $B\in\pi(H)$, and similarly for
terminal points. 

\begin{proof}
  It is clear that (i)-(iii) are necessary for hyperbolicity, since
  projections of folding paths form a uniform collection of
  reparametrized quasi-geodesics and in hyperbolic spaces
  quasi-geodesics stay in bounded neighborhoods of geodesics. The
  converse is due to Bowditch \cite{bo:hyp} (a variant was used
  earlier by Masur-Minsky \cite{MM}). Here is a sketch. We will show
  that any loop $\alpha$ in $\F$ of length $L$ bounds a disk of area
  $\sim L\log L$. (Think of bounded length loops as bounding disks of
  area 1.) Subdivide $\alpha$ into $3\times 2^N$ segments of size
  $\sim 1$ and think of it as a polygon. Subdivide it into triangles
  in a standard way: a big triangle in the middle with vertices $2^N$
  segments apart, then iteratively bisect remaining polygonal paths.
  Represent each diagonal by the image of a folding path -- up to
  bounded Hausdorff distance the choices are irrelevant. Using Thin
  Triangles, each triangle of diameter $D$ can be filled with a disk
  of area $\sim D$. Adding the areas of all the triangles gives $\sim
  N\times 2^N\sim L\log L$.
\end{proof}

Proposition~\ref{contraction} generalizes Algom-Kfir's result
\cite{yael}. 

\begin{prop}\label{contraction}
  Let $H,H'\in\X$ with $d_\X(H,H')\leq M$ and let $\bG_t$ be a folding
  path such that $d_\X(H,\bG_t)\geq M$ for all $t$. Then the distance
  between the projections of $H$ and $H'$ to $\bG_t$ is uniformly
  bounded in $\F$.
\end{prop}

\begin{proof}
  Denote by $\bG_1$ the leftmost of $\lt_{\bG_t}(z)$ and by $\bG_2$
  the rightmost of $\rt_{\bG_t}(z)$ as $z|H$ varies over candidates in
  $H$. Then the interval $[\bG_1,\bG_2]$ is bounded in $\F$ by
  Lemma~\ref{bounded crossing 2}, Proposition~\ref{L2}, and Corollary
  \ref{left-right2}. Let $z_1|H$ be a candidate that realizes the
  distance to $\bG_1$, so $\ell(z_1|\bG_1)\geq e^M\ell(z_1|H)$ and
  $\ell(z_1|H')\leq {e^M}\ell(z_1|H)$. Combining these inequalities
  gives $\ell(z_1|\bG_1)\geq\ell(z_1|H')$, so
  Lemma~\ref{criterion}(ii) shows $\Rt_{\bG_t}(H')$ is coarsely to the
  right of $\bG_1$, measured in $\F$. In the same way one argues that
  $\Lt_{G_t}(H')$ is coarsely to the left of $\bG_2$, measured in
  $\F$. The claim follows.
\end{proof}

\begin{cor}\label{c:contraction}
A folding line that makes $(\kappa,C)$-definite progress in $\F$ is strongly
contracting in $\X$ (with the constants depending on $\kappa$ and $C$).\qed
\end{cor}

This simply means that, in the situation of
Proposition~\ref{contraction}, projections of $H$ and $H'$ to $\bG_t$
are at a uniformly bounded distance in $\X$ (depending on $\kappa$ and
$C$), measured from left to right.

Note that a folding line that makes definite progress in $\F$ is
necessarily in a thick part of $\X$ (i.e.\ the injectivity radius of
$\bG_t$ is bounded below by a positive constant). The converse does not
hold (but recall that it does hold in Teichm\"uller space, and Corollary~\ref{c:contraction} is the direct analog of Minsky's theorem \cite{minsky}
that Teichm\"uller geodesics in the thick part are strongly
contracting). 

\begin{remark}
  One can avoid the use of the technical
  Proposition~\ref{gafa}(surviving illegal turns) in the proof of
  Proposition~\ref{contraction}.
\end{remark}

\section{Fellow Travelers and Symmetry}\label{s:FT}
We will fix constants $C_1$, $C_2$, and $D$ from 
Proposition~\ref{L2} and Proposition~\ref{contraction}, so that:
\begin{itemize}
\item If $B<A$ are free factors at $\F$-distance $\geq C_1$ from a
  folding path $\bG_t$ then the $\X$-distance along $\bG_t$ between
  $\lt(A)$ and $\lt(B)$ is bounded by $D$,
\item The $\F$-diameter of the projection of a path of length $M$ to
  any folding path at distance $\geq M$ is always $\leq C_2$,
\item $C_1>C_2$.
\end{itemize}

\begin{prop}\label{FT0}
Fix $C$ sufficiently large.
Suppose $\bG_t$ and $H_\tau$ are two folding paths,
$d_\F(\bG_t,H_\tau)\geq C$ for all $t,\tau$, but the initial points and
the terminal points are at $\F$-distance $\leq 10C$. Then the
projections of the two paths to $\F$ are uniformly bounded in
diameter.

The same holds if the initial point of $\bG_t$ is $10C$-close to the
terminal point of $H_\tau$ and the terminal point of $\bG_t$ is
$10C$-close to the initial point of $H_\tau$.
\end{prop}

\begin{proof}
  Subdivide $H_\tau$ into a minimal number of segments whose
  $\F$-diameter is bounded by $C_1$. Say the subdivision points are
  $s_0<s_1<s_2<\cdots<s_m$. Let $\bG_{t_i}=\lt_{\bG_t}(H_{s_i})$. When
  $C>2C_1$ we have that the distance, measured from left to right,
  between $\bG_{t_i}$ and $\bG_{t_{i+1}}$ along $\bG_t$ is $\leq C_1D$
  (here $C>2C_1$ is needed so that interpolating free factors are also
  far from $\bG_t$). The $\F$-distance between $\bG_{t_0}$ and the
  initial point of $\bG_t$, and also between $\bG_{t_m}$ and the
  terminal point of $\bG_t$ is bounded (recall Corollary~\ref{coarse
    Lipschitz}). Further, $d_\X(\bG_{t_0},\bG_{t_m})\leq mC_1D$ as long
  as $\bG_{t_0}$ is to the left of $\bG_{t_m}$ (if it is to the right,
  the whole path $\bG_t$ is $\F$-bounded). So the projection of
  $[\bG_{t_0},\bG_{t_m}]$ to $H_\tau$ is bounded by $mC_2$, as long as
  the $\X$-distance between $\bG_t$ and $H_\tau$ is bounded below by
  $C_1D$ (if not then the $\F$-distance between $\bG_t$ and $H_\tau$ is
  bounded, contradiction when $C$ is sufficiently large, see Corollary
  \ref{proj X->F continuous}). So,
$$mC_2+2(\mbox{const})\geq (m-1)C_1$$
and since $C_1>C_2$ this implies that $m$ is bounded above. The claim
follows.

The proof is analogous in the ``anti-parallel'' case.
\end{proof}

Fellow Traveler and Symmetry properties are now immediate.

\begin{prop}\label{FT}
Let $\bG_t$ and $H_\tau$ be folding paths whose initial points are at
$\F$-distance $\leq R$ and the same holds for terminal points. Then
$\pi(\bG_t)\subset\F$ and $\pi(H_\tau)\subset\F$ are in each other's
bounded Hausdorff neighborhoods, the bound depending only on $R$.

The same holds when the initial point of $\pi(\bG_t)$ is $R$-close to the
terminal point of $\pi(H_\tau)$ and the terminal point of $\bG_t$ is $R$-close
to the initial point of $H_\tau$. 
\end{prop}

\begin{proof}
  Let $C>R$ be a sufficiently large constant as in
  Proposition~\ref{FT0}. If $\pi(\bG_t)$ is not contained in the
  Hausdorff $C$-neighborhood of $\pi(H_\tau)$ there is a subpath
  $[\bG_{t_1},\bG_{t_2}]$ such that no point of it is $C$-close to
  $\pi(H_\tau)$, but the endpoints $\bG_{t_1},\bG_{t_2}$ are within
  $10C$. Then there is a subpath $[H_{\tau_1},H_{\tau_2}]$ of $H_\tau$
  whose endpoints are within $10C$ of the endpoints of
  $[\bG_{t_1},\bG_{t_2}]$ (but notice that we don't know in advance if
  the orientations are parallel or anti-parallel). Now by
  Proposition~\ref{FT0} the set $\pi([\bG_{t_1},\bG_{t_2}])$ is in a
  bounded Hausdorff neighborhood of $\pi(H_\tau)$. By the same
  argument $\pi(H_\tau)$ is contained in a bounded Hausdorff
  neighborhood of $\pi(\bG_t)$.
\end{proof}

\begin{remark}\label{FT2}
  Note that, in the situation of Proposition~\ref{FT}, any $\bG_{t_0}$
  is bounded $\F$-distance from its projection to $H_\tau$.  Indeed,
  if $H_{\tau_0}$ is bounded $\F$-distance from $\bG_{t_0}$ then from
  Corollary~\ref{coarse Lipschitz} it follows that the projection of
  $\bG_{t_0}$ is bounded $\F$-distance from $H_{\tau_0}$. 
\end{remark}

\begin{prop}\label{FT for projections}
  Let $\bG_t$ and $H_\tau$ be folding paths whose initial points are
  at $\F$-distance $\leq R$ and the same holds for terminal points.
  There is a uniform bound, depending only on $R$, to
  $d_\F(\Lt_{\bG_t}(A),\Lt_{H_\tau}(A))$ for any free factor $A$.

The same holds in the anti-parallel case.
\end{prop}

\begin{proof}
  For notational simplicity, we may assume that $A=\langle a\rangle$
  is cyclic.  First suppose $\bG_t$ and $H_\tau$ are parallel. Modulo
  interchanging the two paths, we can assume that the projection of
  $\Lt_{\bG_t}(a)$ to $H_\tau$ is ``ahead'' of $\Lt_{H_\tau}(a)$ 
  (i.e., when $\Lt_{\bG_t}(a)$ is projected to $H_\tau$, it is
  coarsely right of $\Lt_{H_\tau}(a)$ measured in $\F$) and the
  projection of $\Lt_{H_\tau}(a)$ to $\bG_t$ is ``behind'' (defined
  analogously) $\Lt_{\bG_t}(a)$. By Lemma~\ref{criterion} if
  $\ell(a|\Lt_{\bG_t}(a))\leq \ell(a|\Rt_{H_\tau}(a))$ then the
  projection of $\Lt_{\bG_t}(a)$ to $H_\tau$ is behind
  $\Rt_{H_\tau}(a)$, and the claim is proved.  If
  $\ell(a|\Lt_{\bG_t}(a))\geq \ell(a|\Rt_{H_\tau}(a))$ then the
  projection of $\Rt_{H_\tau}(a)$ to $\bG_t$ is ahead of
  $\Lt_{\bG_t}(a)$, and we are done again.

  Now suppose $\bG_t$ and $H_\tau$ are anti-parallel. There are two
  subcases to consider, see Figure~\ref{f:ahead}. In the first case
  $\Lt_{\bG_t}(A)$ is ahead the projection to $\bG_t$ of
  $\Lt_{H_\tau}(A)$. This is a symmetric condition with respect to
  interchanging $\bG_t$ and $H_\tau$ (by Remark \ref{FT2}). Say
  $\ell(a|\Lt_{\bG_t}(a))\leq \ell(a|\Lt_{H_\tau}(a))$. Then the
  projection of $\Lt_{\bG_t}(a)$ to $H_\tau$ is ahead of
  $\Lt_{H_\tau}(a)$ by Lemma \ref{criterion}, and the claim follows.

\begin{figure}
\begin{center}
\begin{picture}(0,0)%
\includegraphics{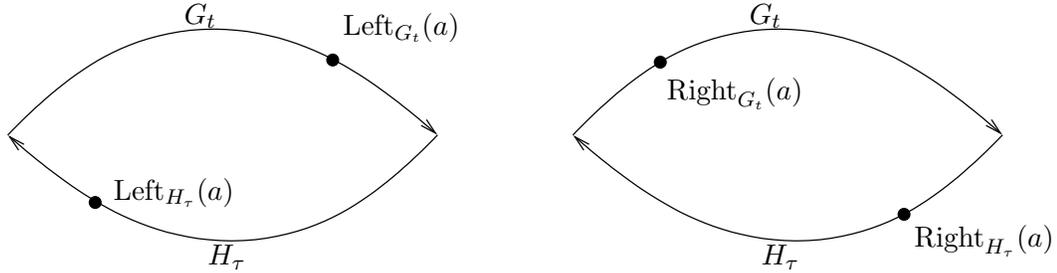}%
\end{picture}%
\setlength{\unitlength}{4144sp}%
\begingroup\makeatletter\ifx\SetFigFont\undefined%
\gdef\SetFigFont#1#2#3#4#5{%
  \reset@font\fontsize{#1}{#2pt}%
  \fontfamily{#3}\fontseries{#4}\fontshape{#5}%
  \selectfont}%
\fi\endgroup%
\begin{picture}(5975,1688)(2924,-4239)
\put(8359,-4060){\makebox(0,0)[lb]{\smash{{\SetFigFont{11}{13.2}{\rmdefault}{\mddefault}{\updefault}{\color[rgb]{0,0,0}$\Rt_{H_\tau}(a)$}%
}}}}
\put(4133,-4166){\makebox(0,0)[lb]{\smash{{\SetFigFont{11}{13.2}{\rmdefault}{\mddefault}{\updefault}{\color[rgb]{0,0,0}$H_\tau$}%
}}}}
\put(3992,-2722){\makebox(0,0)[lb]{\smash{{\SetFigFont{11}{13.2}{\rmdefault}{\mddefault}{\updefault}{\color[rgb]{0,0,0}$\bG_t$}%
}}}}
\put(7373,-2722){\makebox(0,0)[lb]{\smash{{\SetFigFont{11}{13.2}{\rmdefault}{\mddefault}{\updefault}{\color[rgb]{0,0,0}$\bG_t$}%
}}}}
\put(7444,-4166){\makebox(0,0)[lb]{\smash{{\SetFigFont{11}{13.2}{\rmdefault}{\mddefault}{\updefault}{\color[rgb]{0,0,0}$H_\tau$}%
}}}}
\put(4943,-2792){\makebox(0,0)[lb]{\smash{{\SetFigFont{11}{13.2}{\rmdefault}{\mddefault}{\updefault}{\color[rgb]{0,0,0}$\Lt_{\bG_t}(a)$}%
}}}}
\put(3570,-3778){\makebox(0,0)[lb]{\smash{{\SetFigFont{11}{13.2}{\rmdefault}{\mddefault}{\updefault}{\color[rgb]{0,0,0}$\Lt_{H_\tau}(a)$}%
}}}}
\put(6880,-3180){\makebox(0,0)[lb]{\smash{{\SetFigFont{11}{13.2}{\rmdefault}{\mddefault}{\updefault}{\color[rgb]{0,0,0}$\Rt_{\bG_t}(a)$}%
}}}}
\end{picture}%
\caption{The ``ahead'' and the ``behind'' cases.}
\label{f:ahead}
\end{center}
\end{figure}

The ``behind'' case is similar, but we consider right projections. Say
$\ell(a|\Rt_{\bG_t}(a))\leq \ell(a|\Rt_{H_\tau}(a))$. Then the
projection of $\Rt_{\bG_t}(a)$ to $H_\tau$ is behind
$\Rt_{H_\tau}(a)$, and the claim again follows.
\end{proof}

\section{Thin Triangles}\label{s:thin}
\begin{prop}\label{thin}
Triangles in $\F$ made of images of folding paths are uniformly
thin. More precisely, if $A,B,C$ are three free factors coarsely
joined by images of folding paths $AB$, $AC$, $BC$ and $\hat C$ is the
projection of $C$ to $AB$, then $A\hat C$ is in a bounded Hausdorff
neighborhood of $AC$ and $B\hat C$ is in a bounded Hausdorff
neighborhood of $BC$.
\end{prop}

We will consider points $U,V,W\in\X$ and folding paths $H_t$ from $V$
to $W$ and $\bG_t$ from $U$ to $W$. Denote by $P$ the rightmost of
$\Rt_{\bG_t}(z)$ as $z$ ranges over candidates in $V$. By
Lemma~\ref{bounded crossing 2}, $P$ has bounded $\F$-distance from the
projection of $V$ to $\bG_t$. See Figure~\ref{f:thin}.

\begin{figure}[h]
\begin{center}
\begin{picture}(0,0)%
\includegraphics{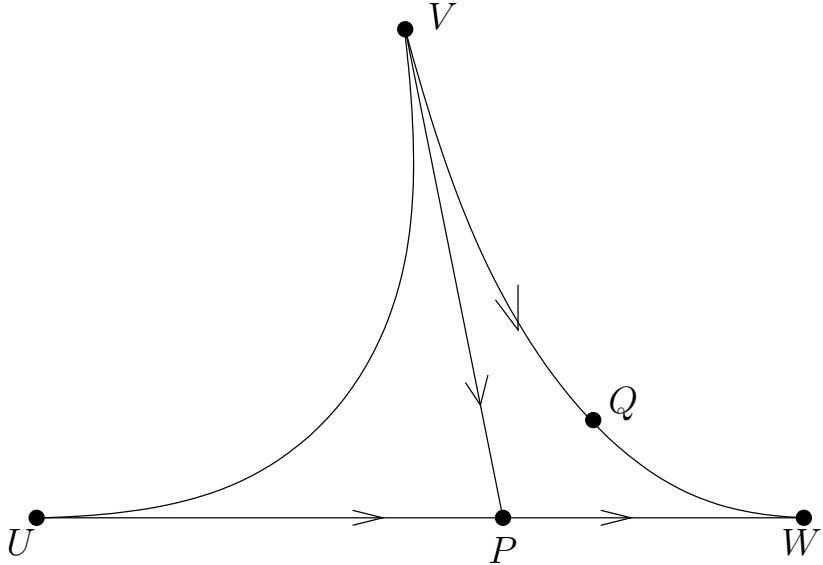}%
\end{picture}%
\setlength{\unitlength}{4144sp}%
\begingroup\makeatletter\ifx\SetFigFont\undefined%
\gdef\SetFigFont#1#2#3#4#5{%
  \reset@font\fontsize{#1}{#2pt}%
  \fontfamily{#3}\fontseries{#4}\fontshape{#5}%
  \selectfont}%
\fi\endgroup%
\begin{picture}(4842,3439)(3046,-3644)
\put(3061,-3526){\makebox(0,0)[lb]{\smash{{\SetFigFont{14}{16.8}{\rmdefault}{\mddefault}{\updefault}{\color[rgb]{0,0,0}$U$}%
}}}}
\put(7696,-3526){\makebox(0,0)[lb]{\smash{{\SetFigFont{14}{16.8}{\rmdefault}{\mddefault}{\updefault}{\color[rgb]{0,0,0}$W$}%
}}}}
\put(5581,-376){\makebox(0,0)[lb]{\smash{{\SetFigFont{14}{16.8}{\rmdefault}{\mddefault}{\updefault}{\color[rgb]{0,0,0}$V$}%
}}}}
\put(5941,-3571){\makebox(0,0)[lb]{\smash{{\SetFigFont{14}{16.8}{\rmdefault}{\mddefault}{\updefault}{\color[rgb]{0,0,0}$P$}%
}}}}
\put(6661,-2671){\makebox(0,0)[lb]{\smash{{\SetFigFont{14}{16.8}{\rmdefault}{\mddefault}{\updefault}{\color[rgb]{0,0,0}$Q$}%
}}}}
\end{picture}%
\caption{A thin triangle.}
\label{f:thin}
\end{center}
\end{figure}

We will prove that $\pi(VP)\cup\pi(PW)$ and $\pi(VW)$ are contained in
uniform Hausdorff neighborhoods of each other (by $VP$ we mean a
folding path from $V$ to $P$ etc). The basic idea is that $VP\cup PW$
behaves like a folding path and the claim is an instance of the Fellow
Traveler Property.

{\bf Claim~1.} $d_\X(V,W)\geq d_\X(V,P)+d_\X(P,W)-C$ for a universal
constant $C$.

To prove the claim, let $v|V$ be a candidate for $d_\X(V,P)$. By
definition of $P$, $v|P$ has only bounded length illegal subsegments.
It follows that after removing the 1-neighborhood of each illegal
turn, a definite percentage of the length of $v|P$ remains, and hence
$$\ell(v|W)\geq e^{d_\X(P,W)}
\epsilon\ell(v|P)=e^{d_\X(V,P)+d_\X(P,W)} \epsilon\ell(v|V)$$ for a
fixed $\epsilon>0$. Thus
$$d_\X(V,W)\geq\log\frac{\ell(v|W)}{\ell(v|V)}\geq d_\X(V,P)+d_\X(P,W)
+\log \epsilon$$

By $Q$ denote the point on $VW$ such that $d_\X(V,Q)=d_\X(V,P)$. From
Claim~1 we see that $d_\X(Q,W)\leq d_\X(P,W)\leq d_\X(Q,W)+C$.

{\bf~Claim 2.} $d_\F(P,Q)$ is bounded.

This time let $v|V$ be a candidate for $d_\X(V,W)$. Thus 
$$\ell(v|Q)=e^{d_\X(V,Q)}\ell(v|V)=e^{d_\X(V,P)}\ell(v|V)\geq
\ell(v|P)$$ Let $Q'$ be the point along $QW$ with
$d_\X(Q,Q')=\log(3(2n-1)C_n)$ (if such a point does not exist, both
$P$ and $Q$ are uniformly close to $W$ and we are done). Then
$\ell(v|Q')=3(2n-1)C_n\ell(v|Q)\geq 3(2n-1)C_n\ell(v|P)$, so
Proposition~\ref{general}(closing up to a simple class) implies that
there is a simple class $p$ such that $p|P$ is of length $<2$ and
$p|Q'$ has a legal segment of length $3(2n-1)$. Now $p|Q'$ cannot
contain many disjoint legal segments of length 3, for otherwise $p$
would grow to a much longer length along $Q'W$ than along $PW$.  Now
Claim~2 follows from Lemma~\ref{michael3}.

\begin{proof}[Proof of Proposition \ref{thin}]
  By Proposition~\ref{FT} we are free to replace a folding path by
  another whose endpoints project close to the endpoints of the
  original, and we are allowed to reverse orientations. By
  Proposition~\ref{FT for projections} these replacements affect the
  projections by a bounded amount, so that the projection $\hat C$ of
  $C$ to $AB$ is coarsely well-defined, independently of the choice of
  a folding path whose projection coarsely connects $A$ and $B$ with
  either orientation. In particular, we may assume that $U,V,W$
  project near $A,C,B$ respectively and we may consider folding paths
  $UW$ and $VW$ that end at the same graph $W$. The above discussion
  then shows that $B\hat C$ is contained in a bounded Hausdorff
  neighborhood of $BC$. Making analogous choices of the folding paths,
  the claim about $A\hat C$ follows similarly.
\end{proof}

The following fact shows that for the purposes of this paper
the collection of folding paths can be replaced by the larger
collection of geodesic paths.

\begin{prop}\label{p:any geodesic}
  Let $V,P,W\in\X$ so that $d_\X(V,P)+d_\X(P,W)=d_\X(V,W)$ and let
  $V'\in\X$ be such that $d_\X(V,V')+d_\X(V',W)=d_\X(V,W)$,
  $d_\F(V,V')$ is bounded, and there is a folding path $\bG_t$ from
  $V'$ to $W$ (see Proposition~\ref{rescaling}). Then the
  $\F$-distance between $P$ and some $\bG_t$ is uniformly bounded.
  Moreover, if $H_\tau$ is a geodesic in $\X$ from $V$ to $P$, then
  the induced correspondence $\tau\mapsto t$ can be taken to be
  monotonic with respect to $\tau$. Consequently, the set of
  projections to $\F$ of geodesics in $\X$ is a uniform collection of
  reparametrized quasi-geodesics in $\F$.
\end{prop}

The proof is a variant of the discussion above.

\begin{proof}
  Let $P'$ be the point on $\bG_t$ with $d_\X(V,P')=d_\X(V,P)$ (if
  such a point does not exist we assign $V'$ to $P$). We need to argue
  that $d_\F(P,P')$ is bounded. Let $v|V$ be a candidate realizing
  $d_\F(V,W)$. Thus $v|\bG_t$ is legal for all $t$ and
  $\ell(v|P)=\ell(v|P')$. Let $Q'$ be a point on $\bG_t$ with
  $d_\X(P',Q')=\log (3(2n-1)C_n)$ (if such a point does not exist,
  both $P$ and $P'$ are close to $W$ and we are done). Then
  $\ell(v|Q')=3(2n-1)C_n\ell(v|P)$, so Proposition~\ref{general}(closing up
  to a simple class) implies that there is a simple class $p$ such
  that $\ell(p|P)<2$ and with $p|Q'$ containing a legal segment of
  length $3(2n-1)$. Now $p|Q'$ cannot contain many disjoint legal
  segments of length 3, for otherwise $p$ would grow to a much longer
  length along $Q'W$ than along $PW$. Lemma~\ref{michael3} implies
  that $d_\F(P,Q')$, and therefore $d_\F(P,P')$, is bounded.
\end{proof}

\begin{thm}
$\F$ is $\delta$-hyperbolic. Images of geodesic paths in $\X$ are in
  uniform Hausdorff neighborhoods of geodesics with the same
  endpoints.

  Furthermore, an element of $Out(\FF)$ has positive translation
  length in $\F$ if and only if it is fully irreducible. The action of
  $Out(\FF)$ on $\F$ satisfies Weak Proper Discontinuity (see \cite{BF})).
\end{thm}

\begin{proof}
  Since we have checked Conditions~(i), (ii), and (iii) in
  Proposition~\ref{hyperbolicity}, $\F$ is hyperbolic. The
  second statement is a consequence of Proposition~\ref{p:any
    geodesic} and last the two statements follow from:
\begin{itemize}
\item There are coarsely well-defined, Lipschitz maps from $\F$ to
  the hyperbolic complexes $\mathcal X$ constructed in \cite[Sections~4.4.1 and 4.4.2]{BF2}.
\item Given a fully irreducible element $f$ of $Out(F_n)$, there is an
  $\mathcal X$ on which $f$ has positive translation length
  (\cite[Main Theorem]{BF2}).
\item Further, for every
  $x\in\mathcal X$ and every $C > 0$ there is $N>0$ such that $\{g\in Out(\FF)\mid d_{\mathcal
    X}(x, xg)\le C, d_\mathcal X(xf^N, xf^Ng)\le C\}$ is finite (\cite[Section~4.5]{BF2}).
\end{itemize}
\end{proof}

\bibliography{./ref}
\end{document}